\documentclass{article}

\usepackage{answers}
\usepackage{xcolor}
\definecolor{sepia}{rgb}{0,0.3,0.2}

\usepackage{enumitem}
\usepackage[T1]{fontenc}
\usepackage[english]{babel}  
\usepackage{dsfont}
\usepackage{amsmath,amstext,am ssymb}
\usepackage{amsfonts}
\usepackage{amssymb}
\usepackage{graphicx}
\usepackage{amsthm}
\usepackage{stmaryrd}
\usepackage{listings}
\usepackage{makeidx}
\usepackage{xcolor}
\usepackage{authblk}
\usepackage{soul}

\newtheorem{prop}{Proposition}
\newtheorem{coro}{Corollary}
\newtheorem{cond}{Condition}
\newtheorem{rmk}{Remark}

\newtheorem{lm}{Lemma}

\definecolor{colo}{rgb}{0,0,0.5}

\DeclareMathOperator{\ind}{\perp \!\!\! \perp}
\DeclareMathOperator{\argmin}{arg\,min}
\DeclareMathOperator{\argmax}{arg\,max}

\DeclareMathOperator{\vecc}{vec}
\DeclareMathOperator{\Tr}{Tr}
\DeclareMathOperator{\V}{Var}
\DeclareMathOperator{\pen}{pen}
\DeclareMathOperator{\cov}{cov}
\DeclareMathOperator{\mat}{mat}

\newcommand{\R}{\mathbb{R}}

\newcommand{\N}{\mathbb{N}}

\newcommand{\E}{\mathbb{E}}
\newcommand{\PP}{\mathbb{P}}

\newcommand{\be}{\begin{eqnarray}}
\newcommand{\ee}{\end{eqnarray}}
\newcommand{\beq}{\begin{eqnarray*}}
\newcommand{\eeq}{\end{eqnarray*}}

\begin{document}
\author[1]{Baptiste Broto}

\author[2]{Fran\c{c}ois Bachoc}

\author[3]{Laura Clouvel}

\author[4]{Jean-Marc Martinez}

\affil[1]{CEA, LIST, Universit\'e Paris-Saclay, F-91120, Palaiseau, France}
\affil[2]{Institut de Math\'ematiques de Toulouse, Universit\'e Paul Sabatier, F-31062 Toulouse, France}
\affil[3]{CEA, SERMA, Universit\'e Paris-Saclay, F-91191 Gif-sur-Yvette, France}
\affil[4]{CEA, DEN-STMF, Universit\'e Paris-Saclay, F-91191 Gif-sur-Yvette, France}

\title{Block-diagonal covariance estimation and application to the Shapley effects in sensitivity analysis}
\date\today

\maketitle

\begin{abstract}
    In this paper, we estimate a block-diagonal covariance matrix from Gaussian variables in high dimension. We prove that, under some mild assumptions, we find the block-diagonal structure of the matrix with probability that goes to one. We deduce estimators of the covariance matrix that are as accurate as if the block-diagonal structure where known, with numerical applications. We also prove the asymptotic efficiency of one of these estimators in fixed dimension. Then, we apply these estimators for the computation of sensitivity indices, namely the Shapley effects, in the Gaussian linear framework. We derive estimators of the Shapley effects in high dimension with a relative error that converges to 0 at the parametric rate, up to a logarithm factor.  Finally, we apply the Shapley effects estimators on nuclear data.
\end{abstract}


\section{Introduction}

Sensitivity analysis, and particularly sensitivity indices, have became important tools in applied sciences. The aim of sensitivity indices is to quantify the impact of the input variables $X_1,...,X_p$ on the output $Y$ of a model. This information improves the interpretability of the model. In global sensitivity analysis, the input variables are assumed to be random variables. In this framework, the Sobol indices \cite{sobol_sensitivity_1993} were the first suggested indices to be applicable to general classes of models. Nevertheless, one of the most important limitations of these indices is the assumption of independence between the input variables. Hence, many variants of the Sobol indices have been suggested for dependent input variables \cite{mara_variance-based_2012,chastaing_indices_2013,mara_non-parametric_2015,chastaing2012generalized}.

Recently, Owen defined new sensitivity indices in \cite{owen_sobol_2014} called "Shapley effects". These sensitivity indices have many advantages over the Sobol indices for dependent inputs \cite{iooss_shapley_2017}. For general models, \cite{song_shapley_2016} suggested an estimator of the Shapley effects. However, this estimation requires to be able to generate samples with the conditional distributions of the input variables. A consistent estimator has been suggested in \cite{broto_variance_2018}, requiring only a sample of the inputs-output. This estimator uses nearest-neighbours methods to mimic the generation of samples with these conditional distributions.\bigskip

In this paper, we focus on Gaussian linear models in large dimension. Gaussian linear models are widely used as numerical models of physical phenomena (see for example \cite{kawano_evaluation_2006,hammer_approximate_2011,rosti_linear_2004}). Indeed, uncertainties are often modelled as Gaussian variables and an unknown function $Y=f(X_1,...,X_p)$ is commonly approximated by its linear approximation around $\E(X)$. Furthermore, high-dimensional Gaussian linear models are widely studied in statistics \cite{buhlmann2011statistics,giraud2014introduction}.
In this particular case of Gaussian linear models, the theoretical values of the Shapley effects can be computed explicitly \cite{owen_shapley_2017,iooss_shapley_2017,broto_sensitivity_2019}.
These values depend on the covariance matrix $\Sigma$ of the inputs and on the coefficients $\beta$ of the linear model.

In this paper, we assume that we observe an i.i.d. sample of the input Gaussian variables in high dimension and that the true covariance matrix $\Sigma$ and the vector $\beta$ are unknown.  In this setting, the Shapley effects need to be estimated, replacing the true vector $\beta$ by its estimation and the theoretical covariance matrix $\Sigma$ by an estimated covariance matrix.

There exists a fair amount of work on high-dimensional covariance matrix estimation.
Many researchers took an interest in the empirical covariance matrix in high dimension \cite{Mar_enko_1967,wachter1978,smallest_1985_silverstein,bai_spectral_2010}.  For particular covariance matrices, different estimators than the empirical covariance can be preferred. For some well-conditioned families of covariance matrices, \cite{bickel_regularized_2008} suggests a banded version of the empirical covariance matrix, and several works address the problem of estimating a sparse covariance matrix \cite{huang_covariance_2006,lam_sparsistency_2009, el_karoui_operator_2008}. \bigskip

However, in general, given a high-dimensional covariance matrix, the computation cost of the corresponding Shapley effects grows exponentially with the dimension.
The only setting where a procedure to compute the Shapley effect with a non-exponential cost is the setting of block-diagonal matrices
\cite{broto_sensitivity_2019}. 
Hence, in high dimension, block-diagonal covariance matrices are a very favorable setting for the estimation of the Shapley effects. Thus, we address the estimation of high-dimensional block-diagonal covariance matrices in this paper. In contrasts, we remark that the above methods
are not relevant for the estimation of the Shapley effects, since they do not provide block-diagonal matrices. \bigskip

In our framework, we assume that the true covariance matrix is block-diagonal and we want to estimate this matrix with a similar structure to compute the deduced Shapley effects. Some works address the block-diagonal estimation of covariance matrices. \cite{perrot-dockes_estimation_2018} gives a numerical procedure to estimate such covariance matrices and \cite{hyodo_testing_2015} suggests a test to verify the independence of the blocks.
A block-diagonal estimator of the covariance matrix is proposed in \cite{devijver_diagonal_2018}. The authors of \cite{devijver_diagonal_2018} choose a more general framework, without assuming that the true covariance matrix is block-diagonal. 
They obtain the estimated block-diagonal structure by thresholding the empirical correlation matrix. They also give theoretical guaranties by bounding the average of the squared Hellinger distance between the estimated probability density function and the true one. This bound depends on the dimension $p$ and the sample size $n$. When $p\slash n$ converges to some constant $y\in ]0,1[$, this bound is larger than $1$ and is no longer relevant as the Hellinger distance is always smaller than $1$. \bigskip

Here, we focus on the high dimension setting, when $p\slash n$ converges to some constant $y\in ]0,1[$, and when the true covariance matrix is assumed to be block-diagonal. We give different estimators of the block-diagonal structure and we show that their complexity is small. Then, we provide new asymptotic results for these estimators. Under mild conditions, we show that the estimators of the block structure are equal to the true block structure, with probability converging to one. Furthermore, the square Frobenius distance between the estimated covariance matrices and the true one, normalized by $p$, converge to zero at rate $1/n$. 
Thus, our work complements the one of \cite{devijver_diagonal_2018}. We also study the fixed-dimensional setting, where we show that one of our suggested estimators is asymptotically efficient.
\bigskip

From the estimated block-diagonal covariance matrices, we deduce estimators of the Shapley effects  in the high-dimensional linear Gaussian framework, with reduced computational cost. We recall that in high dimension, the computation of the Shapley effects requires that the corresponding covariance matrix be block-diagonal. We show that the relative estimation error of these estimators goes to zero at the parametric rate $1/n^{1/2}$, up to a logarithm factor, even if the linear model is estimated from noisy observations. \bigskip

Our convergence results are confirmed by numerical experiments. We also apply our algorithm to semi-generated data from nuclear applications.
\bigskip

The rest of the paper is organized as follows. In Section \ref{section_cov}, we focus on the block-diagonal estimation of the block-diagonal covariance matrix. In Section \ref{section_shap}, we apply this block-diagonal estimation of the covariance matrix to deduce Shapley effects estimators. Section \ref{section_appli} is devoted to the numerical application {\color{colo} on nuclear data}, and the conclusion is given in Section \ref{section_conclu}. All the proofs are postponed to the appendix.

\section{Estimation of block-diagonal covariance matrices}\label{section_cov}

\subsection{Problem and notation}\label{sec_problem}

We assume that we observe  $(X^{(l)})_{l\in [1:n]}$,  an i.i.d. sample with distribution $\mathcal{N}(\mu,\Sigma)$, where $\mu \in \R^p$ and $\Sigma$ are not known. Here, $[1:n]$ denotes the set of the integers from 1 to $n$. We assume that $\Sigma=(\sigma_{ij})_{i,j\in [1:p]} \in S_p^{++}(\R)$ (the set of the symmetric positive definite matrices) and has a block-diagonal decomposition. To be more precise on this block-diagonal decomposition, we need to introduce some notation.

Let us write $\mathcal{P}_p$ the set of all the partitions of $[1:p]$.
We endow the set $\mathcal{P}_p$ with the following partial order. If $B, B' \in \mathcal{P}_p$, we say that $B$ is finer than $B'$, and we write $B\leq B'$, if for all $A \in B'$, there exists $A_1,...,A_i \in B$ such that $A=\bigsqcup_{j=1}^i A_j$. We also compare the elements of a partition $B \in \mathcal{P}_p$ with their smallest element; that enables us to talk about "the $k$-th element" of $B$.
If $B\in \mathcal{P}_p$ and $a_1,...,a_i \in [1:p]$, we write $(a_1,...,a_i)\in B$ if there exists $A \in B$ such that $\{a_1,...,a_i\} \subset A$ (in other words, if $a_1,...,a_i$ are in the same group of $B$).
If $\Gamma \in S_p^{++}(\R)$ with $\Gamma=(\gamma_{ij})_{i,j\in [1:p]}$ and if $B\in \mathcal{P}_p$, we define $\Gamma_B$ by
$$
(\Gamma_{B})_{i,j}=\left\{ \begin{array}{ll}
\gamma_{ij} & \text{if }(i,j)\in B\\
0 & \text{otherwise.}
\end{array}\right.
$$
Let us define
$$
S_p^{++}(\R,B):=\{\Gamma \in S_p^{++}(\R)\;|\; \Gamma=\Gamma_B,\;\text{and } \forall B'< B,\; \Gamma \neq \Gamma_{B'} \},
$$
where we define $B'<B$ if $B'\leq B$ and if $B'\neq B$.
Thus $S_p^{++}(\R)=\bigsqcup_{B \in \mathcal{P}_p} S_p^{++}(\R,B)$ and for all $\Gamma \in S_p^{++}(\R)$, we can define an unique $B(\Gamma) \in \mathcal{P}_p$ such that $\Gamma \in S_p^{++}(\R,B(\Gamma))$. 
Here, we assume that $\Sigma \in S_p^{++}(\R,B^*)$, i.e. $B^*$ is the finest decomposition of $\Sigma$, i.e. $B(\Sigma)=B^*$. We say that $\Sigma$ has a block-diagonal decomposition $B^*$.\\

We also write
$$
\overline{X}_n:=\frac{1}{n}\sum_{l=1}^n X^{(l)},
$$
and
$$
S_n:=\frac{1}{n}\sum_{l=1}^n(X^{(l)}-\overline{X}_n) (X^{(l)}-\overline{X}_n)^T,
$$
which are the empirical estimators of $\mu$ and $\Sigma$. To simplify notation, we write $\overline{X}$ for $\overline{X}_n$ and $S$ for $S_n$ (the dependency on $n$ is implicit).
We know that, for all $\Gamma\in S_p^{++}(\R)$, $\overline{X}$ maximizes the likelihood $L_{\Gamma, m}(X^{(1)},...,X^{(n)})$ over the mean parameter $m$, where
$$
L_{\Gamma, m}(X^{(1)},...,X^{(n)}):=\frac{1}{(2\pi)^{\frac{n}{2}}|\Gamma|^{\frac{1}{2}}}\exp\left(- \frac{1}{2}\sum_{l=1}^n (X^{(l)}-m)^T \Gamma^{-1}(X^{(l)}-m)  \right),
$$
and $|\Gamma|$ is the determinant of $\Gamma$.
Thus, for all $\Gamma \in S_p^{++}(\R)$, we define
$$
l_{\Gamma}:= -\frac{2}{ p} \log \left(L_{\Gamma, \overline{X}}(X^{(1)},...,X^{(n)}) \right) - \frac{n}{p}\log(2\pi)= \frac{1}{p}\left(\log|\Gamma|+\Tr(\Gamma^{-1} S) \right).
$$
As we assume that the true covariance matrix is block-diagonal, we consider a block-diagonal promoting penalization of the form
$$
\pen(\Gamma):=\pen(B(\Gamma)):=\sum_{i=1}^K p_k^2,
$$
if $B(\Gamma)=\{B_1,...,B_K\}$ and $|B_k|=p_k$ for all $k \in [1:K]$.
We consider the penalized likelihood criterion
$$
\Phi:\begin{array}{ccc}
S_p^{++}(\R)& \longrightarrow & \R \\
\Gamma & \longmapsto &  l_\Gamma +\kappa \pen(\Gamma),
\end{array}
$$
where $\kappa\geq 0$. In this work, we suggest to estimate $\Sigma$ by the minimizer of $\Phi$, for some choice of penalisation $\kappa$. First, we show in Proposition \ref{prop_min_phi} that a minimizer of $\Phi$ can only be a block-diagonal decomposition of $S$.

\begin{prop}\label{prop_min_phi}
If $\Gamma$ is a minimizer of $\Phi$, then, there exists $ B\in \mathcal{P}_p$ such that $\Gamma=S_B$.
\end{prop}

Hence, the minimization problem on $S_p^{++}(\R)$ becomes a minimization problem on the finite set $\{ S_B,\; B\in \mathcal{P}_{p} \}$. So, we define $\Psi(B):= \Phi(S_B)$ and we suggest to estimate $B^*$ by
\begin{equation}\label{eq_Bhat}
\widehat{B}_{tot}:= \underset{B \in \mathcal{P}_p}{\argmin} \Psi(B),
\end{equation}
as the minimum structure of the penalized log-likelihood. In this paper, we study theoretically this estimator of $B^*$. However, it is unimplementable in high dimension since the number of partitions $B \in \mathcal{P}_p$ is too large. Hence, we will also define other estimators less costly, and study them theoretically.

\subsection{Convergence in high dimension}\label{sec_conv_high}

\subsubsection{Assumptions}\label{sec_assum}
In Section \ref{sec_conv_high}, we assume that $p$ and $n$ go to infinity. The true covariance matrix $\Sigma$ is not constant and depends on $n$ (or $p$). Nevertheless, to simplify notation, we do not write the dependency on $n$. 
In all Section \ref{sec_conv_high}, we choose a penalisation coefficient $\kappa=\frac{1}{pn^\delta}$ for a fixed $\delta \in ]1\slash 2,1[$.

We also add the following assumptions on $\Sigma$ along Section \ref{sec_conv_high}.
\begin{cond}\label{cond1}
$p \slash n \longrightarrow y \in ]0,1[$.
\end{cond}
\begin{cond}\label{cond2}
There exist $\lambda_{\inf}>0$ and $\lambda_{\sup}<+\infty$ such that, for all $n$, the eigenvalues of $\Sigma$ are in $[\lambda_{\inf},\lambda_{\sup}]$.
\end{cond}
\begin{cond}\label{cond3}
There exists $m\in \N^*$ such that for all $n$, all the blocks of $\Sigma$ are smaller than $m$, i.e. $\forall A \in B^*$, we have $|A| \leq m$.
\end{cond}
For a $q\times q$ matrix $M=(m_{ij})_{(i,j)\in [1:q]^2}$, we let $\|M\|_{\max}=\max_{(i,j)\in [1:q]^2}|m_{ij}|$.
\begin{cond}\label{cond4}
There exists $a>0$ such that for all $n$ and for all $ B < B^*$, we have $ \| \Sigma_{B}-\Sigma\|_{\max} \geq  a n^{-1\slash 4}$.
\end{cond}
These four mild assumptions are discussed in Section \ref{sec_assump_discuss}.  However, we also focus on the case when Condition 4 does not hold. We will provide similar results, both when assuming Conditions \ref{cond1} to \ref{cond4}, and when only Conditions \ref{cond1}, \ref{cond2} and \ref{cond3} hold.

\subsubsection{Convergence of $\widehat{B}$ and reduction of the cost}\label{section_conv_cost}

Now that we defined our estimator $\widehat{B}_{tot}$ of the true decomposition $B^*$ in Equation \eqref{eq_Bhat} and we added assumptions in Section \ref{sec_assum}, we give the convergence of $\widehat{B}$ in Proposition \ref{prop_conclu}. Although $\widehat{B}_{tot}$ is not computable in practice, its convergence remains interesting to strengthen the choice of the penalized likelihood criterion and will be useful to prove the convergence of more practical estimators. In Section \ref{sec_conv_high}, all the limits statements are given as $n,p \to +\infty$.

\begin{prop}\label{prop_conclu}
Under Conditions 1 to 4 and for a fixed $\delta \in ]1\slash 2 , 1[$, we have
$$
\PP\left( \widehat{B}_{tot}=B^* \right)\longrightarrow 1. 
$$
\end{prop}

Hence, under Conditions 1 to 4, the estimator $\widehat{B}_{tot}$ is equal to the true decomposition $B^*$ with probability which goes to one. When Condition 4 does not hold, we can not state such a convergence result but we get a weaker result in Proposition \ref{prop_Btot_Ba}. In this case, we need to define $B({\alpha})$ as the partition given by thresholding $\Sigma$ by $n^{-\alpha}$. In other words, $B({\alpha})$ is the smallest (or finest) partition $B$ such that $\|\Sigma_B -\Sigma\|_{\max}\leq n^{-\alpha}$.

\begin{prop}\label{prop_Btot_Ba}
Under Conditions 1, 2 and 3, for all $\alpha_1<\delta\slash 2$ and $\alpha_2>\delta\slash 2$, we have
$$
\PP\left( B(\alpha_1) \not >  \widehat{B}_{tot} \leq B(\alpha_2) \right) \longrightarrow 1.
$$
\end{prop}

Thus, we defined a consistent estimator of $B^*$ that theoretically solves our problem of the lack of knowledge of the true decomposition $B^*$. However, computing $\widehat{B}_{tot}$ is very costly in practice. Indeed, the number of partitions of $[1:p]$ (the Bell number) is exponential in $p$. As in \cite{devijver_diagonal_2018}, we suggest to restrict our estimates of $B^*$ to the partitions given by thresholding the empirical correlation matrix $\widehat{C}:=(\widehat{C}_{ij})_{i,j\in [1:p]}$ where $\widehat{C}_{ij}:= s_{ij}\slash \sqrt{s_{ii} s_{jj}}$, with $S=(s_{ij})_{(i,j)\in [1:p]^2}$.
If $\lambda \in [0,1]$, let $B_\lambda$ be the finest partition of the thresholded empirical correlation matrix $\widehat{C}_\lambda:=(\widehat{C}_{ij} \mathds{1}_{|\widehat{C}_{i,j}|> \lambda})_{i,j\leq p}$. In other words, $B_\lambda:= B(\widehat{C}_\lambda)$. For some value $\lambda \in [0,1]$, $B_\lambda$ can be found by "Breath-First-Search" (BFS) \cite{lee_algorithm_1961}. Furthermore, we do not need to compute $B_\lambda$ for all $\lambda \in [0,1]$ and we suggest in the following three different choices of grids for $\lambda$.\bigskip

First, we suggest the grid $A_{\widehat{C}}:=\{|\widehat{C}_{ij}|\;|\;1\leq i<j\leq p \} $ and we define the estimator $\widehat{B}_{\widehat{C}}:= \underset{ B_\lambda\;|\; \lambda \in A_{\widehat{C}} }{\argmin}\Psi(B)$. This grid is the finest one because that gives all the partitions $\{B_\lambda|\; \lambda \in ]0,1[\}$. Almost surely, the coefficients $(\widehat{C}_{ij})_{i<j}$ are all different. Thus, when we increase the threshold to the next value of $A_{\widehat{C}}$, we only remove two symmetric coefficients from the empirical correlation matrix.

\begin{prop}\label{prop_comp1}
 The computational complexity of $\widehat{B}_{\widehat{C}}$ is $O(p^4)$.
\end{prop}

Using the rate of convergence of the estimated covariances and by Condition 4, we then suggest the estimator $\widehat{B}_{\lambda}:= B_{n^{-1\slash 3}}$, the partition of the empirical correlation matrix thresholded by $n^{-1\slash 3}$. With this threshold, we can not find all the partitions given by thresholded correlation matrix, but we only have to threshold by only one value.
\begin{prop}\label{prop_comp2}
The complexity of $\widehat{B}_\lambda$ is $O(p^2)$.
\end{prop}
One can see that reducing the grid of thresholds to one value reduces the complexity of the estimator of $B^*$. Finally, we suggest a third grid, in the c case when the maximal size of the groups $m$ is known.

Let $A_s:=\{ s \slash p, (s+1)\slash p,... , (p-1)\slash p, 1\}$, where $s$ is the smallest integer such that all the groups of $B_{s\slash p}$ have a cardinal smaller than $m$. The deduced estimator is $\widehat{B}_s:= \underset{ B_\lambda\;|\; \lambda \in A _{s}}{\argmin}\Psi(B)$. So, this grid is the set $\{ l\slash p\;|\; l\in [1:p]\}$ restricted to the thresholds that give fine enough partition (with groups of size smaller than $m$).

\begin{prop}\label{prop_comp3}
The complexity of $\widehat{B}_s$ is $O(p^2)$.
\end{prop}
One can see that the complexity of this estimator is as small as the complexity of the previous estimator $\widehat{B}_{\lambda}$. Furthermore, it ensures that the estimated blocks are not too large, which was not the case with the previous estimator. However, the computation of $\widehat{B}_s$ requires the knowledge of $m$ while the other estimators do not.

\bigskip

Now that we have defined new estimators of $B^*$, we give their convergence in the following proposition.

\begin{prop}\label{prop_cost}
Let $\widehat{B}$ be either $\widehat{B}_{tot},\; \widehat{B}_{\widehat{C}},\; \widehat{B}_\lambda$ or $\widehat{B}_s$ indifferently. Under Conditions 1 to 4 and for a fixed $\delta \in ]1\slash 2 , 1[$, we have
$$
\PP\left( \widehat{B}=B^*\right) \longrightarrow 1.
$$
\end{prop}
\bigskip

When Condition 4 is not satisfied, we do not study the convergence of the previous estimators. In this case, we suggest to estimate $B^*$ by $B_{n^{-\delta \slash 2}}$, which is the partition given by the empirical correlation matrix thresholded by $n^{-\delta \slash 2}$. The complexity of this estimator is $O(p^2)$, as for the previous estimator $\widehat{B}_\lambda=B_{n^{-1\slash 3}}$. We show the convergence of this estimator in Proposition \ref{prop_seuil_ba}.

\begin{prop}\label{prop_seuil_ba}
Under Conditions 1, 2 and 3, if $\alpha_1<\delta \slash 2$ and $\alpha_2> \delta \slash 2$,
$$
\PP\left(B(\alpha_1)\leq  B_{n^{-\delta\slash 2}} \leq  B(\alpha_2)\right)\longrightarrow 1.
$$
\end{prop} 
As Condition 4 is not satisfied, the true partition $B^*$ is again not reached by this estimator.
Nevertheless, we get stronger results for the practical estimator $B_{n^{-\delta \slash 2}}$ than for the theoretical estimator $\widehat{B}_{tot}$ when Condition 4 is not verified. Indeed, the condition "to be larger or equal than" is stronger that "not to be smaller than".

\subsubsection{Convergence of the estimator of the covariance matrix}
We have seen in Propositions \ref{prop_cost} and \ref{prop_seuil_ba} how to estimate the decomposition $B^*$ by $\widehat{B}$. Now to estimate the covariance matrix $\Sigma$, it suffices to impose the block-diagonal decomposition $\widehat{B}$ to the empirical covariance matrix $S_{\widehat{B}}$.
We show in Proposition \ref{prop_rate_sigma} that the resulting block-diagonal matrix estimator $S_{\widehat{B}}$ reaches the optimal rate of convergence under Conditions 1 to 4.

\begin{prop}\label{prop_rate_sigma}
Let $\|.\|_F$ be the Frobenius norm defined by $\|\Gamma\|_F^2:=\sum_{i,j=1}^p \gamma_{ij}^2$. Let $\widehat{B}$ be either $\widehat{B}_{tot},\; \widehat{B}_{\widehat{C}}, \widehat{B}_{\lambda}$ or $\widehat{B}_s$.
Under Conditions 1 to 4 and for a fixed $\delta \in ]1\slash 2 , 1[$, we have
$$
\frac{1}{p}\| S_{B^*}-\Sigma \|_F^2=O_p(1\slash n)
$$
and 
$$
\frac{1}{p}\| S_{\widehat{B}}-\Sigma \|_F^2=O_p(1\slash n).
$$
Moreover, it is the best rate that we can have because
$$
\frac{1}{p}\| S_{B^*}-\Sigma \|_F^2\neq o_p(1\slash n).
$$
\end{prop}
Thus, we see that the quantity $\frac{1}{p}\| S_{\widehat{B}}-\Sigma \|_F^2$ decreases to 0 in probability with rate $1\slash n$, which is the same rate as $S_{B^*}$ if we know the true decomposition $B^*$. Thus, the lack of knowledge of $B^*$ does not deteriorate the convergence of our estimator.

Now that we gave the rate of convergence of our estimator $S_{\widehat{B}}$, we compare it with that of the empirical estimator $S$ in the next proposition.

\begin{prop}\label{prop_rate_S} Under Conditions 1 and 2, the rate of the empirical covariance is
$$
\frac{1}{p}\| S-\Sigma \|_F^2=O_p(p\slash n).
$$
and we have
$$
\E\left(\frac{1}{p}\| S-\Sigma \|_F^2\right)\geq \frac{\lambda_{\inf}^2p}{2n}.
$$
\end{prop}
So, we know that $\frac{1}{p}\| S-\Sigma \|_F^2$ is lower-bounded in average and is bounded in probability. Thus, the rate of convergence of our suggested estimator $S_{\widehat{B}}$ is better than the empirical covariance matrix $S$.\\

If Condition 4 does not hold, the rate of convergence is given in the following proposition.
\begin{prop}\label{prop_rate_sigma_ba}
Under Conditions 1, 2 and 3, for all $\delta \in ]0,1[$ and for all $\varepsilon>0$,  we have
$$
\frac{1}{p}\| S_{B_{n^{-\delta\slash 2}}}-\Sigma\|_F^2 =o_p\left( \frac{1}{n^{\delta-\varepsilon}}\right).
$$
\end{prop}
We remark that, for $\delta$ close to $1$, this rate of convergence almost reaches the optimal rate of $S_{B^*}$, whereas the partition estimator $B_{n^{-\delta\slash 2}}$ does not reach the true decomposition $B^*$. That comes from the fact that the elements $\sigma_{ij}$ of $\Sigma$ such that the indices $(i,j)$ are not in the estimated partition $B_{n^{-\delta\slash 2}}$ are small (with high probability). Hence, estimating these values by $0$ does not increase so much the error $\frac{1}{p}\| S_{B_{n^{-\delta\slash 2}}}-\Sigma\|_F^2$.

\bigskip

Theoretical guaranties for a block-diagonal estimator of the covariance matrix are also provided in \cite{devijver_diagonal_2018}. Their framework is more general, with a true covariance matrix which is not necessarily block-diagonal. They bound the average of the square Hellinger distance between the true normal density and the density with the block-diagonal estimated covariance matrix. However, when $p\slash n$ does not go to $0$, their theoretical results becom uninformative. Indeed, they give an upper-bound which is larger than one, while the square Hellinger distance remains always smaller than $1$.

\subsubsection{Discussion about the assumptions}\label{sec_assump_discuss}

For  the previous results, we needed to make four assumptions on $\Sigma$ (Conditions 1 to 4, given in Section \ref{sec_assum}).

Condition 1 provides a standard setting for high-dimensional problems, in particular for estimation of covariance matrices \cite{Mar_enko_1967,smallest_1985_silverstein}. Studying an higher dimensional setting where $p\slash n \longrightarrow +\infty$ would be interesting in future work.

Condition 2 is needed to bound the operator norm of $\Sigma$ and $\Sigma^{-1}$ and the eigenvalues of the empirical covariance matrix (with high probability). It also enables to bound the diagonal terms of $\Sigma$, which allow to derive the rate of convergence of each component of the empirical covariance (using in particular Bernstein's inequality, see the proofs for more details).

Condition 3 states that the blocks of the true decomposition have a maximal size. It implies that the number of non-zero terms of $\Sigma$ is $O(p)$.

Condition 4 requires that a finer block decomposition $\Sigma_B$ is not too close to the true $\Sigma$. This condition is needed to not confuse $B^*$ with a finer decomposition. However, Condition 4 seems to be less mild than the others. That is why we also focus on the case when Condition 4 is not satisfied.
\bigskip

Nevertheless, even Condition 4 is not so restrictive. Indeed, we suggest in Proposition \ref{prop_generate_sigma} a reasonable example where $\Sigma$ is randomly generated and where a condition similar to Condition \ref{cond4} holds.

\begin{prop}\label{prop_generate_sigma}
Let $L\in \N$ and $\varepsilon>0$. Assume that for all $p$, $\Sigma$ is generated in the following way:
\begin{itemize}
\item Let $B^*$ be a partition of  $[1:p]$ such that all its elements have a cardinal between $10$ and $m\geq 10$. Let $K$ be the number of groups (the cardinal of $B^*$). For all $k \in [1:K]$, let $p_k$ be the cardinal of the "$k$-th element" of $B^*$.
\item For all $k \in [1:K]$, let $(U_{i}^{(l)})_{i\in [1 : p_k],\; l\in [1: L]}$ be i.i.d. with distribution $\mathcal{U}([-1,1])$. Let $U\in \mathcal{M}_{L,p_k}(\R)$ such that the coefficient $(l,i)$ is $U_{i}^{(l)}$. Let $\Sigma_{B_k^*}=U^TU+\varepsilon I_{p_k}$, where $\Sigma_{B_k^*}$ is the sub-matrix of $\Sigma$ indexed by the elements of $B_k^*$.
\item Let $\sigma_{ij}=0$ for all $(i,j) \notin B^*$.
\end{itemize}
Then, Conditions 2 and 3 are verified and the following slightly modified version of Condition 4 is satisfied for all $a>0$:
$$
\PP\left( \exists B< B^*,\; \|\Sigma_{B}-\Sigma\|_{\max} < a n^{-\frac{1}{4}} \right) \longrightarrow 0.
$$
Thus, if $p\slash n \longrightarrow y\in ]0,1[$, the conclusions of Propositions \ref{prop_conclu}, \ref{prop_cost} and \ref{prop_rate_sigma} remain true when the probabilities are defined with respect to $\Sigma$ and $X$ which distribution conditionally to $\Sigma$ is $\mathcal{N}(\mu,\Sigma)$.
\end{prop}

\subsubsection{Numerical applications}\label{section_num_Sigma}
We present here numerical applications of the previous results with simulated data. We generate a covariance matrix $\Sigma$ as in Proposition \ref{prop_generate_sigma} with blocks of random size distributed uniformly on $[10:15]$, with $L=5$ and $\varepsilon=0.2$. We assume here that we know that the maximal size of the block is $m=15$, so we can use the estimator $\widehat{B}=\widehat{B}_s$ given in Proposition \ref{prop_cost} to reduce the complexity to $O(p^2)$ and to prevent the blocks from being too large.

\begin{figure}
    \centering
    \includegraphics[height=10cm, width=10cm]{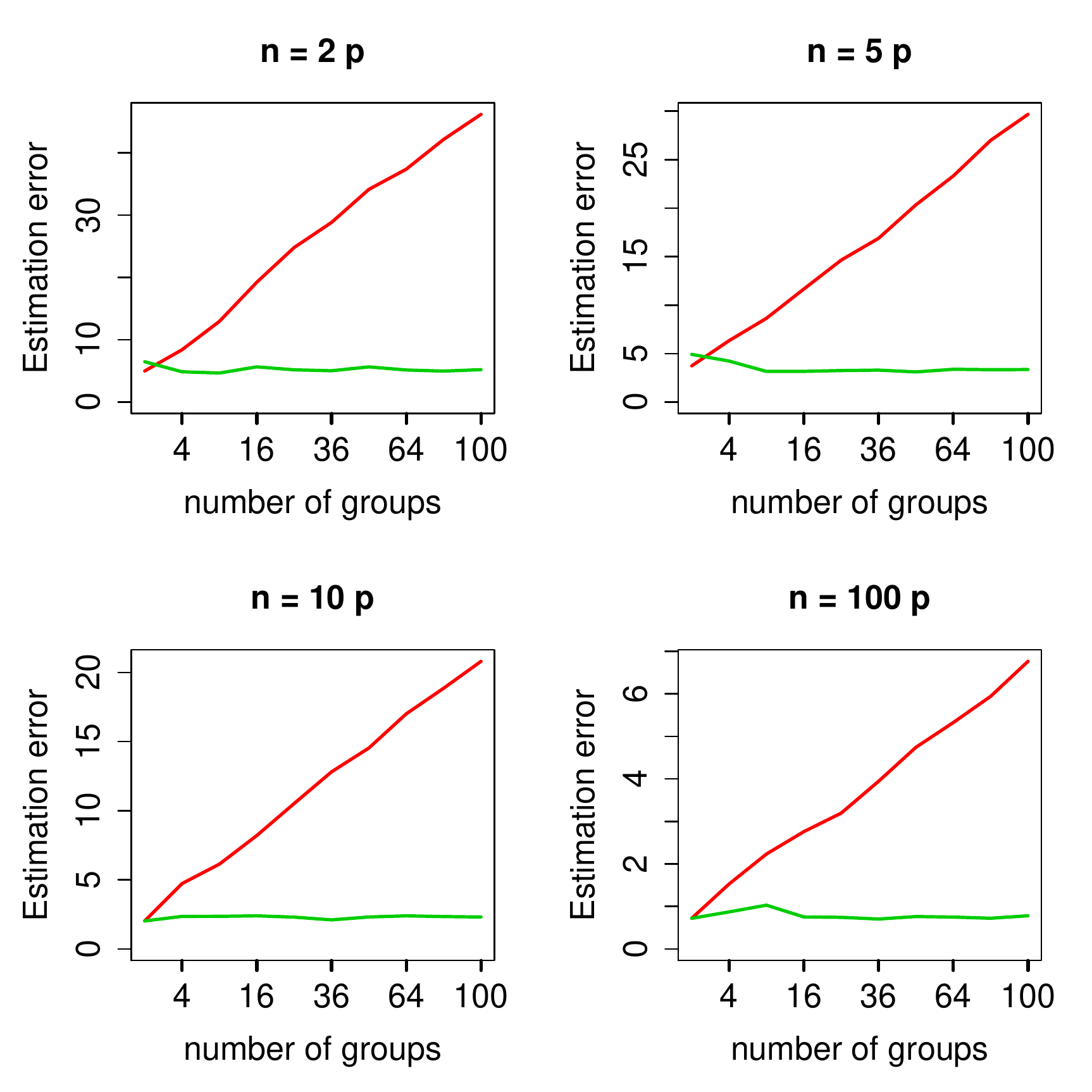}
    \caption{Frobenius error of the empirical covariance matrix $S$ in red and the suggested estimator $S_{\widehat{B}}$ in green, in function of the number of groups $K$. The scale of the x-axis is in $\sqrt{K}$. }
    \label{fig:error_mat}
\end{figure}

We plot in Figure \ref{fig:error_mat} the Frobenius norm of the error of the empirical covariance matrix $S$ and the Frobenius norm of the error of the suggested estimator $S_{\widehat{B}}$, with $n=N\, p$ for different values of $N$.  We can remark that the error of $S$ is in $\sqrt{K}$ (where $K$ is the number of groups) whereas the error of $S_{\widehat{B}}$ stays bounded as in Proposition \ref{prop_rate_sigma}. For $K=100$, the Frobenius error of $S_{\widehat{B}}$ on Figure \ref{fig:error_mat} is about 10 times smaller than the one of $S$.

\subsection{Convergence and efficiency in fixed dimension}\label{sec_pfix}

In this section, $p$ and $\Sigma$ are fixed and $n$ goes to $+\infty$. We choose a different penalisation $\kappa=\frac{1}{pn^\delta}$ with $\delta \in ]0,1\slash 2[$ (instead of $\delta \in ]1\slash 2,1[$ in the previous setting). This framework enables to study the efficiency of estimators of $\Sigma$. Contrary to the high-dimensional setting of Section \ref{sec_conv_high}, we do not assume particular condition in addition to the ones given in Section \ref{sec_problem}.

We first give the convergence of $\widehat{B}_{tot}$ defined in Equation \eqref{eq_Bhat} in the next proposition.

\begin{prop}\label{prop_conv_pfix}
We have
$$
\PP\left( \widehat{B}_{tot}=B^* \right)\longrightarrow 1. 
$$
\end{prop}

\begin{coro}
Let  $\widehat{B}_{\widehat{C}}:= \underset{ B_\lambda\;|\; \lambda \in A_{\widehat{C}} }{\argmin}\Psi(B)
$, where $A_{\widehat{C}}:=\{|\widehat{C}_{ij}|\;|\;1\leq i<j\leq p \} $ as in Proposition \ref{prop_cost}. Then
$$
\PP\left( \widehat{B}_{\widehat{C}}=B^* \right)\longrightarrow 1. 
$$
\end{coro}
In the rest of Section \ref{sec_pfix}, we write $\widehat{B}$ for  $\widehat{B}_{tot}$ or $\widehat{B}_{\widehat{C}}$.
The aim of this framework is to show that the suggested estimator $S_{\widehat{B}}$ is asymptotically efficient as if the true decomposition $B^*$ were known. 

As the parameter $\Sigma$ is in the set $S_p^{++}(\R)$ or even $S_p^{++}(\R,B^*)$, which are not open subsets of $\R^{p^2}$, the classical Cram\'{e}r-Rao bound is no longer a lower-bound for the estimation error. Furthermore, as $B^*$ is not known, the number of parameters of $S_{\widehat{B}}$ is not constant. That is why the classical Cram\'{e}r-Rao bound is not relevant in our setting. We remark that applying this classical Cram\'{e}r-Rao bound to a subset of the matrix estimator does not solve this problem.

A specific Cram\'{e}r-Rao bound is suggested in \cite{stoica_Cramer-rao_1998} for parameters and estimators which satisfy continuously differentiable constraints. We shall consider linear constraints here. We let $\theta \in \mathbb{R}^d$ be the parameter, that is assumed to be restricted to a linear subspace $V$ of dimension $q$ in $\R^d$. In this case, if $U\in \mathcal{M}_{d,q}(\R)$ is a matrix whose columns are the elements of an orthonormal basis of $V$ and if $J$ is the Fisher Information Matrix (FIM) of $\theta$ in the non-constraint case, \cite{stoica_Cramer-rao_1998} states that for unbiased estimator $\widehat{\theta}\in V$, we have
\begin{equation}\label{eq_CRbound}
\E\left[(\widehat{\theta}-\theta)(\widehat{\theta}-\theta)^T\right]\leq U(U^T J U)^{-1} U^T,
\end{equation}
where $\leq$ is the partial order on the symmetric positive semi-definite matrices.

In our setting, remark that $S_p^{++}(\R)$ is an open subset of the linear subspace $S_p(\R)$ of symmetric matrices and $S_p^{++}(\R,B^*)$ is an open subset of the linear subspace $
\overline{S_p(\R,B^*)}:=\{ \Gamma \in S_p(\R),\; \Gamma_{B^*}=\Gamma \}$. We let $\vecc(\Sigma)$ be the column vectorization of $\Sigma$. Hence, the parameter is $\vecc(\Sigma)$ and there are $p(p-1)/2$ linear constraints arising from the symmetry and $p(p-1)/2 - \sum_{k=1}^K p_k(p_k-1)/2$ linear constraints arising from the block structure $B^*$.

So, the Cram\'{e}r-Rao bound of Equation \eqref{eq_CRbound} is adapted to our framework, by considering the parameter $\vecc(\Sigma)\in \R^{p^2}$, and we say that an estimator is efficient if it reaches the Cram\'{e}r-Rao bound \eqref{eq_CRbound} (meaning that there is an equality in this equation), where the constraints (symmetry only or symmetry and block structure) will be stated explicitly.\bigskip

Proposition \ref{prop_CR_Sigma} states that, in general, the empirical covariance matrix is efficient with this Cram\'{e}r-Rao bound. 
This supports this choice of Cram\'{e}r-Rao Bound, since in fixed dimension, one would expect that the empirical matrix is the most appropriate estimator.

If the empirical covariance matrix did not reach the Cram\'{e}r-Rao Bound, we could not hope that $S_{\widehat{B}}$ would be efficient in the model where $B^*$ was known, and this Cram\'{e}r-Rao bound would not be well tuned to our problem.

\begin{prop}\label{prop_CR_Sigma}
If $\mu$ is known, the empirical estimator $S$ is an efficient estimator of $\Sigma$ in the model $\{\mathcal{N}(\mu,\Sigma),\; \Sigma\in S_p^{++}(\R)\}$.
\end{prop}

\begin{rmk}
In Proposition \ref{prop_CR_Sigma}, we assume that $\mu$ is known to reach the Cram\'{e}r-Rao bound for fixed $n$ (and not only asymptotically). This will be the same in Proposition \ref{prop_CR_SigmaB}.
\end{rmk}

Now, we deduce the efficiency of $S_{B^*}$ when $B^*$ is known. 

\begin{prop}\label{prop_CR_SigmaB}
If $\mu$ and $B^*$ are known, $S_{B^*}$ is an efficient estimator of $\Sigma$ in the model  $\{\mathcal{N}(0,\Sigma), \; \Sigma\in S_p^{++}(\R,B^*)\}$.
\end{prop}

Finally, Proposition \ref{prop_CR_final} states the asymptotic efficiency of our estimator $S_{\widehat{B}}$ (even for unknown $\mu$)

\begin{prop}\label{prop_CR_final}
$$
\sqrt{n}( \vecc( S_{\widehat{B}})- \vecc(\Sigma) ) \overset{\mathcal{L} }{\underset{n\rightarrow+\infty }{\longrightarrow}} \mathcal{N} (0,\mathrm{CR}(\Sigma,B^*)),
$$
where $\mathrm{CR}(\Sigma,B^*)$ is the Cram\'{e}r-Rao bound of $\vecc(\Sigma)$ in the model $\{\mathcal{N}(0,\Sigma), \; \Sigma\in S_p^{++}(\R,B^*)\}$.
\end{prop}

The explicit expression of the $p^2 \times p^2$ matrix $\mathrm{CR}(\Sigma,B^*)$ can be found in the appendix where Propositions \ref{prop_CR_Sigma}, \ref{prop_CR_SigmaB} and \ref{prop_CR_final} are proved.

\section{Application to the estimation of the Shapley effects}\label{section_shap}

In this section, we apply the block-diagonal estimation of the covariance matrix $\Sigma$ to estimate the Shapley effects in high dimension and for Gaussian linear models. In Section \ref{section_def_shap}, we recall the definition of the Shapley effects with their particular expression in the Gaussian linear framework with a block-diagonal covariance matrix. In Section \ref{section_conv_shap}, we address the problem of estimating the Shapley effects when the covariance matrix $\Sigma$ is estimated. We derive the convergence of the estimators of the Shapley effects from the results of Section \ref{section_cov}.

\subsection{The Shapley effects}\label{section_def_shap}

Let $(X_i)_{i\in [1:p]}$ be random inputs variables on $\R^p$ and let $Y=f(X)$ be the real random output variable in $L^2$ . We assume that $\V(Y)\neq 0$. Here, $f$ can be a numerical simulation model \cite{santner2003design}.

If $u\subset[1:p]$ and $x=(x_i)_{i \in [1:p]}\in \R^p$, we write $x_u:=(x_i)_{i\in u}$.
We can define the Shapley effects as in \cite{owen_sobol_2014} for the input variable $X_i$ as:
\begin{equation}\label{Shapley}
\eta_i:=\frac{1}{p\V(Y)}\sum_{u\subset -i}  \begin{pmatrix}
p-1\\ |u|
\end{pmatrix} ^{-1}\left(\V(\E(Y|X_{u\cup \{i\}}))-\V(\E(Y|X_u)) \right)
\end{equation}
where $-i$ is the set $[1:p]\setminus \{i\}$. One can see in Equation \eqref{Shapley} that adding a $X_i$ to $X_u$ changes the conditional expectation of $Y$, and increases the variability of this conditional expectation. The Shapley effect $\eta_i$ is large when, on average, the variance of this conditional expectation increases significantly when $X_i$ is observed. Thus, a large Shapley effect $\eta_i$ corresponds to an important input variable $X_i$.

The Shapley effects have interesting properties for global sensitivity analysis. Indeed, there is only one Shapley effect for each variable (contrary to the Sobol indices). Moreover, the sum of all the Shapley effects is equal to $1$ (see \cite{owen_sobol_2014}) and all these values lie in $[0,1]$ even with dependent inputs. This is very convenient for the interpretation of these sensitivity indices.\\

Here, we assume that $X\sim \mathcal{N}(\mu,\Sigma)$, that $\Sigma \in S_p^{++}(\R)$ and that the model is linear, that is $f:x\longmapsto \beta_0+\beta^T x$, for a fixed $\beta_0\in \R$ and a fixed vector $\beta$. This framework is widely used to model physical phenomena (see for example \cite{kawano_evaluation_2006,hammer_approximate_2011,rosti_linear_2004}). Indeed, uncertainties are often modelled as Gaussian variables and an unknown function is commonly estimated by its linear approximation.  Furthermore, the main focus on this paper is on the high-dimensional case, where $p$ is large. In high dimension, linear models are often considered, as more complex models are not necessarily more relevant.
In this framework, the sensitivity indices can be calculated explicitly \cite{owen_shapley_2017}:
\begin{eqnarray}\label{eq_varianceShapley}
\eta_i:=\frac{1}{p\V(Y)}\sum_{u\subset -i}  \begin{pmatrix}
p-1\\ |u|
\end{pmatrix}^{-1}\left(\V(Y|X_u)-\V(Y|X_{u\cup \{i\}}) \right).
\end{eqnarray}
with
\begin{equation}\label{eq_V}
\V(Y|X_u)=\V(\beta_{-u}^T X_{-u}|X_u)=\beta_{-u}^T(\Sigma_{-u,-u}-\Sigma_{-u,u}\Sigma_{u,u}^{-1}\Sigma_{u,-u})\beta_{-u}
\end{equation}
where $\beta_u:=(\beta_i)_{i\in u}$ and $\Gamma_{u,v}:=(\Gamma_{i,j})_{i\in u, j\in v}$. 
Thus, in the Gaussian linear framework, the Shapley effects are functions of the parameters $\beta$ and $\Sigma$. 

Despite the analytical formula \eqref{eq_V}, even in the case where $\Sigma$ and $\beta$ are known, the computational cost of the Shapley effects remains an issue when the number of input variables $p$ is too large ($p\geq 30$), as it is highlighted in \cite{broto_sensitivity_2019}. Indeed, the Shapley effects depend on $2^p$ values, namely the $(\V(Y|X_u))_{u\subset[1:p]}$. However, when the covariance matrix is block-diagonal, \cite{broto_sensitivity_2019} showed that this high-dimensional computational problem boils down to a collection of lower dimensional problems. 

Indeed, assume that $\Sigma \in S_p^{++}(\R,B^*)$ with $B^*=\{B_1^*,B_2^*,...,B_K^*\}$. If $i\in [1:p]$, let $[i]$ denotes the group of $i$, that is $i \in B_{[i]}^*$. Using Corollary 2 of \cite{broto_sensitivity_2019}, we have for all $i \in [1:p]$,
\begin{equation}\label{eq_shap_groups}
\eta_i=\frac{1}{  \beta^T \Sigma \beta}\frac{1}{|B_{[i]}^*|}\sum_{u\subset B_{[i]}^*-i} \begin{pmatrix}
|B_{[i]}^*|-1 \\ |u|
\end{pmatrix}^{-1}\left(V_{u }^{B_{[i]}^*}-V_{u\cup\{i\} }^{B_{[i]}^*}\right),
\end{equation}
where for all $v\subset B_{[i]}^*$, 
\begin{equation}\label{eq_VB}
V_{v}^{B_{[i]}^*}:=\V\left( \beta_{B_{[i]}^*}^T X_{B_{[i]}^*}| X_v\right)=\beta_{B_{[i]}^*- v}^T\left( \Sigma_{B_{[i]}^*- v,B_{[i]}^*- v}- \Sigma_{B_{[i]}^*- v, v}\Sigma_{v,v}^{-1} \Sigma_{v,B_{[i]}^*- v} \right) \beta_{B_{[i]}^*- v}
\end{equation}
and where $w-v=w\subset v$.
Thus, when $\Sigma$ and $\beta$ are known, to compute all the Shapley effects $(\eta_i)_{i\in [1:p]}$, we only have to compute the $\sum_{k=1}^K 2^{|B_k^*|}$ values $\{\V(Y|X_u),\; u\subset B_k^*,\;k\in[1:K]\}$ instead of all the $2^p$ values $\{\V(Y|X_u), \; u\subset[1:p]\}$. Some numerical experiments highlighting this gain are given in \cite{broto_sensitivity_2019}. The complexity of the computation of the Shapley effects is $O(K 2^m)$, where $m$ denotes the size of the maximal group in $B^*$.

If  $\Sigma$ is known, but the decomposition $B^*$ is unknown, we can compute $B^*$ from $\Sigma$. We can for example use "Breath-First-Search" (BFS). The complexity of this algorithm is in $O(p m^2)$.

To conclude, when the parameters $\beta$ and $\Sigma$ are known with $\Sigma \in S_p^{++}(\R,B^*)$, the computation of all the Shapley effects has a complexity $O(K 2^m)$.
\bigskip

\subsection{Estimation of the Shapley effects in high dimension} \label{section_conv_shap}
We now address the problem when the parameters $\mu$, $\Sigma$ and thus $B^*$ are unknown. 

We assume that we just observe a sample $(X^{(l)}, \tilde{Y}^{(l)})_{l\in [1:n]}$ where $\tilde{Y}=(\tilde{Y}^{(l)})_{l\in [1:n]}$ are noisy observations:
$$
\tilde{Y}^{(l)}= \beta_0+ \beta^TX^{(l)}+ \varepsilon^{(l)}, 
$$
for $l \in [1:n]$ where $(\varepsilon^{(l)})_{l\in [1:n]}$ are i.i.d. with distribution $\mathcal{N}(0, \sigma^2 _n)$ and where $\sigma_n \leq C_{\sup}$ is unknown, where $C_{\sup}$ is a fixed finite constant.

Remark that the computation of the Shapley effects requires the parameters $\beta$ and $\Sigma$ (see Equations \eqref{eq_varianceShapley} and \eqref{eq_V} or \eqref{eq_shap_groups} and \eqref{eq_VB}). Here, as we do not know the parameters $\beta$ and $\Sigma$, we will estimate them and replace the true parameters by their estimation in Equations \eqref{eq_varianceShapley} and \eqref{eq_V} or \eqref{eq_shap_groups} and \eqref{eq_VB}.

First, we estimate $(\beta_0\;\beta^T)^T$ as usual by $$
\begin{pmatrix}
\widehat{\beta}_0\\ \widehat{\beta}
\end{pmatrix}:= (A^T A)^{-1}A^T \tilde{Y},
$$
where $A\in \mathcal{M}_{n,p+1}(\R)$ is defined by $A_{l,i+1}:=X_i^{(l)}$ and $A_{l,1}=1$, and where $n> p$.

At first glance, we could estimate $\Sigma$ by the empirical covariance matrix $S$ and replace it in the computation of the Shapley effects given by Equations \eqref{eq_varianceShapley} and \eqref{eq_V} or \eqref{eq_shap_groups} and \eqref{eq_VB}. However, $B^*$ is not known and we can not find it using BFS with the empirical covariance matrix $S$ (which usually has the simple structure $\{ [1:p] \}$ with probability one). Thus, we can not use the formula \eqref{eq_shap_groups} of the Shapley effects with independent groups. So, the only way to estimate the Shapley effects is using Equations \eqref{eq_varianceShapley} and \eqref{eq_V}, replacing $\Sigma$ by the empirical covariance matrix $S$. However, as we have seen, the complexity of this computation would be exponential in $p$ and it would be no longer tractable for $p\geq 30$. Furthermore, in high dimension, the Frobenius error between $S$ and $\Sigma$ does not go to $0$ (see Proposition \ref{prop_rate_sigma}). Thus, using the empirical covariance matrix could yield estimators of the Shapley effects that do not converge.\\

For that reason, to estimate $\eta=(\eta_i)_{i\in [1:p]}$, we suggest to estimate $B^*$ by $\widehat{B}$ (defined in Section \ref{section_conv_cost}) and $\Sigma$ by $S_{\widehat{B}}$ and to replace them in the analytical formula \eqref{eq_shap_groups}. We write $\widehat{\eta}=(\widehat{\eta}_i)_{i\in [1:p]}$ the estimator of the Shapley effects obtained replacing $B^*$ by $\widehat{B}$, $\Sigma$ by $S_{\widehat{B}}$ and $\beta$ by $\widehat{\beta}$ in Equations \eqref{eq_shap_groups} and \eqref{eq_VB}. We use our previous results on the estimation of the covariance matrix to obtain the convergence rate of $\widehat{\eta}$.

We focus on the high-dimensional case, when $p$ and $n$ go to $+\infty$. In this case, $\beta$ and $\Sigma$ are not fixed but depend on $n$ (or $p$). As in Section \ref{sec_conv_high}, we choose $\kappa=\frac{1}{p n^{\delta}}$ with $\delta \in ]1\slash 2,1[$ to compute $\widehat{B}$.
To prevent problematic cases, we also add an assumption on the vector $\beta$.

\begin{cond}\label{cond5}
There exist $\beta_{\inf}>0$ and $\beta_{\sup}<+\infty$ such that for all $n$ and for all $j\leq p$, we have $\beta_{\inf}\leq |\beta_j| \leq \beta_{\sup}$.
\end{cond}

\begin{prop}\label{prop_conv_shapley_high}
Under Conditions 1 to 5 and if $\delta \in ]1\slash 2,1[$, then for all $\gamma>1\slash 2$,  we have
$$
\sum_{i=1}^p \left|\widehat{\eta}_i-\eta_i\right|=o_p\left(\frac{\log(n)^\gamma}{\sqrt{n}}\right).
$$
\end{prop}
Recall that $\sum_{i=1}^p\eta_i=1$. Thus, to quantify the error estimation, the value of $\sum_{i=1}^p \left|\widehat{\eta}_i-\eta_i\right|$ is a relative error. Proposition \ref{prop_conv_shapley_high} states that this relative error goes to zero at the parametric rate $1/n^{1/2}$, up to a logarithm factor.

We have seen in Section \ref{section_def_shap} that, once we have the block-diagonal covariance matrix, the computation of the Shapley effects has the complexity $O(K 2^m)$ which is equal to $O(n)$ under Condition 3. In Section \ref{sec_conv_high}, we gave four different choices of $\widehat{B}$, with four different complexities, all larger than $O(n)$. Thus, the complexity of the whole estimation of the Shapley effects (including the estimation of $\Sigma$) is the same as the complexity of $\widehat{B}$ (see Section \ref{section_conv_cost}).

\bigskip

When Condition 4 is not satisfied, we still have the convergence of the relative error, with almost the same rate.
\begin{prop}\label{prop_conv_shapley_high_ba}
Under Conditions 1, 2, 3 and 5, for all $\delta \in ]0,1[$, choosing the partition $B_{n^{-\delta\slash 2}}$ and for all $\varepsilon>0$,  we have
$$
\sum_{i=1}^p \left|\widehat{\eta}_i-\eta_i\right|=o_p\left(\frac{1}{n^{-(\delta-\varepsilon)\slash 2}}\right).
$$
\end{prop}

\begin{rmk}\label{rmk_conv_shapley_fix}
When the dimension $p$ is fixed, the rate of convergence is $O_p(1\slash \sqrt{n})$, as if we estimated $\Sigma$ by the empirical covariance matrix. Moreover, we have seen in Proposition \ref{prop_CR_final} that the computation of $S_{\widehat{B}}$ enables to reach asymptotically the Cram\'{e}r-Rao bound of \cite{stoica_Cramer-rao_1998} as if $B^*$ were known. We then deduce the asymptotic efficiency of $\widehat{\eta}$. If we define $g:\Sigma \mapsto \eta$, let $\mathrm{CR}(\eta,B^*):=Dg (\Sigma) \mathrm{CR}(\Sigma,B^*)Dg(\Sigma)$ be the Cram\'{e}r-Rao bound of $\eta$ in the model $\{\mathcal{N}(\mu,\Sigma), \; \Sigma\in S_p^{++}(\R,B^*)\}$. Thus,
$$
\sqrt{n}(\widehat{\eta}-\eta) \overset{\mathcal{L} }{\underset{n\rightarrow+\infty }{\longrightarrow}} \mathcal{N} (0,\mathrm{CR}(\eta,B^*)).
$$
\end{rmk}

\subsection{Numerical application}\label{section_num_shap}

We have seen in Proposition \ref{prop_generate_sigma} a way to generate $\Sigma$ which verifies Conditions 1 to 3 and some slightly modified version of Condition 4. So, with this choice of $\Sigma$, we derive in Proposition \ref{prop_shap_generate} the convergence of the Shapley effects estimation.

\begin{prop}\label{prop_shap_generate}
 Under Condition \ref{cond5}, if $\Sigma$ is generated as in Proposition \ref{prop_generate_sigma}, then, for all $\gamma> 1\slash 2$,
$$
\sum_{i=1}^p \left|\widehat{\eta}_i-\eta_i\right|=o_p\left(\frac{\log(n)^\gamma}{\sqrt{n}}\right),
$$
where the probabilities are defined with respect to $\Sigma$ and $X$, which distribution conditionally to $\Sigma$ is $\mathcal{N}(\mu,\Sigma)$.
\end{prop}

We now present a numerical application of Proposition \ref{prop_shap_generate}. The matrix $\Sigma$ is generated by Proposition \ref{prop_generate_sigma} as in Section \ref{section_num_Sigma}, with blocks of random size distributed uniformly on $[10,15]$, $L=5$ and $\varepsilon=0.2$. For all $p$, the vector $\beta$ is generated with distribution $\mathcal{U}([1,2]^p)$, so that Condition \ref{cond5} is satisfied. As in Section \ref{section_num_Sigma}, we assume that we know that the maximal size of the block is $m=15$, so we can use the estimator $\widehat{B}=\widehat{B}_s$ given in Proposition \ref{prop_cost}. As the computation of the Shapley effects is exponential in the maximal block size, the estimator $\widehat{B}_s$ is preferred. The complexity of the estimation of the Shapley effects is then in $O(p^2)$.

\begin{figure}
    \centering
    \includegraphics[height=10cm, width=10cm]{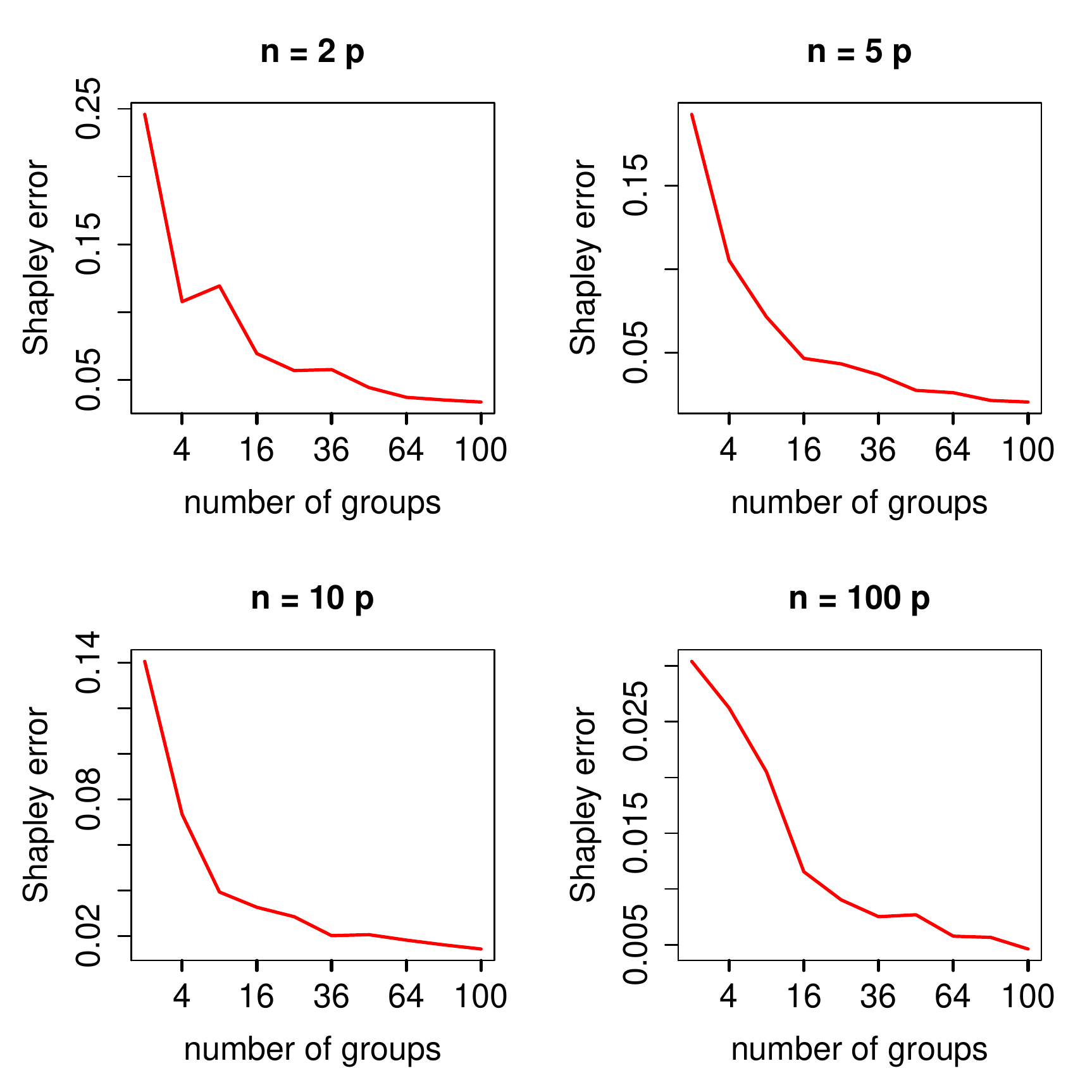}
    \caption{Sum of errors of the Shapley effects estimations in function of the number of groups $K$. The scale of the x-axis is in $\sqrt{K}$. }
    \label{fig:error_shap}
\end{figure}

We plot in Figure \ref{fig:error_shap} the sum of the Shapley effects estimation error $\sum_{i=1}^p \left|\widehat{\eta}_i-\eta_i\right|$, with $n=N\, p$ for different values of $N$. We can remark that the sum of the errors seems to be or order $1\slash\sqrt{K}$, which is confirmed by Proposition \ref{prop_shap_generate}.

\section{Application on real data}\label{section_appli}

In this section, we consider a real application of our suggested estimators of block-diagonal covariance matrices and of the Shapley effects, to nuclear data.

\subsection{The Shapley effects with nuclear data}

Uncertainty propagation methods are increasingly being used in nuclear calculation (neutron dosimetry, reactor design, criticality assessments, etc.) to deduce the accuracy of safety parameters (fast fluence, reactivity coefficients, criticality, etc.) and to establish safety margins. In fact, the resulting output of a nuclear computer model can be considered with a random portion as far as the inputs are uncertain. In this context, sensitivity analysis evaluates the impact of input uncertainties in terms of their relative contributions to the uncertainty in the output. Therefore, it helps to prioritize efforts for uncertainty reduction, improving the quality of the data.

Of particular interest for us, the uncertain inputs, in nuclear applications, tend to be correlated because of the measurement processes and the different calculations made to obtain the variables of interest from the observable quantities. This is why the Shapley effects are particularly convenient as sensitivity indices in nuclear uncertainty quantification. Moreover, these uncertain inputs are easily modeled as a Gaussian vector and the output is often modeled as a linear function of the inputs \cite{kawano_evaluation_2006}. Thus, the Shapley effects can be computed or estimated, as it is described in Section \ref{section_shap}.

\subsection{Details of the nuclear data}

In this application, the output is the neutron flux $Y$ which is a quantity of interest in nuclear studies. For example, it can be calculated to evaluate the vessel neutron irradiation which is in fact one of the limiting factors for pressurized water reactor (PWR) lifetime. The quality of radiation damage prediction depends in part on the calculation of the fast neutron flux for energy larger than $1 MeV=10^6 eV$ ($eV$ means electron-volt). In that sense, a lack of knowledge on the fast neutron flux will require larger safety margins on the plant lifetime affecting operating conditions and the cost of nuclear installations. To make correct decisions when designing the plant lifetime and on safety margins for PWRs, it is therefore essential to assess the uncertainty in vessel flux calculations.

One of the major sources of uncertainties in fast flux calculations are the cross sections which are used to characterise the probability that a neutron will interact with a nucleus and are the inputs $X$ of our model. They are expressed in barn, where 1 $barn = 10^{-4} cm^2$. The values of the cross sections and their uncertainties are provided by international libraries as the American Library ENDF/B-VII \cite{ENDF}, the European library JEFF-3 \cite{JEFF}, and the Japan Library JENDL-4 \cite{JENDL}. Using the standardized format,  each cross section is defined for an isotope $iso$ of the target nuclei, an energy level $E$ of the target nuclei and a reaction number $mt$ (see \cite{ENDF} for more informations on $mt$ numbers).

We assume that if $(iso,mt)\neq (iso',mt')$, then, $X_{(iso,mt,E)} \ind X_{(iso',mt',E')}$ for any $E,E'$. Thus, the covariance $\Sigma$ is block-diagonal, where each block corresponds to a value of $(iso,mt)$. Here, we have 292 input variables divided in 50 groups of size between 2 and 18. Using reference data, \cite{clouvel_quantification_2019} has shown that the perturbation of the cross sections of the $^{56}$Fe, $^1$H, $^{16}$O isotopes are linearly related to the perturbation of the flux:
\begin{equation}
Y = \sum_{j\in  \{(iso, mt,E) \}} \beta_j X_j.
\end{equation}
Thus, if $\Sigma$ and $\beta$ are given, the Shapley effects are easily computable by Equations \eqref{eq_shap_groups} and \eqref{eq_VB}. 
We show the values of the Shapley effects in Figure  \ref{fig:nucleaire_ref}.

\begin{figure}
    \centering
    \includegraphics[scale=0.4]{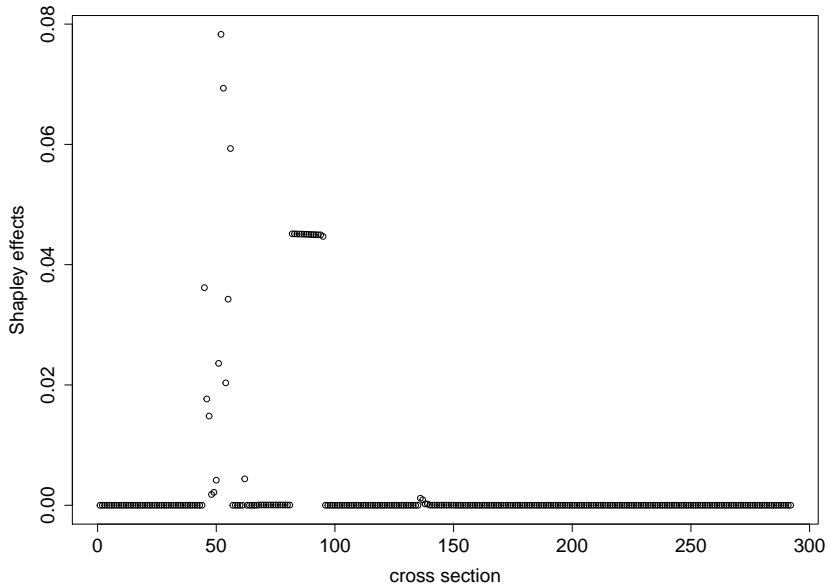}
    \caption{Shapley effects of all the cross sections.}
    \label{fig:nucleaire_ref}
\end{figure}

We can remark that almost all the Shapley effects are close to 0. Now, we plot all the Shapley effects that are larger than $1\%$ on Figure \ref{fig:nucleaire_1} with the names of the corresponding cross sections. For example, "Fe56$\_$S4$\_$950050" means the cross section for the isotope $^{56}Fe$, the reaction scattering 4 and a level of energy larger than $950050 eV$ (and smaller than 1353400$eV$).

\begin{figure}
    \centering
    \includegraphics[scale=0.4]{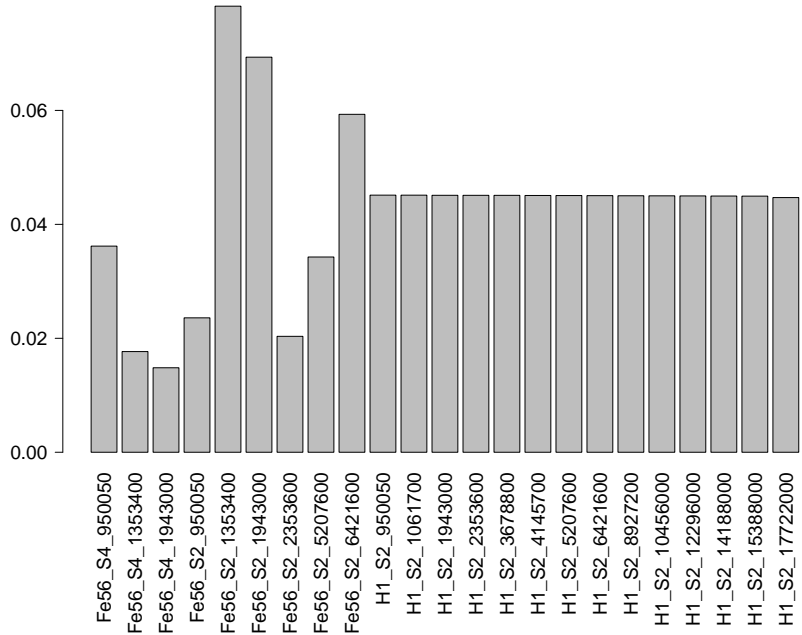}
    \caption{Shapley effects larger than $1\%$.}
    \label{fig:nucleaire_1}
\end{figure}

We remark than only 23 cross sections have a Shapley effect larger than $1\%$. The latter are associated with the lower energies (around 1 to 6 $MeV$). Moreover, they all come from three different groups of $(iso,mt)$: ($^{56}$Fe, scattering 4),  ($^{56}$Fe, scattering 2) and ($^{1}$H, scattering 2).

In fact, in PWR reactors, fission mostly results from the absorption of slow neutrons by nuclei of high atomic number as the uranium $^{235}$U and the plutonium $^{239}$Pu  which are the main fissile isotopes. The nucleus splits into two lighter nuclei, called fission fragments
and often produces an average of 2.5 neutrons with an energy of about 1 $MeV$ or more. Figure \ref{fig:spectrum_U235_P239} illustrates the  neutron fission spectrum which defines the probability for a neutron to be emitted in the energy group $g$ by the isotope $i$. One can note that there are more neutrons produced in the first energy groups between 1 to 6 $MeV$. In that sense, those neutrons have a larger potential to interact with matter. 

\begin{figure}[!htbp]
\centering
\includegraphics[scale=1.1]{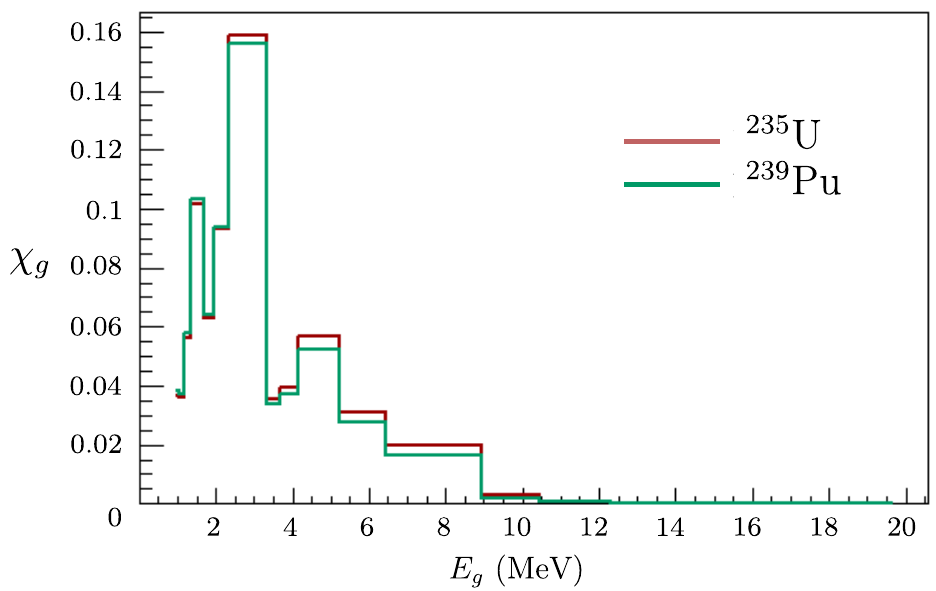}
\caption{$^{235}$U and $^{239}$Pu multigroup fission spectra.}\label{fig:spectrum_U235_P239}
\end{figure}

Moreover, most of the fast neutrons which escape from the core are scattered back by the reflector (essentially comprised of $^{56}$Fe) and slowed down by water
(hydrogen and oxygen), which acts as a moderator, until they reach thermal equilibrium. The accuracy of the neutron flux received by the reflector is closely linked to the precision of the scattering cross sections.

On the other hand, we can notice that all the Shapley effects of the cross sections from $(^{1}$H, scattering 2) are close, and that comes from the fact that the correlations between these different levels of energy are close to $1$ in this group.  

Thus, when the true parameters $\Sigma$ and $\beta$ are known, we can compute the corresponding Shapley effects of the uncertainties of the cross sections on the uncertainty of the neutron flux $Y$. Furthermore, the interpretation of these Shapley effects is insightful and consistent with the available expert knowledge.

\subsection{Estimation of the Shapley effects}

In order to assess the efficiency of our suggested estimation procedures of the Shapley effects, we now assume that the true covariance matrix $\Sigma$ is unknown and that we observe an i.i.d. sample $(X^{(l)})_{l\in [1:n]}$ with distribution $\mathcal{N}(\mu,\Sigma)$ (with $\mu$ unknown). We assume that the maximal group size is known to be smaller or equal to $20$ and that the vector $\beta$ is known. Then, we estimate the block-diagonal structure by the block-diagonal structure $\widehat{B}$ that maximizes the penalized likelihood $\Phi$ among all the block-diagonal structures obtained by thresholding the empirical correlation matrix from its largest value to the smallest value such that the maximal size of the blocks is smaller or equal to $20$. Thus, our estimator $\widehat{B}$ is a mix of the estimators $\widehat{B}_{\widehat{C}}$ and $\widehat{B}_{s}$ detailed in Section \ref{section_conv_cost}.

We plot the Frobenius error of the estimated covariance matrix and the sum of the absolute values of the errors of the estimated Shapley effects for different values of $y= p\slash n$ in Figure \ref{fig:Shapley_nucleaire_estim}, where $p=292$.

\begin{figure}
    \centering
    \includegraphics[scale=0.4]{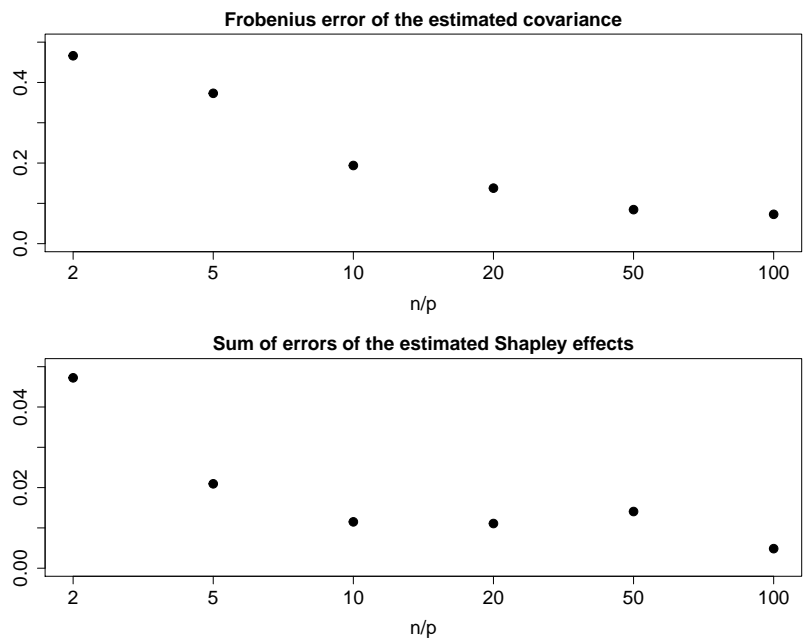}
    \caption{Errors of the estimated covariance matrix and the corresponding Shapley effects for different values of $\frac{n}{p}$.}
    \label{fig:Shapley_nucleaire_estim}
\end{figure}

We can remark that the errors decrease globally when the value of $\frac{n}{p}$ increases. The larger value of the sum of the errors of the estimated Shapley effects for $n\slash p=50$ is due to the randomness of the estimated Shapley effects. Note that, even when $n=2p$, the sum of the errors of the Shapley effects is less than $0.05$ (recall that, in comparison, the sum of the Shapley effects is $1$).
We plot in Figure \ref{fig:estim_1} the estimated Shapley effects that are larger than $1\%$ with $n\slash p=2$. Remark that these estimated values are similar to the true ones displayed in Figure \ref{fig:nucleaire_1} and the physical interpretation is the same.

\begin{figure}
    \centering
    \includegraphics[scale=0.4]{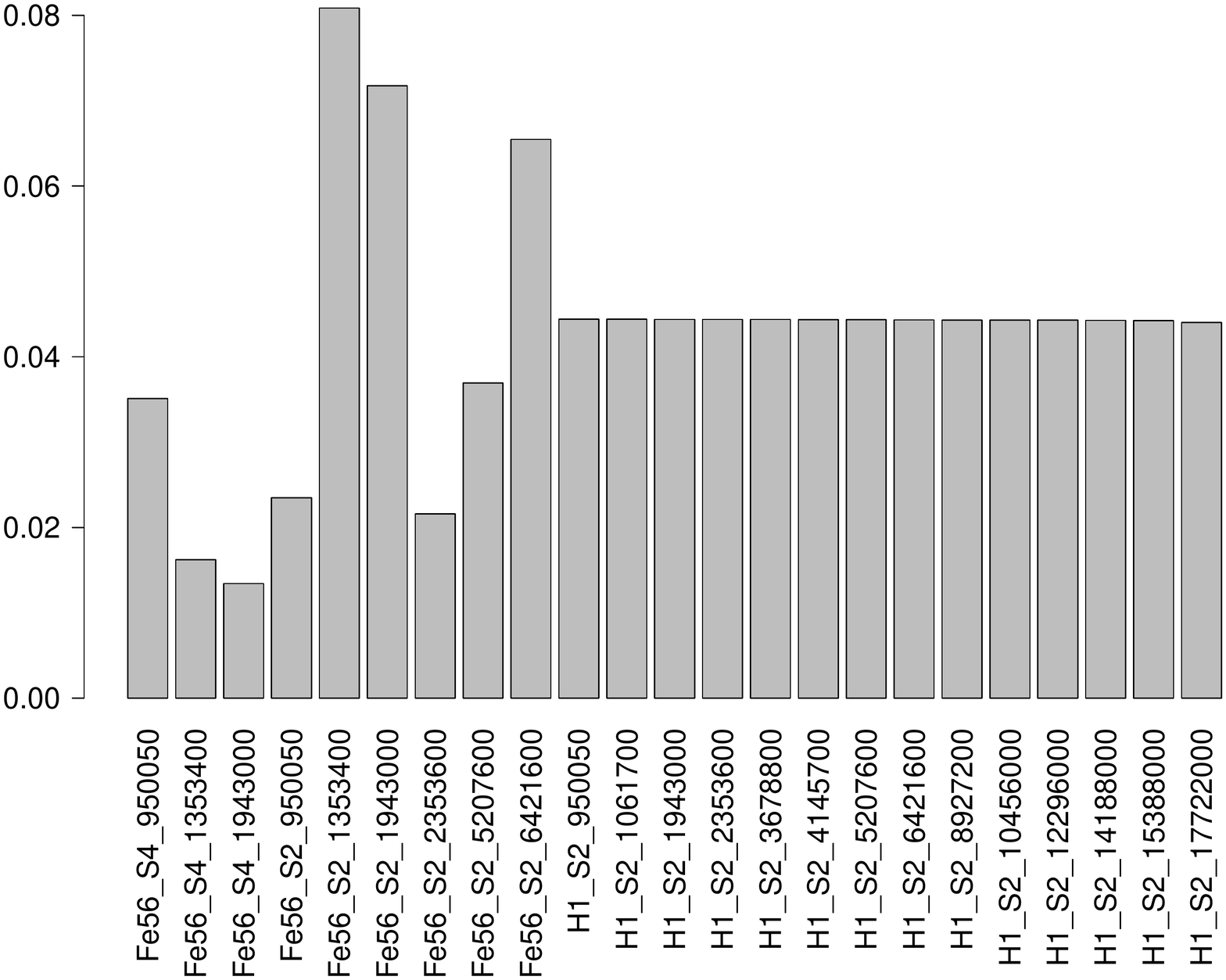}
    \caption{Shapley effects that are larger that $1\%$ estimated with $n\slash p=2$.}
    \label{fig:estim_1}
\end{figure}

\bigskip

In conclusion, we implemented an estimator of the block-diagonal covariance matrix originating from nuclear data when we only observe an i.i.d. sample of the inputs. Then, the derived estimated Shapley effects are shown to be very close to the true Shapley effects, that quantify the impact of the uncertainties of cross sections on the uncertainty on the neutron flux. When the sample size $n$ is equal to $2p$, the physical conclusions are the same as when the true covariance matrix is known.

\section{Conclusion}\label{section_conclu}
In this work, we suggest an estimator of a block-diagonal covariance matrix for Gaussian data. We prove that in high dimension, this estimator converges to the same block-diagonal structure with complexity in $O(p^2)$. For fixed dimension, we also prove the asymptotic efficiency of this estimator, that performs asymptotically as well as as if the true block-diagonal structure were known. Then, we deduce convergent estimators of the Shapley effects in high dimension for Gaussian linear models. These estimators are still available for thousands input variables, as long as the maximal block is not too large. Moreover, we prove the convergence of the Shapley effects estimators when the observations of the output are noisy and so the parameter $\beta$ is estimated. Finally, we applied these estimator on real nuclear data.

In future works, it would be interesting to treat the higher dimension setting when $p\slash n$ goes to $+\infty$.

\section*{Acknowledgements}
We acknowledge the financial support of the Cross-Disciplinary Program on Numerical Simulation of CEA, the French Alternative Energies and Atomic Energy Commission. We would like to thank BPI France for co-financing this work, as part of the PIA (Programme d'Investissements d'Avenir) - Grand D\'{e}fi du Num\'{e}rique 2, supporting the PROBANT project. We thank Vincent Prost for his help.

\bibliographystyle{acm}
\bibliography{biblio}

\section*{Appendix}

\textbf{Notation}

We will write $C_{\sup}$ for a generic non-negative finite constant (depending only on $\lambda_{\inf}$, $\lambda_{\sup}$ and $m$ in Conditions 2 and 3). The actual value of $C_{sup}$ is of no interest and can change in the same sequence of equations. Similarly, we will write $C_{\inf}$ for a generic strictly positive constant.

If $B, B' \in \mathcal{P}_p$, and $(i,j)\in [1:p]^2$, we will wite $(i,j) \in B \setminus B'$ if $(i,j)\in B$ and $(i,j) \notin B'$, that is, if $i$ and $j$ are in the same group with the partition $B$ and are in different groups with the partition $B'$.

If $B, B' \in \mathcal{P}_p$, we define $B\cap B'$ as the maximal partition $B''$ such that $B''\leq B$ and $B''\leq B'$.

If $\Gamma \in \mathcal{M}_p(\R)$ (the set of the matrices of dimension $p \times p$), and if $u,v\subset [1:p]$, we define $\Gamma_{u,v}:=(\Gamma_{i,j})_{i\in u, j\in v}$ and $\Gamma_u:=\Gamma_{u,u}$. 

Recall that $\vecc:\mathcal{M}_p(\R) \rightarrow \R^{p^2}$ is defined by $(\vecc(M))_{p(j-1)+i}:=M_{i,j}$.

If $\Gamma \in S_p(\R)$ (the set of the symmetric positive definite matrices) and $i\in [1:p]$, let $\phi_i(M)$ be the $i$-th largest eigenvalue of $M$. We also write $\lambda_{\max}(M)$ (resp. $\lambda_{\min}(M)$) for the largest (resp. smallest) eigenvalue of $M$.

We define $\widehat{\Sigma}:= \frac{1}{n-1}\sum_{l=1}^l(X^{(l)}-\overline{X})(X^{(l)}-\overline{X})^T=\frac{n}{n-1}S$, the unbiased empirical estimator of $\Sigma$. Let $(\widehat{\sigma}_{ij})_{i,j\leq p}$ be the coefficients of $\widehat{\Sigma}$ and $(s_{ij})_{i,j\leq p}$ be the coefficients of $S$.

Recall that when Condition 4 does not hold, we need to define $B({\alpha})$ as the partition given by thresholding $\Sigma$ by $n^{-\alpha}$. We also define $K(\alpha):=|B(\alpha)|$ and write $B(\alpha)=\{B_1(\alpha),B_2(\alpha),...B_{K(\alpha)}(\alpha) \}$.
\bigskip

\textbf{Proof of Proposition \ref{prop_min_phi}}

\begin{proof}

Let us write
$$
\overline{S_p^{++}(\R,B)}:=\bigsqcup_{B'\leq B} S_p^{++}(\R,B')=\{\Gamma \in S_p^{++},\;\Gamma=\Gamma_B\},
$$
which is the closure of $S_p^{++}(\R,B)$ in $S_p^{++}(\R)$.

First, let us show that, for all $B$, $S_B$ is the minimum of $\Gamma \mapsto l_\Gamma$ on $\overline{S_p^{++}(\R,B)}$. If $\Gamma_B \in \overline{S_p^{++}(\R,B)}$, we have
\begin{eqnarray*}
p\left(l_{\Gamma_B}-l_{S_B} \right)
&=&-\log(|\Gamma_B^{-1}|)+\Tr(\Gamma_B^{-1}S) +\log(|S_B^{-1}|)-\Tr(S_B^{-1}S)\\
&=& - \log\left( \left| \Gamma_B^{-1} S_B\right| \right) +\Tr \left( \Gamma_B^{-1} S_B \right)-p \\
&=&\sum_{i=1}^p\left \{ - \log\left(\phi_i \left[ \Gamma_B^{-1\slash 2} S_B\Gamma_B^{-1\slash 2}\right] \right) +\phi_i \left( \Gamma_B^{-1\slash 2} S_B \Gamma_B^{-1\slash 2} \right)-1 \right\}.
\end{eqnarray*}
The function $f:\R_+^*\rightarrow \R$ defined by $f(t):=-\log(t)+t-1$ has an unique minimum at $1$. Thus, the function $g:S_p^{++}\rightarrow \R$ defined by $g(M):=\sum_{i=1}^p- \log\left(\phi_i \left[ M\right] \right) +\phi_i \left( M \right)-1$ has an unique minimum at $I_p$. Thus
$\Gamma_B \in \overline{S_p^{++}(\R,B)} \mapsto l_{\Gamma_B}-l_{S_B}$ has an unique minimum at $\Gamma_B=S_B$.

Now, the penalisation term is constant on each $S_p^{++}(\R,B)$. Thus $\Phi$ has a global minimum (not necessary unique) at $S_B$, for some $B \in \mathcal{P}_p$. 
\end{proof}
\bigskip

\textbf{Notation for Section \ref{sec_conv_high} }\\
Here and in all the proofs of Section \ref{sec_conv_high}, we assume Conditions 1 to 3 of Section \ref{sec_assum}.

In the following, we introduce some notation.

We know that
$$
\sum_{k=1}^{n}(X^{(k)}-\overline{X})(X^{(k)}-\overline{X})^T\sim \mathcal{W}(n-1,\Sigma),
$$
where $\mathcal{W}(n-1,\Sigma)$ is the Wishart distribution with parameter $n-1$ and $\Sigma$ \cite{gupta1999matrix}. Thus, if we write $(M^{(k)})_{k}$ i.i.d. with distribution $\mathcal{N}(0,\Sigma)$, we have
$$
\widehat{\Sigma}:=\frac{1}{n-1}\sum_{k=1}^{n}(X^{(k)}-\overline{X})(X^{(k)}-\overline{X})^T\sim \frac{1}{n-1}\sum_{k=1}^{n-1} M^{(k)} M^{(k)T}.
$$

\begin{lm}\label{lm_eingenvalue}
For all $C_{\inf}>0$,
$$
\PP\left( \lambda_{\max}(S)> \lambda_{\sup}(1+\sqrt{y})^2+ C_{\inf}  \right)\longrightarrow 0,
$$
$$
\PP\left( \lambda_{\max}(S)< \lambda_{\inf}(1+\sqrt{y})^2- C_{\inf}  \right)\longrightarrow 0,
$$
$$
\PP\left( \lambda_{\min}(S)< \lambda_{\inf}(1-\sqrt{y})^2- C_{\inf}  \right)\longrightarrow 0,
$$
and
$$
\PP\left( \lambda_{\min}(S)> \lambda_{\sup}(1-\sqrt{y})^2+ C_{\inf}  \right)\longrightarrow 0.
$$
Moreover, that holds also for $\widehat{\Sigma}$ instead of $S$.
\end{lm}

\begin{proof}
Let $(A^{(k)})_k$ i.i.d. with distribution $\mathcal{N}(0,I_p)$.
Using the result in \cite{smallest_1985_silverstein} which states that
$$
\lambda_{\max}\left(\frac{1}{n-1}\sum_{k=1}^{n-1}A^{(k)}A^{(k)T}\right)\overset{a.s.}{\underset{n\rightarrow +\infty}{\longrightarrow}}(1+\sqrt{y})^2,
$$
we have,
\begin{eqnarray*}
\lambda_{\max}(S) &=& \frac{n}{n-1} \lambda_{\max}\left( \frac{1}{n-1} \sum_{k=1}^{n-1}M^{(k)}M^{(k)T} \right)\\
&\leq &\frac{n}{n-1} \lambda_{\sup}\lambda_{\max}\left(\frac{1}{n-1}\sum_{k=1}^{n-1} A^{(k)}A^{(k)T} \right)\\
&=&\lambda_{\sup}(1+\sqrt{y})^2+o_p(1),
\end{eqnarray*}
and
$$
\lambda_{\max}(S)\geq \frac{n}{n-1} \lambda_{\inf}\lambda_{\max}\left(\frac{1}{n-1}\sum_{k=1}^{n-1} A^{(k)}A^{(k)T} \right)=\lambda_{\inf}(1+\sqrt{y})^2+o_p(1).
$$
Thus,
$$
\lambda_{\inf}(1+\sqrt{y})^2+o_p(1)\leq \lambda_{\max}(S) \leq \lambda_{\sup}(1+\sqrt{y})^2+o_p(1).
$$

The proof is the same for $\lambda_{\min}$.
\end{proof}

We also verify the assumptions of Bernstein's inequality (see for example Theorem 2.8.1 in \cite{vershynin_high-dimensional_2018}).
For all $i,j,k$, let
\begin{equation}\label{eq_Z}
Z_{ij}^{(k)}:=M_i^{(k)}M_j^{(k)}-\sigma_{ij}.
\end{equation} 
The random-variables $(Z_{ij}^{(k)})_k$ are independent, mean zero, sub-exponential and we have $\|Z_{ik}^{(k)}\|_{\psi_1}\leq \|M_i\|_{\psi_2}\|M_j\|_{\psi_2}\leq C_{\sup} \sqrt{\sigma_{ii}\sigma_{jj}}\leq C_{\sup}$, where $\|.\|_{\psi_1}$ is the sub-exponential norm (for example, see Definition 2.7.5 in \cite{vershynin_high-dimensional_2018}). So, we can use Bernstein's inequality with $(Z_{ij}^{(k)})_k$: there exists $C_{\inf}$ such that, for all $\varepsilon>0$ and $n\in \N$,
$$
\max_{i,j\in [1:p]}\PP\left( \left| \sum_{k=1}^{n}Z_{ij}^{(k)} \right| \geq n\varepsilon \right) \leq 2 \exp \left(-C_{\inf}n\min(\varepsilon,\varepsilon^2) \right).
$$

\bigskip

\textbf{Proof of Proposition \ref{prop_conclu} }\\

In this proof, we assume that Conditions 1 to 4 are satisfied. We first show several Lemmas.

\begin{lm}\label{lm_vp_blocs}
For all symmetric positive definite $\Gamma$ and for all $B\in \mathcal{P}_p$, if we write $\Delta=\Gamma-\Gamma_B$, we have:
\begin{itemize}
\item $v\mapsto \lambda_{\min}(\Gamma_B+v\Delta)$ decreases and so $\min_{v \in [0,1]} \lambda_{\min}(\Gamma_B+v\Delta)=\lambda_{\min}(\Gamma)$.
\item $v\mapsto \lambda_{\max}(\Gamma_B+v\Delta)$ increases and so $\max_{v \in [0,1]} \lambda_{\max}(\Gamma_B+v\Delta)=\lambda_{\max}(\Gamma)$.
\end{itemize}
\end{lm}

\begin{proof}
Let us show that $v\mapsto \lambda_{\max}(\Gamma_B+v\Delta)$ increases (the proof if the same for $\lambda_{\min}$).

For all $v\in [0,1]$, let $\Gamma_v=\Gamma_B+v\Delta$, $\lambda_v=\lambda_{\max}(\Gamma_v)$ and $e_v$ an unit eigenvector of $\Gamma_v$ associated to $\lambda_v$. Let $v,v'\in [0,1], \;v<v'$. Thus
$$
\lambda_{v'}=\max_{u,\;\|u\|=1} u^T (\Gamma_B+v'\Delta) u\geq e_v^T(\Gamma_B+v'\Delta)e_v= \lambda_v+(v-v') e_v^T\Delta e_v.
$$
If we show that $e_v^T\Delta e_v\geq 0$, we proved that $v\mapsto\lambda_v$ increases.
First, assume that $v=0$. If we write $B_k$ the group of the largest eigenvalue of $\Gamma_B$, then $(e_0)_i$ is equal to zero for all $i \notin B_k$, so 
$(e_0^T\Delta)_j$ is equal to zero for all $j \in B_k$, and so $e_0^T\Delta e_0$ is equal to zero.

Assume now that $v>0$ and let us show that $e_v^T\Delta e_v\geq 0$ by contradiction. Assume that $e_v^T\Delta e_v<0$. Then
\begin{eqnarray*}
e_v^T(\Gamma_B+v\Delta)e_v  <  e_v^T \Gamma_B e_v \leq e_0^T\Gamma_B e_0.
\end{eqnarray*}
Furthermore, we have seen that $e_0^T \Delta e_0=0$. Thus, we have
$$
e_v^T(\Gamma_B+v\Delta)e_v <e_0^T(\Gamma_B+v\Delta)e_0,
$$
that is in contradiction with $e_v\in \argmax_{u,\;\|u\|=1} u^T (\Gamma_B+v\Delta)u$.
\end{proof}

In the following, let $\Delta_{B,B'}:=S_{B}-S_{B'}$ for all $B,B'\in \mathcal{P}_p$.

\begin{lm}\label{lm_etape1}
For all $B\in \mathcal{P}_p$, we have 
\begin{equation}\label{eq1_etape1}
l_{S_{B\cap B^*}}-l_{S_B}\leq \frac{1}{2 \lambda_{\min}(S)}\frac{1}{p}\| \Delta_{B,B\cap B^*}\|_F^2.
\end{equation}
Moreover, for all $B<B^*$, we have
\begin{equation}\label{eq2_etape1}
l_{S_{B}}-l_{S_{B^*}}\geq \frac{1}{2 \lambda_{\max} (S_{B^*})}\frac{1}{p}\| \Delta_{B^*,B}\|_F^2.
\end{equation}
\end{lm}

\begin{proof}
First, we prove Equation \eqref{eq1_etape1}.
Doing the Taylor expansion of $t\mapsto\log \circ \det\left(S_{B\cap B^*}+t\Delta_{B,B\cap B^*}\right)$ and using the integral form of the remainder (as Equation (9) of \cite{rothman_sparse_2008} or in \cite{lam_sparsistency_2009}), we have 
\begin{eqnarray*}
& &  p\left(l_{S_{B\cap B^*}}-l_{S_{B}}\right)\\
&=& \log(|S_{B\cap B^*}|)-\log(|S_B|)\\
&=& -\Tr(S_{B\cap B^*} \Delta_{B,B\cap B^*})  +\vecc(\Delta_{B,B\cap B^*})^T \Bigg[\int_0^1(S_{B\cap B^*}+v\Delta_{B,B\cap B^*})^{-1}  \\
& &  \otimes(S_{B\cap B^*}+v\Delta_{B,B\cap B^*})^{-1}(1-v)dv\Bigg]\vecc(\Delta_{B,B\cap B^*}),
\end{eqnarray*}
where $\otimes$ is the Kronecker product.
The trace is equal to zero.  Now,
\begin{eqnarray*}
p\left(l_{S_{B\cap B^*}}-l_{S_{B}}\right) & \leq &  1\slash 2 \max_{0\leq v \leq 1} \lambda_{\max}^2[(S_{B\cap B^*}+v\Delta_{B,B\cap B^*})^{-1}]\|\vecc(\Delta_{B,B\cap B^*})\|^2\\
&=& 1\slash 2 \max_{0\leq v \leq 1} \lambda_{\min}^{-2}(S_{B\cap B^*}+v\Delta_{B,B\cap B^*})\|\vecc(\Delta_{B,B\cap B^*})\|^2\\
&= & \frac{1}{2\min_v(\lambda_{\min}(S_{B\cap B^*}+v\Delta_{B,B\cap B^*}))^{2}}\|\vecc(\Delta_{B,B\cap B^*})\|^2\\
&= & \frac{1}{2\lambda_{\min}(S_B)^{2}}\|\vecc(\Delta_{B,B\cap B^*})\|^2\\
& \leq & \frac{1}{2\lambda_{\min}(S)^{2}}\|\vecc(\Delta_{B,B\cap B^*})\|^2,
\end{eqnarray*}
using Lemma \ref{lm_vp_blocs} for the two last steps.

Now, we prove Equation \eqref{eq2_etape1} similarly. We have, using Lemma \ref{lm_vp_blocs},
\begin{eqnarray*}
  p\left(l_{S_{B}}-l_{S_{B^*}}\right) &=& -\Tr(S_{B} \Delta_{B^*,B})  +\vecc(\Delta_{B^*,B})^T \Bigg[\int_0^1(S_{B}+v\Delta_{B^*,B})^{-1}  \\
& &  \otimes(S_{B}+v\Delta_{B^*,B})^{-1}(1-v)dv\Bigg]\vecc(\Delta_{B^*,B})\\
& \geq &  1\slash 2 \min_{0\leq v \leq 1} \lambda_{\min}^2[(S_{B}+v\Delta_{B^*,B})^{-1}]\|\vecc(\Delta_{B^*,B})\|^2\\
&=& 1\slash 2 \min_{0\leq v \leq 1} \lambda_{\max}^{-2}(S_{B}+v\Delta_{B^*,B})\|\vecc(\Delta_{B^*,B})\|^2\\
&= & \frac{1}{2\max_v(\lambda_{\max}(S_{B}+v\Delta_{B^*,B}))^{2}}\|\vecc(\Delta_{B^*,B})\|^2\\
&= & \frac{1}{2\lambda_{\max}(S_B)^{2}}\|\vecc(\Delta_{B^*,B})\|^2\\
& \geq & \frac{1}{2\lambda_{\max}(S)^{2}}\|\vecc(\Delta_{B^*,B})\|^2.
\end{eqnarray*}
\end{proof}

\begin{lm}\label{lm_B_notleq}
$$
\PP\left( \max_{B\not \leq B^*} \Phi(B\cap B^*)-\Phi(B)\geq 0\right) \longrightarrow 0.
$$
\end{lm}

\begin{proof}
Using Lemma \ref{lm_etape1}, we have
\begin{eqnarray*}
& & \PP\left( \max_{B\not \leq B^*} \Phi(B\cap B^*)-\Phi(B)\geq 0\right)\\
& \leq & \PP\left( \max_{B\not \leq B^*} \left[l_{S_{B\cap B^*}}-l_{S_B}-\frac{1}{pn^\delta}\left(\pen(B) - \pen (B\cap B^*) \right) \right]\geq 0\right)\\
& \leq & \PP\left( \max_{B\not \leq B^*} \left[ \frac{1}{2 \lambda_{\min}(S)^2}\| \Delta_{B,B\cap B^*}\|_F^2-\frac{1}{n^\delta}\left(\pen(B) - \pen (B\cap B^*) \right) \right]\geq 0\right)\\
& \leq & \PP\left( \lambda_{\min}(S)\leq  \frac{1}{2}\lambda_{\inf}(1-\sqrt{y})^2 \right)\\
& & + \PP\left( \max_{B\not \leq B^*} \left[ \frac{1}{ (1-\sqrt{y})^4\lambda_{\inf}^2}\| \Delta_{B,B\cap B^*}\|_F^2-\frac{1}{n^\delta}\left(\pen(B) - \pen (B\cap B^*) \right) \right]\geq 0\right).
\end{eqnarray*}
We show that the two terms go to $0$.
The first term goes to $0$ with Lemma \ref{lm_eingenvalue}. For the second term, we have
\begin{eqnarray*}
&& \PP\left( \max_{B\not \leq B^*} \left[ \frac{1}{ (1-\sqrt{y})^4\lambda_{\inf}^2}\| \Delta_{B,B\cap B^*}\|_F^2-\frac{1}{n^\delta}\left(\pen(B) - \pen (B\cap B^*) \right) \right]\geq 0\right)\\
& =&  \PP\left( \max_{B\not \leq B^*} \left[ \sum_{(i,j)\in B \setminus B^*} \left\{\frac{1}{ (1-\sqrt{y})^4 \lambda_{\inf}^2} s_{ij}^2-\frac{1}{n^\delta} \right\} \right]\geq 0\right)\\
& \leq & \PP\left( \exists (i,j)\notin B^*,\;\frac{1}{ (1-\sqrt{y})^4\lambda_{\inf}^2}s_{ij}^2-\frac{1}{n^\delta} \geq 0\right)\\
& \leq & \PP\left( \exists (i,j)\notin B^*,\;\frac{2}{ (1-\sqrt{y})^4\lambda_{\inf}^2}\widehat{\sigma}_{ij}^2-\frac{1}{n^\delta} \geq 0\right)\\
& \leq & p^2\max_{ (i,j)\notin B^*}\PP\left(  \frac{2}{ (1-\sqrt{y})^4\lambda_{\inf}^2}\widehat{\sigma}_{ij}^2- \frac{1}{ n^\delta}\geq 0 \right)\\
& \leq & p^2\max_{ (i,j)\notin B^*}\PP\left(  \frac{\sqrt{2}}{ (1-\sqrt{y})^2\lambda_{\inf}}|\widehat{\sigma}_{ij}|\geq \frac{1}{ n^{\delta\slash 2}} \right)\\
& \leq & p^2\max_{ (i,j)\notin B^*} \PP \left( \left|\sum_{k=1}^{n-1} Z_{ij}^{(k)} \right|\geq \frac{(1-\sqrt{y})^2}{ \sqrt{2}} \lambda_{\inf} n^{1-\delta\slash2} \ \right)\\
&\leq &  2p^2 \exp\left( -C_{\inf} n^{1-\delta}\right) \longrightarrow 0,
\end{eqnarray*}
using Bernstein's inequality, where $Z_{ij}^{(k)}$ is defined in Equation \eqref{eq_Z}. That concludes the proof.
\end{proof}

\begin{lm}\label{lm_B<}
$$
\PP\left( \max_{B<B^*} \Phi(B^*)-\Phi(B)\geq 0\right)\longrightarrow 0. 
$$
\end{lm}

\begin{proof}
Using Lemma \ref{lm_etape1}, we have
\begin{eqnarray*}
& & \PP\left( \max_{B<B^*} \Phi(B^*)-\Phi(B)\geq 0\right)\\
& \leq & \PP\left( \min_{B<B^*} \left[l_{S_{ B}}-l_{S_{B^*}}-\frac{1}{pn^\delta}\left(\pen(B^*) - \pen (B) \right) \right]\leq 0\right)\\
& \leq & \PP\left(  \min_{B<B^*} \left[ \frac{1}{2 \lambda_{\max}(S_{B^*})^2}\| \Delta_{B^*, B}\|_F^2-\frac{1}{n^\delta}\left(\pen(B^*) - \pen (B) \right) \right]\leq 0\right)\\
& \leq & \PP\left( \lambda_{\max}(S_{B})\leq \frac{\lambda_{\inf}(1+\sqrt{y})^2}{2} \right)\\
& & + \PP\left( \min_{B<B^*} \left[ \frac{1}{(1+\sqrt{y})^4\lambda_{\inf}^2}\| \Delta_{B^*, B}\|_F^2-\frac{1}{n^\delta}\left(\pen(B^*) - \pen (B) \right) \right]\leq 0\right).
\end{eqnarray*}
The first term goes to 0 with Lemma \ref{lm_eingenvalue}. The second term is
\begin{eqnarray*}
& & \PP\left( \exists B<B^*,\;\sum_{(i,j)\in B^* \setminus B}\left[ \frac{1}{(1+\sqrt{y})^4\lambda_{\inf}^2}s_{i,j}^2-n^{-\delta}\right] \leq 0\right)\\
& \leq & \PP\left( \exists B<B^*,\;\sum_{(i,j)\in B^* \setminus B}\left[ \frac{1}{\lambda_{\inf}^2}s_{i,j}^2-n^{-\delta}\right] \leq 0\right)\\
& \leq & \PP\left( \exists k \in [1:K],\;\emptyset \varsubsetneq B_1 \varsubsetneq B_k^*,\;\sum_{i\in B_1, \;j\in B_k^*\setminus B_1} \left[  \frac{1}{\lambda_{\inf}^2}s_{ij}^2-n^{-\delta} \right] \leq 0\right)\\
& \leq & p 2^m \max_{\substack{k\in [1:K],\\ \emptyset \varsubsetneq B_1\varsubsetneq B_k^*}}  \PP\left( \sum_{i\in B_1, \;j\in B_k^*\setminus B_1} \left[  \frac{1}{\lambda_{\inf}^2}s_{ij}^2-n^{-\delta}\right] \leq 0\right)\\
& \leq & p 2^m \max_{\substack{k\in [1:K],\\ \emptyset \varsubsetneq B_1\varsubsetneq B_k^*}}  \PP\left( \sum_{i\in B_1, \;j\in B_k^*\setminus B_1} \left[  \frac{1}{2\lambda_{\inf}^2}\widehat{\sigma}_{ij}^2-n^{-\delta} \right] \leq 0\right).
\end{eqnarray*}
Now, for all $k\in [1;K]$ and for all $\emptyset \varsubsetneq B_1\varsubsetneq B_k^*$, let $(i^*,j^*)\in \argmax_{i\in B_1,j \in B_k^*\setminus B_1}|\sigma_{ij}|$ (with an implicit dependence on $k$ and $B_1$). Remark that
$$
\frac{1}{2\lambda_{\inf}^2}\widehat{\sigma}_{i^*j^*}^2 \geq m^2 n^{-\delta} \Longrightarrow  \sum_{i\in B_1, \;j\in B_k^*\setminus B_1}\left(\frac{1}{2\lambda_{\inf}^2}\widehat{\sigma}_{ij}^2-n^{-\delta}\right)\geq 0.
$$
Thus,
\begin{eqnarray*}
& & \PP\left( \exists B<B^*,\;\sum_{(i,j)\in B^* \setminus B} \left[  \frac{1}{\lambda_{\inf}^2}s_{i,j}^2-n^{-\delta}\right] \leq 0\right)\\
 & \leq & p2^m  \max_{\substack{k\in [1:K],\\ \emptyset \varsubsetneq B_1\varsubsetneq B_k^*}}  \PP \left( \frac{1}{2\lambda_{\inf}^2}\widehat{\sigma}_{i^*j^*}^2 \leq m^2 n^{-\delta}\right)\\
&=&p2^m \max_{\substack{k\in [1:K],\\ \emptyset \varsubsetneq B_1\varsubsetneq B_k^*}} \PP \left( |\widehat{\sigma}_{i^*j^*}| \leq \sqrt{2}\lambda_{\inf} m n^{-\delta\slash 2}\right)\\
& \leq &p2^m \max_{\substack{k\in [1:K],\\ \emptyset \varsubsetneq B_1\varsubsetneq B_k^*}} \PP \left( |\widehat{\sigma}_{i^*j^*}-\sigma_{i^*j^*}| \geq |\sigma_{i^*j^*}|-\sqrt{2}\lambda_{\inf} m n^{-\delta\slash 2}\right)\\
& = & p2^m \max_{\substack{k\in [1:K],\\ \emptyset \varsubsetneq B_1\varsubsetneq B_k^*}} \PP \left( \left|\sum_{k=1}^{n-1} Z_{i^*j^*}^{(k)}\right| \geq n\left[|\sigma_{i^*j^*}|- \sqrt{2}\lambda_{\inf} m n^{-\delta\slash 2}\right] \right)\\
& = & p2^m  \PP \left( \left|\sum_{k=1}^{n-1} Z_{i^*j^*}^{(k)}\right| \geq n \min_{\substack{k\in [1:K],\\ \emptyset \varsubsetneq B_1\varsubsetneq B_k^*}}\left[|\sigma_{i^*j^*}|- \sqrt{2}\lambda_{\inf} m n^{-\delta\slash 2}\right] \right).
\end{eqnarray*}
Now, by Condition 4, we know that $\min_{\substack{k\in [1:K],\\ \emptyset \varsubsetneq B_1\varsubsetneq B_k^*}}|\sigma_{i^*j^*}|\geq a n^{-1\slash 4}$, so, for $n$ large enough,
$$
\min_{\substack{k\in [1:K],\\ \emptyset \varsubsetneq B_1\varsubsetneq B_k^*}} \left[|\sigma_{i^*j^*}|- \sqrt{2}\lambda_{\inf} m n^{-\delta\slash 2}\right]\geq C_{\inf}( n^{- 1 \slash 4}-n^{ - \delta\slash 2})\geq C_{\inf}(\delta)n^{-1\slash 4}.
$$
Thus, by Bernstein's inequality, for $n$ large enough,
\begin{eqnarray*}
& & \PP\left( \exists B<B^*,\;\sum_{(i,j)\in B^* \setminus B} \left[  \frac{1}{(1-\sqrt{y})^2\lambda_{\inf}^2}s_{i,j}^2-n^{-\delta}\right] \leq 0\right)\\
& &  \leq p2^{m+1}\exp\left( - C_{\inf}(\delta)n^{1\slash 2} \right) \longrightarrow 0.
\end{eqnarray*}
\end{proof}

Now, we can prove Proposition \ref{prop_conclu}.
\begin{proof}
We have
\begin{eqnarray*}
\PP\left( \max_{B\neq B^*} \Phi(B^*)-\Phi(B)\geq 0\right)
& \leq & \PP\left( \max_{B< B^*} \Phi(B^*)-\Phi(B)\geq 0\right)\\
& & + \PP\left( \max_{B> B^*} \Phi(B^*)-\Phi(B)\geq 0\right)\\
& & + \PP\left( \max_{B\not \leq B^*,\; B \not \geq B^*} \Phi(B^*)-\Phi(B)\geq 0\right).
\end{eqnarray*}

The two first terms go to 0 tanks to  Lemmas \ref{lm_B_notleq} and \ref{lm_B<}. For the last term, we have
\begin{eqnarray*}
& & \PP\left( \max_{B\not \leq B^*,\; B \not \geq B^*} \Phi(B^*)-\Phi(B)\geq 0\right)\\
& =&  \PP\left( \max_{B\not \leq B^*,\; B \not \geq B^*} \Phi(B^*)-\Phi(B\cap B^*)+\Phi(B \cap B^*)-\Phi(B)\geq 0\right)\\
& \leq & \PP\left( \max_{B\not \leq B^*,\; B \not \geq B^*} \Phi(B^*)-\Phi(B\cap B^*)\geq 0\right)\\
& & + \PP\left( \max_{B\not \leq B^*,\; B \not \geq B^*} \Phi(B \cap B^*)-\Phi(B)\geq 0\right)\\
& \leq &  \PP\left( \max_{B'< B^*} \Phi(B^*)-\Phi(B')\geq 0\right)\\
& & + \PP\left( \max_{B\not \leq B^*} \Phi(B \cap B^*)-\Phi(B)\geq 0\right).
\end{eqnarray*}
These two last terms go to 0 thanks to Lemmas \ref{lm_B_notleq} and \ref{lm_B<}.

\end{proof}
\bigskip

\textbf{Proofs of Proposition \ref{prop_Btot_Ba}}

In this proof, we assume that Conditions 1 to 3 hold.

\begin{lm}\label{lm_etape1_Ba}
For all $B\in \mathcal{P}_p$, we have 
\begin{equation}\label{eq1_etape1_Ba}
l_{S_{B\cap B(\alpha_2)}}-l_{S_B}\leq \frac{1}{2 \lambda_{\min}(S)}\frac{1}{p}\| \Delta_{B,B\cap B(\alpha_2)}\|_F^2.
\end{equation}
Moreover, for all $B<B(\alpha_2)$, we have
\begin{equation}\label{eq2_etape2_Ba}
l_{S_{B}}-l_{S_{B(\alpha_2)}}\geq \frac{1}{2 \lambda_{\max} (S_{B(\alpha_2)})}\frac{1}{p}\| \Delta_{B(\alpha_2),B}\|_F^2.
\end{equation}
\end{lm}

\begin{proof}
Same proof as Lemma \ref{lm_etape1} replacing $B^*$ by $B(\alpha_2)$.
\end{proof}

\begin{lm}\label{lm_B_notleq_Ba}  If $\alpha_2 > \delta\slash 2$, then,
$$
\PP\left( \max_{B\not \leq B(\alpha_2)} \Phi(B\cap B(\alpha_2))-\Phi(B)\geq 0\right) \longrightarrow 0.
$$
\end{lm}

\begin{proof}
Following the proof of Lemma \ref{lm_B_notleq} (and using Lemma \ref{lm_etape1_Ba}), it is enough to prove that the following term goes to 0:
\begin{eqnarray*}
&& \PP\left( \max_{B\not \leq B(\alpha_2)} \left[ \frac{1}{ (1-\sqrt{y})^4\lambda_{\inf}^2}\| \Delta_{B,B\cap B(\alpha_2)}\|_F^2-\frac{1}{n^\delta}\left(\pen(B) - \pen (B\cap B(\alpha_2)) \right) \right]\geq 0\right)\\
& =&  \PP\left( \max_{B\not \leq B(\alpha_2)} \left[ \sum_{(i,j)\in B \setminus B(\alpha_2)} \left\{\frac{1}{ (1-\sqrt{y})^4 \lambda_{\inf}^2} s_{ij}^2-\frac{1}{n^\delta} \right\} \right]\geq 0\right)\\
& \leq & \PP\left( \exists (i,j)\notin B(\alpha_2),\;\frac{1}{ (1-\sqrt{y})^4\lambda_{\inf}^2}s_{ij}^2-\frac{1}{n^\delta} \geq 0\right)\\
& \leq & \PP\left( \exists (i,j)\notin B(\alpha_2),\;\frac{2}{ (1-\sqrt{y})^4\lambda_{\inf}^2}\widehat{\sigma}_{ij}^2-\frac{1}{n^\delta} \geq 0\right)\\
& \leq & p^2\max_{ (i,j)\notin B(\alpha_2)}\PP\left(  \frac{2}{ (1-\sqrt{y})^4\lambda_{\inf}^2}\widehat{\sigma}_{ij}^2- \frac{1}{ n^\delta}\geq 0 \right)\\
& \leq & p^2\max_{ (i,j)\notin B(\alpha_2)}\PP\left(  \frac{\sqrt{2}}{ (1-\sqrt{y})^2\lambda_{\inf}}|\widehat{\sigma}_{ij}|\geq \frac{1}{ n^{\delta\slash 2}} \right)\\
& \leq & p^2\max_{ (i,j)\notin B(\alpha_2)}\PP\left(  \frac{\sqrt{2}}{ (1-\sqrt{y})^2\lambda_{\inf}}|\widehat{\sigma}_{ij}-\sigma_{ij}|\geq \frac{1}{ n^{\delta\slash 2}} -\frac{\sqrt{2}}{ (1-\sqrt{y})^2\lambda_{\inf}}n^{-\alpha_2} \right)\\
&\leq &  2p^2 \exp\left( -C_{\inf} n^{1-\delta}\right) \longrightarrow 0,
\end{eqnarray*}
using again Bernstein's inequality. That concludes the proof.
\end{proof}

\begin{lm}\label{lm_B<_Ba}
If $\alpha_1 < \delta\slash 2$, then, 
$$
\PP\left( \max_{B<B(\alpha_1)} \Phi(B(\alpha_1))-\Phi(B)\geq 0\right)\longrightarrow 0. 
$$
\end{lm}

\begin{proof}
Following the proof of Lemma \ref{lm_B<}, it suffices to prove that
\begin{eqnarray*}
 p2^m  \PP \left( \left|\sum_{k=1}^{n-1} Z_{i^*j^*}^{(k)}\right| \geq n \min_{\substack{k\in [1:K(\alpha_1)],\\ \emptyset \varsubsetneq B_1\varsubsetneq B_k(\alpha_1)}}\left[|\sigma_{i^*j^*}|- \sqrt{2}\lambda_{\inf} m n^{-\delta\slash 2}\right] \right) \longrightarrow 0.
\end{eqnarray*}
Now, by definition of $B(\alpha_1)$, we know that $\min_{\substack{k\in [1:K(\alpha_1)],\\ \emptyset \varsubsetneq B_1\varsubsetneq B_k(\alpha_1)}}|\sigma_{i^*j^*}|\geq n^{-\alpha_1}$, so, for $n$ large enough,
$$
\min_{\substack{k\in [1:K(\alpha_1)],\\ \emptyset \varsubsetneq B_1\varsubsetneq B_k(\alpha_1)}} \left[|\sigma_{i^*j^*}|- \sqrt{2}\lambda_{\inf} m n^{-\delta\slash 2}\right]\geq C_{\inf}( n^{- \alpha_1}-n^{ - \delta\slash 2})\geq C_{\inf}(\alpha_1,\delta)n^{-\alpha_1}.
$$
Thus, by Bernstein's inequality, for $n$ large enough,
\begin{eqnarray*}
& & \PP\left( \exists B<B(\alpha_1),\;\sum_{(i,j)\in B(\alpha_1) \setminus B} \left[  \frac{1}{(1-\sqrt{y})^2\lambda_{\inf}^2}s_{i,j}^2-n^{-\delta}\right] \leq 0\right)\\
& &  \leq p2^{m+1}\exp\left( - C_{\inf}(\alpha_1,\delta)n^{1-2\alpha_1} \right) \longrightarrow 0.
\end{eqnarray*}
\end{proof}

We can now prove Proposition \ref{prop_Btot_Ba}.
\begin{proof}
We have
\begin{eqnarray*}
 & & \PP \left( \left\{ B (\alpha_1)\not > \widehat{B}_{tot} \leq B(\alpha_2) \right\}^c \right)\\
 &= & \PP \left( \min_{B < B(\alpha_1) \text{ or } B \not \leq B(\alpha_2)} \Phi(B) \leq  \min_{B \not< B(\alpha_1) \text{ and } B  \leq B(\alpha_2)} \Phi(B) \right)\\
 & \leq & \PP \left( \min_{B < B(\alpha_1)} \Phi(B) \leq \min_{B \not< B(\alpha_1) \text{ and } B  \leq B(\alpha_2)} \Phi(B) \right)\\
 && + \PP \left( \min_{ B \not \leq B(\alpha_2) \text{ s.t. }B\cap B(\alpha_2) \not < B(\alpha_1)} \Phi(B) \leq \min_{B \not< B(\alpha_1) \text{ and } B  \leq B(\alpha_2)} \Phi(B) \right)\\
 && + \PP \left( \min_{ B \not \leq B(\alpha_2) \text{ s.t. }B\cap B(\alpha_2) < B(\alpha_1)} \Phi(B) \leq \min_{B \not< B(\alpha_1) \text{ and } B  \leq B(\alpha_2)} \Phi(B) \right).
\end{eqnarray*}
First, 
\begin{eqnarray*}
& & \PP \left( \min_{B < B(\alpha_1)} \Phi(B) \leq \min_{B \not< B(\alpha_1) \text{ and } B  \leq B(\alpha_2)} \Phi(B) \right)\\
& \leq & \PP \left( \min_{B < B(\alpha_1)} \Phi(B) -\Phi(B(\alpha_1)) \leq 0 \right) \longrightarrow 0,
\end{eqnarray*}
from Lemma \ref{lm_B<_Ba}. Secondly,
\begin{eqnarray*}
&&  \PP \left( \min_{ B \not \leq B(\alpha_2) \text{ s.t. }B\cap B(\alpha_2) \not < B(\alpha_1)} \Phi(B) \leq \min_{B \not< B(\alpha_1) \text{ and } B  \leq B(\alpha_2)} \Phi(B) \right)\\
& \leq & \PP\left(\min_{ B \not \leq B(\alpha_2) \text{ s.t. }B\cap B(\alpha_2) \not < B(\alpha_1)} \Phi(B) -\Phi(B\cap B(\alpha_2)) \leq 0\right) \longrightarrow 0,
\end{eqnarray*}
from Lemma \ref{lm_B_notleq_Ba}.
Finally,
\begin{eqnarray*}
& & \PP \left( \min_{ B \not \leq B(\alpha_2) \text{ s.t. }B\cap B(\alpha_2) < B(\alpha_1)} \Phi(B) \leq \min_{B \not< B(\alpha_1) \text{ and } B  \leq B(\alpha_2)} \Phi(B) \right)\\
& \leq &  \PP \left( \min_{ B \not \leq B(\alpha_2) \text{ s.t. }B\cap B(\alpha_2) < B(\alpha_1)} \Phi(B) \leq  \Phi(B(\alpha_1)) \right)\\
& \leq &  \PP \bigg( \min_{ B \not \leq B(\alpha_2) \text{ s.t. }B\cap B(\alpha_2) < B(\alpha_1)} \Phi(B)-\Phi(B\cap B(\alpha_2))\\ & & +\Phi(B\cap B(\alpha_2)) - \Phi(B(\alpha_1)\leq 0 \bigg)\\
& \leq & \PP \left( \min_{ B \not \leq B(\alpha_2) \text{ s.t. }B\cap B(\alpha_2) < B(\alpha_1)} \Phi(B)-\Phi(B\cap B(\alpha_2))\leq 0 \right)\\
&  &+ \PP \left( \min_{ B \not \leq B(\alpha_2) \text{ s.t. }B\cap B(\alpha_2) < B(\alpha_1)} \Phi(B\cap B(\alpha_2)) - \Phi(B(\alpha_1)\leq 0\right)\\
& \leq & \PP \left( \min_{ B \not \leq B(\alpha_2) } \Phi(B)-\Phi(B\cap B(\alpha_2))\leq 0 \right)\\
&  &+ \PP \left( \min_{ B < B(\alpha_1)} \Phi(B) - \Phi(B(\alpha_1)\leq 0\right)\\
& \longrightarrow &0,
\end{eqnarray*}
from Lemmas \ref{lm_B_notleq_Ba} and \ref{lm_B<_Ba}.
\end{proof}

\bigskip

\textbf{Proofs of Propositions \ref{prop_comp1}, \ref{prop_comp2} and \ref{prop_comp3}}

\begin{proof}
In the three cases, the computation of $\widehat{B}$ requires carrying out the BFS algorithm for $B_\lambda$ and the computation of a determinant for $\Psi(B_\lambda)$. Recall that if $G=(V,E)$ is a graph (where $V$ is the set of vertices and $E$ the set of edges), the complexity of the BFS algorithm is $O(|V|+|E|)$. Recall that, if $M$ is a squared matrix of size $p$, the complexity of $\det(M)$ is $O(p^3)$ using the LU decomposition.

Now, we compute the complexity of the three estimators $\widehat{B}_{\widehat{C}}$, $\widehat{B}_A$ and $\widehat{B}_s$.
\begin{itemize}
    \item  For all $\lambda \in A_{\widehat{C}}$, the complexity of $B_\lambda$ is $O(p^2)$, and the cardinal of $A_{\widehat{C}}$ is $O(p^2)$. Thus, the complexity of the computation of $\{ B_\lambda\;|\; \lambda \in A_{\widehat{C}}\}$ is $O(p^4)$. 
    
    Now, for all $\lambda\in A_{\widehat{C}}$, the complexity of $\Psi(B_\lambda)$ is $O(p^3)$ and the cardinal of $\{ B_\lambda\;|\; \lambda \in A_{\widehat{C}}\}$ is $O(p)$ (because the function $\lambda \mapsto B_\lambda$ decreases). Thus, the complexity of the evaluations $\{\Psi(B),\; B\in \{ B_\lambda\;|\; \lambda \in A_{\widehat{C}}\} \} $ is $O(p^4)$
    
    So the complexity of $\widehat{B}_{\widehat{C}}$ is $O(p^4)$.
    \item For the threshold $n^{-1\slash 3}$, the complexity of $B_{n^{-1\slash 3}}$ is $O(p^2)$.
    
    So the complexity of $\widehat{B}_\lambda$ is $O(p^2)$.
    
    \item One can divide the computation of $\widehat{B}_s$ into two steps.
    
    For the first step, as we do not know the value of $s$, we have to compute $B_{l\slash p}$ from $l=p$ to $l=s-1$, verifying each time if the maximal size of group is smaller than $m$ or not.
    First, for each value of $l$ from $p$ decreasing to $s$, the complexity of the BFS algorithm to $B_{l\slash p}$  is $O(p \times m^2)=O(p)$, thus, the complexity of all these partitions if $O(p^2)$. Then, for $l=s-1$, the complexity of $B_{(s-1)\slash p}$  is $O(p^2)$. So, the complexity of this first step is $O(p^2)$.
    
    In the second step, we have to evaluate $\Psi(B_{l\slash p})$ for all $l \in [s:p]$. The complexity of each evaluation 
    is $O(p m^3)=O(p)$, and the the number of evaluations is $O(p)$. Thus, the complexity of this second step is $O(p^2)$.
\end{itemize}
\end{proof}

\textbf{Proof of Proposition \ref{prop_cost} }\\

To prove the convergence of $\widehat{B}$ in the three cases, we need the three following Lemmas.

\begin{lm}\label{lm_seuil}
For all sequence $(\lambda_n)_n$ such that for all $n$, $\lambda_n \in [ n^{-1\slash 3}, a n^{-1\slash 4}\slash 3 \lambda_{\sup}(1+\sqrt{y})^2]$ (we assume that $n$ is large enough and that subset is not empty), we have
$$
\PP(B_\lambda = B^*)\longrightarrow 1.
$$
\end{lm} 

\begin{proof}

\underline{Step 1: $B_\lambda \leq B^*$ with probability which goes to 1.}

\begin{eqnarray*}
& & \PP\left( B_\lambda \not \leq B^*\right) \\
&=& \PP \left( \exists (i,j)\notin B^*,\; |\widehat{C}_{ij}| \geq \lambda \right)\\
&\leq & \PP \left( \exists (i,j)\notin B^*,\; |\widehat{\sigma}_{ij}| \geq \lambda \frac{\lambda_{\inf}(1-\sqrt{y})^2}{2} \right)+ \PP\left( \exists i \leq p,\; \widehat{\sigma}_{ii}< \lambda_{\inf}\frac{(1-\sqrt{y})^2}{2}\right) \\
& \leq & p^2 \max_{(i,j)\notin B^*} \PP \left( |\widehat{\sigma}_{ij}| \geq \lambda\frac{\lambda_{\inf}(1-\sqrt{y})^2}{2} \right)+ \PP\left(\lambda_{\min}(\widehat{\Sigma})< \lambda_{\inf}\frac{(1-\sqrt{y})^2}{2}\right) \\
& \leq & p^2 \max_{(i,j)\notin B^*} \PP \left( |\widehat{\sigma}_{ij}| \geq \frac{\lambda_{\inf}(1-\sqrt{y})^2}{2} n^{-1\slash 3} \right)+ o(1)\\
& \leq & 2 p^2 \exp\left( - C_{\inf}n^{1\slash 3} \right) + o(1) \longrightarrow 0,
\end{eqnarray*}
using Lemma \ref{lm_eingenvalue} and Bernstein's inequality.

\underline{Step 2: $B_\lambda \geq  B^*$ with probability which goes to 1.}\bigskip

For all $k \in [1:K]$, and all $\emptyset \varsubsetneq B_1 \varsubsetneq B_k^*$, let $B_2:=B_k^*\setminus B_1$ and $(i^*,j^*):=\argmax_{(i,j)\in B_1 \times B_2} |\sigma_{ij}|$, where the dependency on $k$ and $B_1$ is implicit. Thanks to Condition 4, we have $|\sigma_{i^*j^*}| \geq a n^{-1 \slash 4}$. Then, using Lemma \ref{lm_eingenvalue},
\begin{eqnarray*}
& & \PP\left( B_\lambda \not \geq B^*\right)\\
&=& \PP \left( \exists k \in [1:K],\; \exists \; \emptyset \varsubsetneq B_1 \varsubsetneq B_k^*,\; \max_{(i,j)\in B_1 \times B_2} |\widehat{C}_{ij}|< \lambda \right)\\
&\leq & \PP \left( \exists k \in [1:K],\; \exists \; \emptyset \varsubsetneq B_1 \varsubsetneq B_k^*,\; \max_{(i,j)\in B_1 \times B_2} |\widehat{\sigma}_{ij}|< 2\lambda \lambda_{\sup} (1+\sqrt{y})^2\right)\\
& & + \PP\left( \exists i \leq p,\; \widehat{\sigma}_{ii}\geq 2 \lambda_{\sup}(1+\sqrt{y})^2 \right) \\
&\leq & \PP \left( \exists k \in [1:K],\; \exists \; \emptyset \varsubsetneq B_1 \varsubsetneq B_k^*,\;  |\widehat{\sigma}_{i^*j^*}|< \frac{2}{3}a n^{-1\slash 4} \right)\\
& & + \PP\left( \lambda_{\max}(\widehat{\Sigma})\geq 2 \lambda_{\sup}(1+\sqrt{y})^2 \right) \\
&\leq & \PP \left(\exists k \in [1:K],\; \exists \; \emptyset \varsubsetneq B_1 \varsubsetneq B_k^*,\;  |\widehat{\sigma}_{i^*j^*}-\sigma_{i^*j^*}|> \frac{1}{3}a n^{-1\slash 4} \right)+o(1)\\
&\leq & \PP \left( \exists (i,j)\in [1:p]^2,\;  |\widehat{\sigma}_{ij}-\sigma_{ij}|> \frac{1}{3}a n^{-1\slash 4} \right)+o(1)\\
& \leq & p^2 \max_{(i,j)} \PP \left(  |\widehat{\sigma}_{ij}-\sigma_{ij}|> \frac{1}{3}a n^{-1\slash 4} \right)+o(1)\\
& \leq & 2 p^2 \exp\left( - C_{\inf}n^{1\slash 2} \right) + o(1) \longrightarrow 0,
\end{eqnarray*}
by Bernstein's inequality.
\end{proof}

\begin{lm}\label{lm_seuil_racinep}
Let $c>0$. Let $\tilde{A}:=\{a_0,\;a_1,...,\;a_L\}$ such that $a_0=0,\; a_L=1,\; 0<a_{l+1}-a_l< c \slash \sqrt{p}$ for all $l \in [0:L-1]$. Then,
$$
\PP \left( B^* \in \left\{ B_\lambda,\;\lambda \in \tilde{A} \right\} \right) \longrightarrow 1.
$$
\end{lm}

\begin{proof}
Thanks to Lemma \ref{lm_seuil}, it suffices to show that, for $n$ large enough, there exists $l \in [0:L]$ such that $a_l \in [ n^{-1\slash 3}, a n^{-1\slash 4}\slash 3 \lambda_{\sup}(1+\sqrt{y})^2]$. By contradiction, let us assume that there does not exist such $l$. Let $j \in [0:L]$ such that $a_j< n^{-1\slash 3}$ and $a_{j+1}> a n^{-1\slash 4}\slash 3 \lambda_{\sup}$. Thus, we have 
\begin{eqnarray*}
\sqrt{p}\left( a_{j+1}-a_j \right)& >&\sqrt{p}\left(\frac{a n^{-1\slash 4}}{3 \lambda_{\sup}(1+\sqrt{y})^2}-n^{-1\slash 3}\right) \\
&\geq & C_{\inf}n^{1\slash 4}\longrightarrow +\infty,
\end{eqnarray*}
which is in contradiction with the definition of $\tilde{A}$.

\end{proof}

\begin{lm}\label{lm_Asn}
We have,
$$
\PP \left( B^* \in \left\{ B_\lambda,\;\lambda \in A_{s} \right\} \right) \longrightarrow 1.
$$
\end{lm}

\begin{proof}
Let $\mathcal{P}_{p}(m)$ be the set of the partitions of $[1:p]$ such that all their elements have cardinal smaller than $m$. By assumption (Condition 3), $B^*\in \mathcal{P}_{p}(m)$. Let $G:=\{l \slash p|\; l \in [0,p]\}$.
Thus $G$ verifies the assumption of $\tilde{A}$ in Lemma \ref{lm_seuil_racinep}, so
$$
\PP \left( B^* \in \left\{ B_\lambda,\;\lambda \in  G \right\} \right) \longrightarrow 1.
$$
Thus 
$$
\PP \left( B^* \in \left\{ B_\lambda,\;\lambda \in  G\right\} \cap \mathcal{P}_{p}(m)  \right) \longrightarrow 1.
$$
To conclude, it suffices to prove that $\left\{ B_\lambda,\;\lambda \in  G_{s}\right\} \cap \mathcal{P}_{p}(m) = \left\{ B_\lambda,\;\lambda \in  A_{s}\right\} $.\\

We have immediately $ \left\{ B_\lambda,\;\lambda \in  A_{s}\right\}\subset  \left\{ B_\lambda,\;\lambda \in  G\right\} \cap \mathcal{P}_{p}(m)$. We have to prove the other inclusion. Assume that $B \in     \left\{ B_\lambda,\;\lambda \in  G\right\} \cap \mathcal{P}_{p}(m)$. We know that there exists $\lambda=l\slash p \in G$ such that $B=B_\lambda$. As $B_{l\slash p} \in \mathcal{P}_p(m)$, we know by definition of $s$ that $l\geq s$ and thus $\lambda \in A$.
\end{proof}

Now, we prove Proposition \ref{prop_cost}.
\begin{proof} 
\begin{itemize}
    \item  Using Lemma \ref{lm_seuil}, Proposition \ref{prop_conclu}, and the fact that $ \{ B_\lambda\;|\; \lambda \in A_{\widehat{C}} \}= \{ B_\lambda\;|\; \lambda \in [0,1[ \} $, we have  $\PP\left( \widehat{B}_{\widehat{C}}=B^*\right) \longrightarrow 1$.

    \item Using Lemma \ref{lm_seuil} and Proposition \ref{prop_conclu}, we have $\PP\left( \widehat{B}_\lambda=B^*\right) \longrightarrow 1$.
    
    \item Using Lemma \ref{lm_Asn} and Proposition \ref{prop_conclu}, we have $\PP\left( \widehat{B}_s=B^*\right) \longrightarrow 1$.
\end{itemize}
\end{proof}
\bigskip

\textbf{Proof of Proposition \ref{prop_seuil_ba} }\\

\begin{proof}
We follow the proof of Lemma \ref{lm_seuil}.

\underline{Step 1: $B_{n^{-\delta \slash 2}} \leq B(\alpha_2)$ with probability which goes to 1.}

\begin{eqnarray*}
& & \PP\left( B_{n^{-\delta\slash 2}} \not \leq B(\alpha_2) \right) \\
&=& \PP \left( \exists (i,j)\notin B(\alpha_2),\; |\widehat{C}_{ij}| \geq n^{-\delta\slash 2} \right)\\
&\leq & \PP \left( \exists (i,j)\notin B(\alpha_2),\; |\widehat{\sigma}_{ij}| \geq n^{-\delta} \frac{\lambda_{\inf}(1-\sqrt{y})^2}{2} \right)+ \PP\left( \exists i \leq p,\; \widehat{\sigma}_{ii}< \lambda_{\inf}\frac{(1-\sqrt{y})^2}{2}\right) \\
& \leq & p^2 \max_{(i,j)\notin B(\alpha_2)} \PP \left( |\widehat{\sigma}_{ij}| \geq n^{-\delta\slash 2} \frac{\lambda_{\inf}(1-\sqrt{y})^2}{2} \right)+ \PP\left(\lambda_{\min}(\widehat{\Sigma})< \lambda_{\inf}\frac{(1-\sqrt{y})^2}{2}\right) \\
& \leq & p^2 \max_{(i,j)\notin B(\alpha_2)} \PP \left( |\widehat{\sigma}_{ij}-\sigma_{ij}| \geq n^{-\delta\slash 2} \frac{\lambda_{\inf}(1-\sqrt{y})^2}{2}-n^{-\alpha_2} \right)+ o(1)\\
& \leq & 2 p^2 \exp\left( - C_{\inf}(\delta,\alpha_2)n^{1-\delta} \right) + o(1) \longrightarrow 0,
\end{eqnarray*}
using Lemma \ref{lm_eingenvalue} and Bernstein's inequality.\bigskip

\underline{Step 2: $B_{n^{-\delta \slash 2}} \geq  \mathcal{B}(\alpha_1)$ with probability which goes to 1.}\bigskip

For all $k \in [1:K(\alpha_1)]$, and all $\emptyset \varsubsetneq B_1 \varsubsetneq B_k(\alpha_1)$, let $B_2:=B_k(\alpha_1) \setminus B_1$ and $(i^*,j^*):=\argmax_{(i,j)\in B_1 \times B_2} |\sigma_{ij}|$, where the dependency on $k$ and $B_1$ is implicit. Then, using Lemma \ref{lm_eingenvalue},
\begin{eqnarray*}
& & \PP\left(B_{n^{-\delta \slash 2}} \not \geq  \mathcal{B}(\alpha_1) \right)\\
&=& \PP \left( \exists k\in [1:K(\alpha_1)],\; \exists \; \emptyset \varsubsetneq B_1 \varsubsetneq B_k(\alpha_1),\; \max_{(i,j)\in B_1 \times B_2} |\widehat{C}_{ij}|< n^{-\alpha_1} \right)\\
&\leq & \PP \left( \exists k\in [1:K(\alpha_1)],\; \exists \; \emptyset \varsubsetneq B_1 \varsubsetneq B_k(\alpha_1),\;  |\widehat{\sigma}_{i^*j^*}|< 2 \lambda_{\sup} (1+\sqrt{y})^2 n^{-\alpha_1}\right)\\
& & + \PP\left( \exists i \leq p,\; \widehat{\sigma}_{ii}\geq 2 \lambda_{\sup}(1+\sqrt{y})^2 \right) \\
&\leq & \PP \left( \exists k\in [1:K(\alpha_1)],\; \exists \; \emptyset \varsubsetneq B_1 \varsubsetneq B_k(\alpha_1),\;  |\widehat{\sigma}_{i^*j^*}|<2 \lambda_{\sup} (1+\sqrt{y})^2 n^{-\alpha_1} \right)\\
& & + \PP\left( \lambda_{\max}(\widehat{\Sigma})\geq 2 \lambda_{\sup}(1+\sqrt{y})^2 \right) \\
&\leq & \PP \left(\exists k \in [1:K],\; \exists \; \emptyset \varsubsetneq B_1 \varsubsetneq B_k^*,\;  |\widehat{\sigma}_{i^*j^*}-\sigma_{i^*j^*}|>  n^{-\alpha_1}-2 \lambda_{\sup} (1+\sqrt{y})^2 n^{-\delta \slash 2} \right)+o(1)\\
&\leq & \PP \left( \exists (i,j)\in [1:p]^2,\;  |\widehat{\sigma}_{ij}-\sigma_{ij}|> n^{-\alpha_1}-2 \lambda_{\sup} (1+\sqrt{y})^2 n^{-\delta \slash 2} \right)+o(1)\\
& \leq & p^2 \max_{(i,j)} \PP \left(  |\widehat{\sigma}_{ij}-\sigma_{ij}|> n^{-\alpha_1}-2 \lambda_{\sup} (1+\sqrt{y})^2 n^{-\delta \slash 2} \right)+o(1)\\
& \leq & 2 p^2 \exp\left( - C_{\inf}(\delta,\alpha_1)n^{1-2\alpha_1} \right) + o(1) \longrightarrow 0,
\end{eqnarray*}
by Bernstein's inequality.\\
\end{proof}\bigskip

\textbf{Proof of Proposition \ref{prop_rate_sigma} }\\

\begin{proof} First, we prove the results for $\widehat{\Sigma}_{B^*}$.
We have, using again the notation $M\sim \mathcal{N}(0,\Sigma)$,
\begin{eqnarray*}
\E\left( \frac{n}{p}\|\widehat{\Sigma}_{B^*}-\Sigma \|_F^2\right)& \leq & n\; m^2 \max_{(i,j)\in B^*} \E\left[ (\widehat{\sigma}_{ij} -\sigma_{ij})^2  \right]\\
&= &  n\; m^2 \max_{(i,j)\in B^*} \V\left( \widehat{\sigma}_{ij} \right)\\
&\leq  &  m^2 \frac{n}{n-1} \max_{(i,j)\in B^*} \V\left( M_i M_j \right)\\
& \leq  &2m^2 \max_{(i,j)\in B^*} \left( \sigma_{ii}\sigma_{jj}+2\sigma_{ij}^2 \right)\\
& \leq & 6m^2 \lambda_{\sup}^2.
\end{eqnarray*}
By Markov's inequality, that proves 
$$
\frac{1}{p}\| \widehat{\Sigma}_{B^*}-\Sigma \|_F^2=O_p(1\slash n).
$$
Now, we want to prove that
$$
\frac{1}{p}\| \widehat{\Sigma}_{B^*}-\Sigma \|_F^2\neq o_p(1\slash n).
$$
First, we have
\begin{eqnarray*}
\E\left( \frac{n}{p}\| \widehat{\Sigma}_{B^*}-\Sigma \|_F^2\right)& \geq & n \min_{i \in [1:p]} \E\left[ (\widehat{\sigma}_{ii} -\sigma_{ii})^2  \right]\\
& =& n \min_{i \in [1:p]}\V\left( \widehat{\sigma}_{ii} \right)\\
& \geq & \frac{n}{n-1}\V(M_{ii}^2)\\
&\geq & \min_{i \in [1:p]} 2 \sigma_{ii}^2\\
& \geq & 2 \lambda_{\inf}^2.
\end{eqnarray*}
Now, the variance is 
$$
\V\left(\frac{1}{p}\| \widehat{\Sigma}_{B^*}-\Sigma \|_F^2\right)=\sum_{k=1}^K \V \left( \frac{1}{p} \| \widehat{\Sigma}_{B_k^*}-\Sigma_{B_k^*} \|_F^2\right)\leq p \max_{k \in [1:K]}  \V \left( \frac{1}{p} \| \widehat{\Sigma}_{B_k^*}-\Sigma_{B_k^*}\|_F^2\right).
$$
Now,
\begin{eqnarray*}
\V \left( \frac{1}{p} \| \widehat{\Sigma}_{B_k^*}-\Sigma_{B_k^*}\|_F^2\right) &=&\frac{1}{p^2}\V\left(\sum_{i,j\in B_k^*}(\widehat{\sigma}_{ij}-\sigma_{ij})^2\right).
\end{eqnarray*}
Remark that if  $A_1,...,A_d$ are random variables, we have 
$$
\V\left(\sum_{i=1}^d A_i\right)=\sum_{i,j=1}^d \cov(A_i,A_j)\leq \sum_{i,j=1}^d\sqrt{\V(A_i)}\sqrt{\V(A_j)}=\left( \sum_{i=1}^d \sqrt{\V(A_i)}\right)^2\leq d \sum_{i=1}^d \V(A_i).
$$
Thus
\begin{eqnarray*}
\V \left( \frac{1}{p} \| \widehat{\Sigma}_{B_k^*}-\Sigma_{B_k^*}\|_F^2\right) &\leq & \frac{m^4}{p^2} \max_{i,j \in B_k^*} \V\left( (\widehat{\sigma}_{ij}-\sigma_{ij})^2 \right).
\end{eqnarray*}
Let $i,j \in B_k^*$ for some $k$. We want to upper-bound $\V\left( (\widehat{\sigma}_{ij}-\sigma_{ij})^2\right)$. Let us define $a_{k}:=X_i^{(k)}X_j^{(k)}-\sigma_{ij}$. We know that
$$
(\widehat{\sigma}_{ij}-\sigma_{ij})^2= \left( \frac{1}{n}\sum_{k=1}^n a_k \right)^2= \frac{1}{n^2}\sum_{k=1}^n a_k^2+ \frac{1}{n^2}\sum_{k\neq k'}a_k a_{k'}.
$$
So, using the independence of $a_1,...,a_n$, we obtain
\begin{eqnarray*}
\V\left( (\widehat{\sigma}_{ij}-\sigma_{ij})^2 \right) &=& \frac{1}{n^4}\sum _{k_1,k_2=1}^n \cov \left(a_{k_1}^2, a_{k_2}^2 \right) + 2\frac{1}{n^4} \sum_{\substack{k_1, k_2, k_2'=1,\\ k_2 \neq k_2'}}^n\cov \left( a_{k_1}^2, a_{k_2} a_{k_2'} \right)\\
&& + \frac{1}{n^4}\sum_{\substack{k_1,k_1', k_2, k_2'=1,\\ k_1\neq k_1',\; k_2 \neq k_2'}}^n \cov \left( a_{k_1}a_{k_1'}, a_{k_2} a_{k_2'} \right)\\
&=& \frac{1}{n^3} \cov \left(a_{1}^2, a_{1}^2 \right) + 4\frac{n-1}{n^3} \cov \left( a_{1}^2, a_{1} a_{2} \right)\\
&& + 2\frac{n-1}{n^3}\cov \left( a_{1}a_{2}, a_{1} a_{2} \right),
\end{eqnarray*}
where we observed that $\cov(a_1 a_2,a_1,a_3)=0$.
Now, by Isserlis' theorem and using the fact that $\sigma_{ij}$ is upper-bounded by $\lambda_{\sup}$, we have $\cov \left(a_{1}^2, a_{1}^2 \right)\leq C_{\sup}$, $\cov \left( a_{1}^2, a_{1} a_{2} \right)\leq C_{\sup}$ and $\cov \left( a_{1}a_{2}, a_{1} a_{2} \right)\leq C_{\sup}$ (and these bounds do not depend on $k,i,j$). So
$$
\V\left( \widehat{\sigma}_{ij}^2 \right) \leq \frac{C_{\sup}}{n^2}.
$$
Thus, 
$$
\V \left( \frac{1}{p} \| \widehat{\Sigma}_{B_k^*}-\Sigma_{B_k^*}\|_F^2\right) \leq \frac{C_{\sup}}{p^2n^2},
$$
and
$$
\V \left( \frac{1}{p} \| \widehat{\Sigma}_{B^*} -\Sigma\|_F^2\right) \leq \frac{C_{\sup}}{p\;n^2}.
$$
Thus, by Chebyshev's inequality
\begin{eqnarray*}
& & \PP\left( \frac{1}{p} \| \widehat{\Sigma}_{B^*}-\Sigma\|_F^2 < \frac{\lambda_{\inf}^2}{n} \right) \\
& \leq & \PP\left(\left| \frac{1}{p} \| \widehat{\Sigma}_{B^*}-\Sigma\|_F^2- \E \left[ \frac{1}{p} \| \widehat{\Sigma}_{B^*}-\Sigma\|_F^2 \right]\right| > \frac{\lambda_{\inf}^2}{n}\right)\\
& \leq &  \frac{\V \left(\frac{1}{p} \| \widehat{\Sigma}_{B^*}-\Sigma\|_F^2 \right)  n^2}{ \lambda_{\inf}^4}\\
& \leq & \frac{C_{\sup}}{p} \longrightarrow 0.
\end{eqnarray*}
So, we proved that $ \frac{1}{p} \| \widehat{\Sigma}_{B^*}-\Sigma\|_F^2$ is not an $o_p(1\slash n)$.\\

Now, we show that the same results hold for $S_{B^*}$ proving that $  \frac{1}{p}\| S_{B^*}-\Sigma \|_F^2- \frac{1}{p}\| \widehat{\Sigma}_{B^*}-\Sigma \|_F^2=o_p(1\slash n)$. We have
\begin{eqnarray*}
 & &\left| \frac{1}{p}\| S_{B^*}-\Sigma \|_F^2- \frac{1}{p}\| \widehat{\Sigma}_{B^*}-\Sigma \|_F^2\right|\\
 & = &\left| \frac{1}{p}\sum_{(i,j)\in B^*} 2 \sigma_{ij}\widehat{\sigma}_{ij} \frac{1}{n}- \widehat{\sigma}_{ij}^2 \frac{2n-1}{n^2} \right|\\
 & \leq & \frac{ m^2}{n} \max_{(i,j)\in B^*} \left|2 \sigma_{ij} \widehat{\sigma}_{ij}-\frac{2n-1}{n}\widehat{\sigma}_{ij}^2 \right|\\
 & \leq & \frac{m^2}{n} \max_{(i,j)\in B^*} \left(2 |\widehat{\sigma}_{ij} | |\widehat{\sigma}_{ij}-\sigma_{ij}| +  |\widehat{\sigma}^2_{ij} | \slash n \right).
\end{eqnarray*}
Yet, by Bernstein's inequality,
$$
\max_{(i,j)\in B^*} |\widehat{\sigma}_{ij} |=O_p(1),
$$
$$
\max_{(i,j)\in B^*} |\widehat{\sigma}_{ij}-\sigma_{ij}|=o_p(1),
$$
and
$$
\max_{(i,j)\in B^*} \widehat{\sigma}_{ij}^2 =O_p(1).
$$
That proves 
$$
 \frac{1}{p}\| S_{B^*}-\Sigma \|_F^2- \frac{1}{p}\| \widehat{\Sigma}_{B^*}-\Sigma \|_F^2=o_p(1\slash n).
$$
Now, on the one hand, we have
$$
\frac{1}{p}\|S_{B^*}-\Sigma \|_F^2=O_p(1\slash n),
$$
and by Proposition \ref{prop_cost},
$$
\frac{1}{p}\|S_{\widehat{B}}-\Sigma \|_F^2=O_p(1\slash n).
$$
On the other hand,
$$
\frac{1}{p}\|S_{B^*}-\Sigma \|_F^2\neq o_p(1\slash n).
$$
\end{proof}
\bigskip

\textbf{Proof of Proposition \ref{prop_rate_S}}

\begin{proof}
It suffices to prove that
\begin{equation}\label{eq_ineg_frob}
\frac{\lambda_{\inf}^2}{2}\leq \E\left( \frac{n}{p^2}\| S-\Sigma \|_F^2\right)\leq C_{\sup}.
\end{equation}
First,
\begin{eqnarray*}
\E\left( \frac{n}{p^2}\| S-\Sigma \|_F^2\right)& \leq & n\;  \max_{(i,j)\in [1:p]^2} \E\left[ (s_{ij} -\sigma_{ij})^2  \right]\\
&= &  n\; \max_{(i,j)\in [1:p]^2} \V\left( s_{ij} \right)+\frac{\sigma_{ij}^2}{n}\\
&= & \frac{n-1}{n}  \max_{(i,j)\in [1:p]^2} \V\left( M_i M_j \right)+\frac{\sigma_{ij}^2}{n}\\
&  \leq  & \max_{(i,j)\in [1:p]^2} \left( \sigma_{ii}\sigma_{jj}+\sigma_{ij}^2+ \frac{\sigma_{ij}^2}{n}\right)\\
& \leq & 3 \lambda_{\sup}^2.
\end{eqnarray*}
Secondly,
\begin{eqnarray*}
\E\left( \frac{n}{p^2}\| S-\Sigma \|_F^2\right)& \geq & n \min_{(i,j)\in [1:p]^2} \E\left[ (s_{ij} -\sigma_{ij})^2  \right]\\
& \geq & n\min_{(i,j)\in [1:p]^2}\V\left( s_{ij} \right)\\
&=& \frac{n-1}{n} \min_{(i,j)\in [1:p]^2} \left(\sigma_{ii}\sigma_{jj}+\sigma_{ij}^2 \right)\\
& \geq & \frac{1}{2} \lambda_{\inf}^2.
\end{eqnarray*}
\end{proof}
\bigskip

\textbf{Proof of Proposition \ref{prop_rate_sigma_ba} }\\

\begin{proof}
We follow the proof of Proposition \ref{prop_rate_sigma}.
Let $\delta \in ]1\slash 2,1[,\varepsilon>0$, and $\alpha_1:=\delta\slash 2-\varepsilon \slash 4$.

We have
\begin{eqnarray*}
& & \max_{ B(\alpha_1) \leq  B \leq B^*}\E \left( \frac{n^{\delta-\varepsilon}}{p}\|\widehat{\Sigma}_B -\Sigma\|_F^2 \right) \\
&=& \max_{ B(\alpha_1) \leq  B \leq B^*}  \frac{n^{\delta-\varepsilon}}{p}\sum_{k=1}^K \left[ \sum_{\substack{i,j \in B_k^*,\\ (i,j)\in B}} \E \left( (\widehat{\sigma}_{ij}-\sigma_{ij})^2 \right)+ \sum_{\substack{i,j \in B_k^*,\\ (i,j)\in B}} \sigma_{ij}^2 \right]\\
 & \leq & n^{\delta-\varepsilon}m^2 \left(\max_{(i,j)\in B^*} \E\left( \widehat{\sigma}_{ij}-\sigma_{ij})^2 \right)+ \max_{ B(\alpha_1) \leq  B \leq B^*}\max_{(i,j)\in B^*\setminus B}\sigma_{ij}^2 \right)\\
 & \leq & n^{\delta-\varepsilon}m^2 \left(O\left( \frac{1}{n}\right)+ n^{-2\alpha_1} \right) \longrightarrow 0.
\end{eqnarray*}
Thus,
$$
\max_{ B(\alpha_1) \leq  B \leq B^*}\frac{1}{p}\|\widehat{\Sigma}_B -\Sigma\|_F^2 =o_p\left( \frac{1}{n^{\delta-\varepsilon}}\right),
$$
and thus
$$
\max_{ B(\alpha_1) \leq  B \leq B^*}\frac{1}{p}\|S_B -\Sigma\|_F^2 =o_p\left( \frac{1}{n^{\delta-\varepsilon}}\right).
$$
We conclude using Proposition \ref{prop_seuil_ba} and using that $B(\alpha_2)\leq B^*$.
\end{proof}
\bigskip

\textbf{Proof of Proposition \ref{prop_generate_sigma} }\\

\begin{proof}
The eigenvalues of $\Sigma$ are lower-bounded by $\varepsilon$ and upper-bounded by $mL$, so $\Sigma$ verifies Condition 2. Condition 3 is verified by construction. It remains to prove the slightly modified Condition 4 given in Proposition \ref{prop_generate_sigma}. Let $a>0$.
\begin{eqnarray*}
& & \PP\left( \exists B< B^*,\; \| \Sigma_{B}-\Sigma\|_{\max} < a n^{-1 \slash 4}\right)\\
&=& \PP \left( \exists k ,\; \max_{i,j\in B_k^*,\; i\neq j}|\sigma_{ij}| < a n^{-1 \slash 4}\right)\\
& \leq & p \PP \left( \max_{i,j\in [1:10],\; i\neq j} |\sum_{l=1}^L U_i^{(l)}U_{j}^{(l)}| \leq a n^{-1 \slash 4} \right),
\end{eqnarray*}
using an union bound and the fact that all the blocks have a size larger that 10. Then, by independence of  $\left(\sum_{l=1}^L U_{2k-1}^{(l)}U_{2k}^{(l)}\right)_{k\leq 5}$, we have
$$
\PP\left( \exists B< B^*,\; \| \Sigma_{B}-\Sigma\|_{\max} < a n^{-1 \slash 4} \right)  \leq p\PP \left(  |\sum_{l=1}^L U_1^{(l)}U_{2}^{(l)}| \leq a n^{-1 \slash 4} \right)^5
$$
Let $U_i:=(U_i^{(l)})_{l\leq L} \in \R^L$ for $i=1,2$. Then, $U_1$ and $U_2$ are independent and uniformly distributed on $[-1,1]^L$. Thus
$$
\PP \left(  |\sum_{l=1}^L U_1^{(l)}U_{2}^{(l)}| \leq a n^{-1 \slash 4} \right) = \E \left[\PP \left(\left.  |\langle U_1,U_2 \rangle| \leq a n^{-1 \slash 4}\right| U_2 \right) \right]
$$
Let $u_2\in [-1,1]^L \setminus \{0\}$. The set $\{u_1\in [-1,1]^L|\;  |\langle u_1,u_2 \rangle| \leq a n^{-1 \slash 4}\}$ is a subset of $\{ \sum_{l=1}^L x_i e_i|\; -a n^{-1 \slash 4} \|u_2\|\leq x_1 \leq a n^{-1 \slash 4}\|u_2\|,\; |x_l| \leq \sqrt{L}\;  \forall l \}$ where $e_1=u_2\slash \|u_2\|$ and $(e_1,...,e_L)$ is an orthonormal basis of $\R^L$. The Lebesgue measure of this subset is $(2\sqrt{L})^{L-1}2 a n^{-1 \slash 4}\|u_2\|$. Furthermore, (conditionally to $U_2=u_2$) the probability density function of $U_1$ on this set is either $0$ or $2^{-L}$. So, for all $u_2 \in [-1,1]^L\setminus\{0\}$,
$$
\PP \left(\left.  |\langle U_1,U_2 \rangle| \leq a n^{-1 \slash 4}\right| U_2=u_2 \right)\leq (2\sqrt{L})^{L-1}2 a n^{-1 \slash 4}\|u_2\|  2^{-L} \leq \sqrt{L}^{L-1}a n^{-1 \slash 4}. 
$$
Thus
$$
\PP \left(  |\sum_{l=1}^L U_1^{(l)}U_{2}^{(l)}| \leq a n^{-1 \slash 4} \right)\leq  \sqrt{L}^{L-1}a n^{-1 \slash 4}. 
$$
Then
$$
\PP\left( \exists B < B^*,\; \| \Sigma_{B}-\Sigma\|_{\max} < a n^{-1 \slash 4}\right)   \leq p (\sqrt{L}^{L-1}a n^{-1 \slash 4})^5 \longrightarrow 0.
$$
Hence, it remains to prove that the conclusion of Proposition \ref{prop_conclu} holds. That will imply the same for Propositions \ref{prop_cost} and \ref{prop_rate_sigma}. Let $a>0$ and $E:=\{ \Gamma \in S_p^{++}(\R,B^*)|\;  \forall B < B^*,\; \| \Sigma_{B}-\Sigma\|_{\max} \geq a n^{-1 \slash 4} \}$, where the generation of $B^*$ is defined in Proposition \ref{prop_generate_sigma}. We have
\begin{eqnarray*}
\PP\left( \widehat{B}_{tot} \neq B^* \right) & \leq & \PP \left( \Sigma \notin E\right) + \PP \left( \widehat{B}_{tot} \neq B^*| \; \Sigma \in E\right)\\
& \leq & o(1)+ \int_{E}\PP \left( \widehat{B}_{tot} \neq B^*| \; \Sigma= \Gamma \right) d \PP_{\Sigma}(\Gamma).
\end{eqnarray*}
Yet, for all $\Sigma \in E$, $\PP \left( \widehat{B}_{tot} \neq B^*| \; \Sigma= \Gamma \right)\longrightarrow 0 $ thanks to Proposition \ref{prop_conclu} (even in Condition 4 is not verified, the proof is still valid since the covariance matrix is in $E$).  We conclude by dominated convergence theorem.
\end{proof}
\bigskip

\textbf{Notation for the proofs of Section \ref{sec_pfix}}

For all $i,j \in [1:p]$, let $e_i\in \R^p$ be such that all coefficients are zero except the $i$-th one which is equal to 1, and let $e_{ij}\in \mathcal{M}_p(\R)$ be such that all coefficients are zero except the $(i,j)$-th one which is equal to 1. Let $\gamma_{ij}$ be the $(i,j)$-th coefficient of $\Sigma^{-1}$. Finally, as we use matrices $M$ of size $p^2\times p^2$, and vectors $v$ of size $p^2$, we define $v_{ij}:=v_{(j-1)p+i}$ and $M_{ij,kl}:=M_{(j-1)p+i, (l-1)p+k}$.

\bigskip

\textbf{Proof of Proposition \ref{prop_conv_pfix} }\\
 We see that, for all $B \in \mathcal{P}_p$, $l_{S_B}=\log(| S_B|)\slash p + \frac{n-1}{n}$ converges almost surely to $\log(| \Sigma_B|)\slash p+1$. The following Lemma gives a central limit theorem for this convergence.

\begin{lm}\label{lm_conv_pfix}
For all $B\in \mathcal{P}_p$, we have
\begin{equation}\label{eq_loi_asymp}
\sqrt{n}(\log|S_B|-\log |\Sigma_B|)\overset{\mathcal{L}}{\underset{n\rightarrow +\infty}{\longrightarrow}}\mathcal{N}(0,2\Tr(\Sigma_{B}^{-1} \Sigma \Sigma_{B}^{-1} \Sigma)\slash p),
\end{equation}
with $2\Tr(\Sigma_{B}^{-1} \Sigma \Sigma_{B}^{-1} \Sigma)\slash p \leq 2p$. In particular
$$
\sqrt{n}(\log|S|-\log |\Sigma|)\overset{\mathcal{L}}{\underset{n\rightarrow +\infty}{\longrightarrow}}\mathcal{N}(0,2).
$$
\end{lm}

\begin{proof}

Let $Z^{(k)}=M^{(k)}M^{(k)T}$, where $M^{(k)}=(M_i^{(k)})_{i\leq p} \in R^p$. We know that $\E(Z)=\Sigma$ and $\cov(Z_{i,j},Z_{k,l})=\E(X_iX_jX_kX_l)-\sigma_{ij}\sigma_{kl} =\sigma_{ij}\sigma_{kl}+\sigma_{ik}\sigma_{jl}+\sigma_{il}\sigma_{jk}-\sigma_{ij}\sigma_{kl} =\sigma_{ik}\sigma_{jl}+\sigma_{il}\sigma_{jk}$. 
Let $\Gamma\in \mathcal{M}_{p^2,p^2}$, be such that $\Gamma_{ij,kl}:= \sigma_{ik}\sigma_{jl}+\sigma_{il}\sigma_{jk}=\cov(Z_{i,j},Z_{k,l})$.
Using the central limit Theorem,
$$
\sqrt{n-1}\left(\vecc(\widehat{\Sigma}_B)-\vecc(\Sigma_B)\right))\overset{\mathcal{L}}{\underset{n\rightarrow +\infty}{\longrightarrow}}\mathcal{N}(0,\Gamma_B),
$$
and by Slutsky Lemma,
\begin{equation}\label{TLCveccB}
\sqrt{n}\left(\vecc(S_B)-\vecc(\Sigma_B)\right)\overset{\mathcal{L}}{\underset{n\rightarrow +\infty}{\longrightarrow}}\mathcal{N}(0,\Gamma_B),
\end{equation}
where $(\Gamma_{B})_{ij,kl}=\Gamma_{ij,kl}$ if $(i,j)\in B$ and $(k,l)\in B$ and $(\Gamma_{B})_{ij,kl}=0$ otherwise.

Let us apply the Delta-method to \eqref{TLCveccB} with the function $\log \circ \det \circ \mat$, where $\mat=\R^{p^2}\rightarrow \mathcal{M}_{p}(\R)$ is the inverse function of $\vecc$. If we write $L$ the Jacobian matrix of $\log \circ \det \circ \mat$, we have:
$$
\sqrt{n}(\log|S_B|-\log |\Sigma_B|)\overset{\mathcal{L}}{\underset{n\rightarrow +\infty}{\longrightarrow}}\mathcal{N}(0,L(\vecc(\Sigma_B))\Gamma_B L(\vecc(\Sigma_B))^T).
$$
Let us compute the linear map $L(\vecc(\Sigma_B)):\R^{p^2}\rightarrow \R$, that we identify with its matrix. Let us recall that, for the dot product $\langle A , B \rangle:= Tr(A^T B)$, the gradient of $\log\circ \det$ on $A$ is $A^{-1}$. Thus, if $v\in \R^{p^2}$, we have 
\begin{eqnarray*}
L(\vecc(\Sigma_B))(v) &=& D (\log \circ \det)( \mat( \vecc (\Sigma_B)) \circ D \mat(\vecc(\Sigma_B))(v)\\
&=& \langle \nabla(\log \circ \det)( \Sigma_B), D \mat(\vecc(\Sigma))(v) \rangle,\\
&=& \langle \Sigma_B^{-1}, \Sigma \rangle\\
&=& \Tr( \Sigma_B^{-1} \mat(v))\\
&=&\vecc (\Sigma_B^{-1})^T v.
\end{eqnarray*}
So  $L(\vecc(\Sigma_B))= \vecc (\Sigma_B^{-1})^T$, then
$$
\sqrt{n}(\log|S_B|-\log |\Sigma_B|)\overset{\mathcal{L}}{\underset{n\rightarrow +\infty}{\longrightarrow}}\mathcal{N}(0,\vecc(\Sigma_B^{-1})^T\Gamma_B \vecc(\Sigma_B^{-1})).
$$
Now,
\begin{eqnarray*}
& &\vecc(\Sigma_B^{-1})^T\Gamma_B \vecc(\Sigma_B^{-1})\\
&=&\sum_{i,j,k,l}(\Sigma_B)_{i,j}^{-1}(\sigma_{ik}\sigma_{jl}+\sigma_{il}\sigma_{jk})(\Sigma_B)_{k,l}^{-1}\\
&=&\sum_{i,j,k,l}(\Sigma_B)_{i,j}^{-1}\sigma_{ik}\sigma_{jl}(\Sigma_B)_{k,l}^{-1}+\sum_{i,j,k,l}(\Sigma_B)_{i,j}^{-1}\sigma_{il}\sigma_{jk}(\Sigma_B)_{k,l}^{-1}\\
&=&2\Tr(\Sigma_{B}^{-1} \Sigma \Sigma_{B}^{-1} \Sigma)\\
&=& 2 \Tr\left[ \left(\Sigma_B^{-\frac{1}{2}} \Sigma \Sigma_B^{-\frac{1}{2}}\right) \left(\Sigma_B^{-\frac{1}{2}} \Sigma \Sigma_B^{-\frac{1}{2}}\right) \right]\\
& \leq & 2\Tr\left(\Sigma_B^{-\frac{1}{2}} \Sigma \Sigma_B^{-\frac{1}{2}}\right) ^2\\
& = & 2 \Tr(\Sigma_B^{-1}\Sigma)^2\\
& =& 2 p^2.
\end{eqnarray*}
Indeed, as $A:=\Sigma_B^{-\frac{1}{2}} \Sigma \Sigma_B^{-\frac{1}{2}}$ is symmetric positive definite, we have $\Tr(AA)\leq \Tr(A)^2$.
\end{proof}

\begin{lm}\label{lm_det}
For all $\Gamma \in S_p^{++}(\R)$ and for all $B\in \mathcal{P}_p$ such that $\Gamma\neq \Gamma_B$, we have $\det(\Gamma_B)>\det(\Gamma)$.
\end{lm}

\begin{proof}
First, let us prove it for $|B|=K=2$. We have $B=\{I,J\}$.
\begin{eqnarray*}
\det(\Gamma)=\det(\Gamma_{I,I})\det(\Gamma_{J,J}-\Gamma_{J,I}\Gamma_{I,I}^{-1}\Gamma_{I,J}).
\end{eqnarray*}
Now, $\det(\Gamma_B)=\det(\Gamma_{I,I})\det(\Gamma_{J,J})$.
Thus, it suffices to show that $\det(\Gamma_{J,J})>\det(\Gamma_{J,J}-\Gamma_{J,I}\Gamma_{I,I}^{-1}\Gamma_{I,J})$. We then write $A_1:=\Gamma_{J,J}-\Gamma_{J,I}\Gamma_{I,I}^{-1}\Gamma_{I,J}$ which is symmetric positive definite (Schur's complement), and $A_2=\Gamma_{J,I}\Gamma_{I,I}^{-1}\Gamma_{I,J}$ which is also symmetric positive definite. Then, we have
\begin{eqnarray*}
\det(A_1+A_2)=\det(A_1)\det(I_p+A_1^{-\frac{1}{2}}A_2 A_1^{-\frac{1}{2}})>\det(A_1),
\end{eqnarray*}
because $\det(I_p+A_1^{-\frac{1}{2}}A_2 A_1^{-\frac{1}{2}})=\prod_{i=1}^p (1+\phi_i(A_1^{-\frac{1}{2}}A_2 A_1^{-\frac{1}{2}}))$.

Now, we prove the lemma for any value of $|B|=K$.
Let $\Gamma \in S_p^{++}(\R)$ and $B\in \mathcal{P}_p$ such that $\Gamma\neq \Gamma_B$. Let $B^{(j)}:=\{ \bigcup_{i=1}^j B_i, \bigcup_{i=j+1}^K B_i\}$ for all $j \in [1:K-1]$. We now define $(\Gamma^{(j)})_{j\in [1,K]}$ with the recurrence relation $\Gamma_{(j+1)}=\Gamma^{(j)}_{B^{(j)}}$ and with $\Gamma^{(1)}=\Gamma$, we then have $\Gamma_B=\Gamma^{K}$.
Thus
$$
\det(\Gamma_B)=\det(\Gamma^{(K)})\geq \det(\Gamma^{(K-1)}) \geq ... \geq \det(\Gamma^{(1)})=\det (\Gamma).
$$
Furthermore, as $\Gamma\neq \Gamma_B$, there exists $j$ such that $\Gamma_{B^{(j)}}^{(j)}\neq \Gamma^{(j)}$. Thus, at least one of the previous inequality is strict, and so $\det(\Gamma_B)>\det(\Gamma)$.
\end{proof}

Using Lemmas \ref{lm_conv_pfix} and \ref{lm_det}, we can prove Proposition \ref{prop_conv_pfix}.

\begin{proof}
It suffices to show that, for all $B \neq B^*$, 
$$
\PP(\widehat{B}_{tot}=B)\underset{n\rightarrow +\infty}{\longrightarrow} 0.
$$
We split the proof into two steps: for $B\not \geq B^*$ and for $B> B^*$.\\

\underline{Step 1: $B\not \geq B^*$.}\\
Let $h:=\min\{ \log(|\Sigma_B|)-\log(|\Sigma|)|\; B \not \geq  B^*\}= \min \{\log( |\Sigma_B|)-\log(|\Sigma|),\; B < B^*\}$, since $\Sigma_B=\Sigma_{B\cap B^*}$. Thanks to Lemma \ref{lm_det}, we know that $h>0$.

Let $B\not \geq B^*$. Using the convergence in probability of $l_{S'_B}$, we know that $\PP(l_{S_B}<\log |\Sigma_B| \slash p +1-h \slash3) \underset{n\rightarrow +\infty}{\longrightarrow} 0$ and  $\PP(l_{S_{B^*}}>\log |\Sigma| \slash p +1+h \slash3) \underset{n\rightarrow +\infty}{\longrightarrow} 0$.

Now, we know that for $n> (3p \slash h )^{1\slash \delta} $, the term of penalisation satisfies $\kappa \pen(B^*) < h\slash 3$.
Thus, 
$$
\PP(\widehat{B}_{tot}=B)\underset{n\rightarrow +\infty}{\longrightarrow} 0.
$$

\underline{Step 2: $B>B^*$.}\\
Let $B>B^*$.
We know that
\begin{eqnarray*}
& & \sqrt{n}\left(\Psi(B)-\Psi(B^*)\right)\\
&=&\sqrt{n}\left( l_{S_B}+\kappa \pen(B)- l_{S_{B^*}}-\kappa \pen(B^*) \right) \\
&=&\sqrt{n}\kappa( \pen(B) - \pen (B^*))+\sqrt{n}(l_{S_B}-l_{\Sigma_B})-\sqrt{n}(l_{S_{B^*}}-l_{\Sigma_{B^*}}),
\end{eqnarray*}
since $\Sigma_B=\Sigma_{B^*}$ for $B>B^*$.
Let $a_n$ be equal to $\sqrt{n}\kappa( \pen(B) - \pen (B^*))$ (which converges to $+\infty$), $b_n$ to be equal to $\sqrt{n}(l_{S_B}-l_{\Sigma_B})$ (which converges to a zero mean normal distribution) and $c_n$  to be equal to $\sqrt{n}(l_{S_{B^*}}-l_{\Sigma_{B^*}})$ (which converges to a zero mean normal distribution). We have
\begin{eqnarray*}
\PP\left[ \sqrt{n}\left(\Psi(B)-\Psi(B^*)\right)>0\right]& =& \PP(b_n-c_n< -a_n)\\
& \leq & \PP(b_n\leq - a_n\slash 2 \text{ or } c_n\geq  a_n\slash 2)\\
& \leq & \PP(b_n \leq -a_n\slash 2)+\PP(c_n \geq a_n\slash 2) \underset{n\rightarrow +\infty}{\longrightarrow} 0.
\end{eqnarray*}
Thus, $\PP(\widehat{B}_{tot}=B)\underset{n\rightarrow +\infty}{\longrightarrow} 0$.
\end{proof}
\bigskip

\textbf{Proof of Proposition \ref{prop_CR_Sigma} }\\

\begin{proof}
We follow the notation of \cite{stoica_Cramer-rao_1998}.

An othonormal basis of $S_p(\R)$ is $\{\frac{1}{\sqrt{2}}(e_{ij}+e_{ji})|\; i< j\}\cup \{e_{ii}|\;i\leq p\}$ with the following total order on $\{(i,j)\in [1:p]^2|\; i\leq j\}$: we write $(i,j)\leq  (i',j')$ if $j< j'$ or if $j=j$ and $i\leq i'$. We define $U\in \mathcal{M}_{p^2,p(p+1)\slash 2}(\R)$ as the matrix which columns are the vectorizations of the components of this basis of $\vecc(S_p(\R))$. Thus $U_{ij,kl}=\frac{1}{\sqrt{2}}(\mathds{1}_{(i,j)=(k,l)}+\mathds{1}_{(i,j)=(l,k)})$, for all $k<l$ and $U_{ij,kk}=\mathds{1}_{(i,j)=(k,k)}$.

Thus, $U(U^TJU)^{-1} U^T$ is the Cram\'{e}r-Rao bound, where $J$ is the standard Fisher information matrix in the model $\{ \mathcal{N}(\mu,\Sigma),\; \Sigma \in \mathcal{M}_p(\R)\}$. As the sample is i.i.d, it suffices to prove if with $n=1$. In the rest of the proof, we compute the Cram\'{e}r-Rao bound, and we show that this bound is equal to $\E\left( (S-\Sigma)(S-\Sigma)^T\right)$. We split the proof into several Lemmas.

\begin{lm}\label{lm_UJU}
Recall that $\Sigma^{-1}=(\gamma_{ij})_{i,j\leq p}$.
Let $A=(A_{mn,m'n'})_{m\leq n,m'\leq n'}\in \mathcal{M}_{p (p+1)\slash 2}(\R)$ defined by
$$
A_{mn,m'n'}=\left\{ \begin{array}{ll}
\frac{1}{2}(\gamma_{mm'}\gamma_{nn'}+\gamma_{mn'}\gamma_{nm'}) & \text{if $m<n$ and $m'<n'$}\\
\frac{1}{\sqrt{2}}\gamma_{mm'}\gamma_{nn'} &\text{if either $m=n$ or $m'=n'$}\\
\frac{1}{2}\gamma_{mm'}^2 & \text{if $m=n$ and $m'=n'$,}
\end{array}\right.
$$
Then, $A=U^T J U$.
\end{lm}

\begin{proof}
Deriving twice the log-likelihood with respect to $\sigma_{ij}$ and $\sigma_{kl}$ (for $i,j,k,l\in [1:p]$) and taking the expectation, we get
\begin{eqnarray*}
J_{ij,kl}&=&\frac{1}{2}\Tr\left( \Sigma^{-1}e_ie_j^T \Sigma^{-1}e_ke_l^T\right)\\
&=&\frac{1}{2}\gamma_{li}\gamma_{jk}.
\end{eqnarray*}
Thus, for all $m<n,\;m'<n'$, we have
\begin{eqnarray*}
(U^T J U)_{mn,m'n'}&=&\sum_{i,j,k,l=1}^p U_{ij,mn}J_{ij,kl}U_{kl,m'n'}\\
&=&\sum_{i,j,k,l=1}^p \frac{1}{\sqrt{2}} (\mathds{1}_{(i,j)=(m,n)}+\mathds{1}_{(i,j)=(n,m)})J_{ij,kl}\frac{1}{\sqrt{2}}(\mathds{1}_{(k,l)=(m',n')}+\mathds{1}_{(k,l)=(n',m')})\\
&=&\frac{1}{2}(J_{mn,m'n'}+J_{mn,n'm'}+J_{nm,m'n'}+J_{nm,n'm'})\\
&=&\frac{1}{2}(\gamma_{mm'}\gamma_{nn'}+\gamma_{mn'}\gamma_{nm'}).
\end{eqnarray*}
Now, if $m'<n'$, we have
\begin{eqnarray*}
(U^T J U)_{mm,m'n'}&=&\sum_{i,j,k,l=1}^p U_{ij,mm}J_{ij,kl}U_{kl,m'n'}\\
&=&\sum_{i,j,k,l=1}^p\mathds{1}_{(i,j)=(m,m)}J_{ij,kl}\frac{1}{\sqrt{2}}(\mathds{1}_{(k,l)=(m',n')}+\mathds{1}_{(k,l)=(n',m')})\\
&=&\frac{1}{\sqrt{2}}(J_{mm,m'n'}+J_{mm,n'm'})\\
&=&\frac{1}{\sqrt{2}}\gamma_{mm'}\gamma_{mn'}.
\end{eqnarray*}
If $m<n$, we have
\begin{eqnarray*}
(U^T J U)_{mn,m'm'}&=&\sum_{i,j,k,l=1}^p U_{ij,mn}J_{ij,kl}U_{kl,m'm'}\\
&=&\sum_{i,j,k,l=1}^p \frac{1}{\sqrt{2}}  (\mathds{1}_{(i,j)=(m,n)} +\mathds{1}_{(i,j)=(n,m)})J_{ij,kl} \mathds{1}_{(k,l)=(m',m')}\\
&=&\frac{1}{\sqrt{2}}(J_{mn,m'm'}+J_{nm,m'm'})\\
&=&\frac{1}{\sqrt{2}}\gamma_{mm'}\gamma_{nm'}.
\end{eqnarray*}
Finally,
\begin{eqnarray*}
(U^T J U)_{mm,m'm'}&=&\sum_{i,j,k,l=1}^pU_{ij,mm}J_{ij,kl}U_{kl,m'm'}\\
&=&\sum_{i,j,k,l=1}^p\mathds{1}_{(i,j)=(m,m)}J_{ij,kl}\mathds{1}_{(k,l)=(m',m')}\\
&=&J_{mm,m'm'}\\
&=&\frac{1}{2}\gamma_{mm'}^2.
\end{eqnarray*}
\end{proof}

\begin{lm}\label{lm_inv}
Let $B=(B_{mn,m'n'})_{m\leq n,m'\leq n'}\in \mathcal{M}_{p (p+1)\slash 2}(\R)$ defined by
$$
B_{mn,m'n'}=\left\{ \begin{array}{ll}
2(\sigma_{mm'}\sigma_{nn'}+\sigma_{mn'}\sigma_{nm'}) & \text{if $m<n$ and $m'<n'$}\\
2\sqrt{2}\sigma_{mm'}\sigma_{nn'} &\text{if either $m=n$ or $m'=n'$}\\
2\sigma_{mm'}^2 & \text{if $m=n$ and $m'=n'$,}
\end{array}\right.
$$
then, $B=A^{-1}$. Moreover $(U B U^T)_{ij,i'j'}=\sigma_{ii'}\sigma_{jj'}+\sigma_{ij'}\sigma_{ji'}$ for all $i,j,i',j'\in [1:p]$.
\end{lm}

\begin{proof}
We compute the product $A \;B$. First of all, let $m<n$ and $m'<n'$. We have
\begin{eqnarray*}
(A\;B)_{mn,m'n'}&=&\sum_{i\leq j} A_{mn,ij}B_{ij,m'n'}\\
&=&\sum_{i< j} A_{mn,ij}B_{ij,m'n'}+\sum_{i=j}A_{mn,ij}B_{ij,m'n'}\\
&=&\sum_{i< j}\left(\gamma_{mi}\gamma_{nj}+\gamma_{mj}\gamma_{ni}\right)\left(\sigma_{im'}\sigma_{jn'}+\sigma_{in'}\sigma_{jm'}\right)\\
&&+2\sum_{i=j}\gamma_{mi}\gamma_{nj}\sigma_{im'}\sigma_{jn'}\\
&=& I_1+I_2,
\end{eqnarray*}
with
$$
I_1=\sum_{i<j}\gamma_{mi}\gamma_{nj}\sigma_{im'}\sigma_{jn'}+\sum_{i<j} \gamma_{mj}\gamma_{ni}\sigma_{in'}\sigma_{jm'}+\sum_{i=j}\gamma_{mi}\gamma_{nj}\sigma_{im'}\sigma_{jn'},
$$
and
$$
I_2= \sum_{i<j}\gamma_{mj}\gamma_{ni}\sigma_{im'}\sigma_{jn'}+\sum_{i<j}\gamma_{mi}\gamma_{nj}\sigma_{in'}\sigma_{jm'}+\sum_{i=j}\gamma_{mj}\gamma_{ni}\sigma_{im'}\sigma_{jn'}.
$$
We then have
\begin{eqnarray*}
I_1&=&\sum_{i<j}\gamma_{mi}\gamma_{nj}\sigma_{im'}\sigma_{jn'}+\sum_{j<i} \gamma_{mi}\gamma_{nj}\sigma_{jn'}\sigma_{im'}+\sum_{i=j}\gamma_{im}\gamma_{jn}\sigma_{im'}\sigma_{jn'}\\
&=&\sum_{i,j}\gamma_{mi}\gamma_{nj}\sigma_{im'}\sigma_{jn'}\\
&=&\sum_{i}\gamma_{mi}\sigma_{im'}\sum_{j}\gamma_{nj}\sigma_{jn'}\\
&=&\mathds{1}_{(m,n)=(m',n')}.
\end{eqnarray*}
Similarly,
\begin{eqnarray*}
I_2&=& \sum_{i<j}\gamma_{mj}\gamma_{ni}\sigma_{im'}\sigma_{jn'}+\sum_{j<i}\gamma_{mj}\gamma_{ni}\sigma_{jn'}\sigma_{im'}+\sum_{i=j}\gamma_{mj}\gamma_{ni}\sigma_{im'}\sigma_{jn'}\\
&=&\sum_{i,j} \gamma_{mj}\gamma_{ni}\sigma_{im'}\sigma_{jn'}\\
&=&\sum_{i}\gamma_{ni}\sigma_{im'}\sum_{j}\gamma_{mj}\sigma_{jn'}\\
&=& \mathds{1}_{(n,m)=(m',n')}\\
&=&0.
\end{eqnarray*}
Now, if $m'<n'$
\begin{eqnarray*}
(A\;B)_{mm,m'n'}&=&\sum_{i\leq j} A_{mm,ij}B_{ij,m'n'}\\
&=&\sum_{i< j} A_{mm,ij}B_{ij,m'n'}+\sum_{i=j}A_{mm,ij}B_{ij,m'n'}\\
&=&\sqrt{2}\sum_{i< j}\gamma_{mi}\gamma_{mj}\left(\sigma_{im'}\sigma_{jn'}+\sigma_{in'}\sigma_{jm'}\right)\\
&&+\sqrt{2}\sum_{i=j}\gamma_{mi}\gamma_{mj}\sigma_{im'}\sigma_{jn'}\\
&=&\sqrt{2}\sum_{i,j} \gamma_{mi}\gamma_{mj}\sigma_{im'}\sigma_{jn'}\\
&=&\sqrt{2}\sum_{i}\gamma_{mi}\sigma_{im'} \sum_{j}\gamma_{mj}\sigma_{jn'}\\
&=&\sqrt{2}\mathds{1}_{(m,m)=(m',n')}\\
&=&0.
\end{eqnarray*}
If $m<n$, then
\begin{eqnarray*}
(A\;B)_{mn,m'm'}&=&\sum_{i\leq j} A_{mn,ij}B_{ij,m'm'}\\
&=&\sum_{i< j} A_{mn,ij}B_{ij,m'm'}+\sum_{i=j}A_{mn,ij}B_{ij,m'm'}\\
&=&\sqrt{2}\sum_{i< j}\left(\gamma_{mi}\gamma_{nj}+\gamma_{mj}\gamma_{ni}\right)\sigma_{im'}\sigma_{jm'}\\
&&+\sqrt{2}\sum_{i=j}\gamma_{mi}\gamma_{nj}\sigma_{im'}\sigma_{jm'}\\
&=&\sqrt{2}\sum_{i,j} \gamma_{mi}\gamma_{nj}\sigma_{im'}\sigma_{jm'}\\
&=&\sqrt{2}\sum_{i}\gamma_{mi}\sigma_{im'} \sum_{j}\gamma_{nj}\sigma_{jn'}\\
&=&\sqrt{2}\mathds{1}_{(m,n)=(m',m')}\\
&=&0.
\end{eqnarray*}
Finally,
\begin{eqnarray*}
(A\;B)_{mm,m'm'}&=&\sum_{i\leq j} A_{mm,ij}B_{ij,m'm'}\\
&=&\sum_{i< j} A_{mm,ij}B_{ij,m'm'}+\sum_{i=j}A_{mm,ij}B_{ij,m'm'}\\
&=&2\sum_{i< j}\gamma_{mi}\gamma_{mj}\sigma_{im'}\sigma_{jm'}\\
&&+\sum_{i=j}\gamma_{mi}\gamma_{mj}\sigma_{im'}\sigma_{jm'}\\
&=&\sum_{i,j} \gamma_{mi}\gamma_{mj}\sigma_{im'}\sigma_{jm'}\\
&=&\mathds{1}_{m=m'}.
\end{eqnarray*}
We proved that $B=A^{-1}$. Let us show that $(U^T B U)_{ij,i'j'}=\sigma_{ii'}\sigma_{jj'}+\sigma_{ij'}\sigma_{ji'}$ for all $i,j,i',j'$. First of all, assume $i\neq j$ and $i'\neq j'$. Assume for example $i<j$ and $i'<j'$. Then we have
\begin{eqnarray*}
&&(U B U^T)_{ij,i'j'}\\
&=&\sum_{m\leq n, m'\leq n'} U_{ij,mn} B_{mn,m'n'} U_{i'j',m'n'}\\
&=&\sum_{m\leq n, m'\leq n'} \frac{1}{\sqrt{2}}(\mathds{1}_{(m,n)=(i,j)}+\mathds{1}_{(m,n)=(j,i)})B_{mn,m'n'} \frac{1}{\sqrt{2}}(\mathds{1}_{(m',n')=(i',j')}+\mathds{1}_{(m',n')=(j',i')})\\
&=&\frac{1}{2}B_{ij,i'j'}\\
&=&\sigma_{ii'}\sigma_{jj'}+\sigma_{ij'}\sigma_{ji'}.
\end{eqnarray*}
We apply the same method for $i<j$ and $i'>j'$, for $i>j$ and $i'<j'$, and for $i>j$ and $i'>j'$. Then, let $i=j$ and $i'\neq j'$, for example $i'<j'$. We have
\begin{eqnarray*}
&&(U B U^T)_{ii,i'j'}\\
&=&\sum_{m\leq n, m'\leq n'} U_{ii,mn}B_{mn,m'n'}U_{i'j',m'n'}\\
&=&\sum_{m\leq n, m'\leq n'}\mathds{1}_{(m,n)=(i,i)}B_{mn,m'n'} \frac{1}{\sqrt{2}}(\mathds{1}_{(m',n')=(i',j')}+\mathds{1}_{(m',n')=(j',i')})\\
&=&\frac{1}{\sqrt{2}}B_{ii,i'j'}\\
&=&\sigma_{ii'}\sigma_{ij'}+\sigma_{ij'}\sigma_{ii'}.
\end{eqnarray*}
The other cases are similar.
\end{proof}

We thus have the component of the Cram\'{e}r-Rao bound:
$$
\left[U(U^T J U)^{-1}U^T\right]_{ij,i'j'}=\sigma_{ii'}\sigma_{jj'}+\sigma_{ij'}\sigma_{ji'}.
$$
This matrix is equal to $\E\left( (\vecc(S)-\vecc(\Sigma))(\vecc(S)-\vecc(\Sigma))^T\right)$ for $n=1$ and when the mean $\mu$ is known.
\end{proof}\bigskip

\textbf{Proof of Proposition \ref{prop_CR_SigmaB}}\\

\begin{proof}
An orthonormal basis of  $S_p(\R,B^*)$ is $\{\frac{1}{\sqrt{2}}(e_{ij}+e_{ji})|\; i< j,\; (i,j)\in B^*\}\cup\{e_{ii}|\;i\leq p \}$ with the following total order on $\{(i,j)\in [1:p]^2|\; i\leq j,\; (i,j)\in B^*\}$: we write $(i,j)\leq  (i',j')$ if $j< j'$ or if $j=j$ and $i\leq i'$. Thus, we define $U$ as the matrix which the columns are the vectorizations of the components of this basis of $S_p(\R,B^*)$. We have $U_{ij,kl}=\frac{1}{\sqrt{2}}(\mathds{1}_{(i,j)=(k,l)}+\mathds{1}_{(i,j)=(l,k)})$, for all $k<l$ with $(k,l)\in B^*$ and $U_{ij,kk}=\mathds{1}_{(i,j)=(kk)}$.

Thus, $U(U^TJU)^{-1} U^T$ is the Cram\'{e}r-Rao bound. As the sample is i.i.d, it suffices to prove the proposition with $n=1$.

\begin{lm}
Let $A=(A_{mn,m'n'})_{(m,n),\;(m',n')\in B^*,\;m\leq n,\;m'\leq n'}$ defined by
$$
A_{mn,m'n'}= \left\{ \begin{array}{ll}
\frac{1}{2}(\gamma_{mm'}\gamma_{nn'}+\gamma_{mn'}\gamma_{nm'}) & \text{if $m<n$ and $m'<n'$}\\
\frac{1}{\sqrt{2}}\gamma_{mm'}\gamma_{nn'} &\text{if either $m=n$ or $m'=n'$}\\
\frac{1}{2}\gamma_{mm'}^2 & \text{if $m=n$ and $m'=n'$,}
\end{array}\right.
$$
Then, $A=U^TJU$.
\end{lm}

\begin{proof}
The proof is similar to the proof of Lemma \ref{lm_UJU}, except that the values of $m,n,m'$ and $n'$ are more constraint.
First of all
\begin{eqnarray*}
J_{ij,kl}&=&\frac{1}{2}\Tr\left( \Sigma^{-1}e_{i}e_{j}^T \Sigma^{-1}e_{k}e_{l}^T\right)\\
&=&\frac{1}{2}\gamma_{il}\gamma_{jk}.
\end{eqnarray*}
Now, if $(m,n)\in B^*,\;(m',n')\in B^*,\;m<n,\;m'<n'$,
\begin{eqnarray*}
(U^T J U)_{mn,m'n'}&=&\sum_{i,j,k,l=1}^p U_{ij,mn}J_{ij,kl}U_{kl,m'n'}\\
&=&\sum_{i,j,k,l=1}^p \frac{1}{\sqrt{2}}(\mathds{1}_{(i,j)=(m,n)}+\mathds{1}_{(i,j)=(n,m)})J_{ij,kl}\frac{1}{\sqrt{2}}(\mathds{1}_{(k,l)=(m',n')}+\mathds{1}_{(k,l)=(n',m')})\\
&=&\frac{1}{2}\left(J_{mn,m'n'}+J_{mn,n'm'}+J_{nm,m'n'}+J_{nm,n'm'}\right)\\
&=&\frac{1}{2}(\gamma_{mm'}\gamma_{nn'}+\gamma_{mn'}\gamma_{nm'}).
\end{eqnarray*}
If $m'<n'$ and $(m',n')\in B^*$,
\begin{eqnarray*}
(U^T J U)_{mm,m'n'}&=&\sum_{i,j,k,l=1}^pU_{ij,mm}J_{ij,kl}U_{kl,m'n'}\\
&=&\sum_{i,j,k,l=1}^p\mathds{1}_{(i,j)=(m,m)}J_{ij,kl}\frac{1}{\sqrt{2}}(\mathds{1}_{(k,l)=(m',n')}+\mathds{1}_{(k,l)=(n',m')})\\
&=&\frac{1}{\sqrt{2}}(J_{mm,m'n'}+J_{mm,n'm'})\\
&=&\frac{1}{\sqrt{2}}\gamma_{mm'}\gamma_{mn'}.
\end{eqnarray*}
If $m<n$ and $(m,n)\in B^*$, we have
\begin{eqnarray*}
(U^T J U)_{mn,m'm'}&=&\sum_{i,j,k,l=1}^pU_{ij,mn}J_{ij,kl}U_{kl,m'm'}\\
&=&\sum_{i,j,k,l=1}^p\frac{1}{\sqrt{2}}(\mathds{1}_{(i,j)=(m,n)}+\mathds{1}_{(i,j)=(n,m)})J_{ij,kl}\mathds{1}_{(k,l)=(m',m')}\\
&=&\frac{1}{\sqrt{2}}(J_{mn,m'm'}+J_{nm,m'm'})\\
&=&\frac{1}{\sqrt{2}}\gamma_{mm'}\gamma_{nm'}.
\end{eqnarray*}
Finally,
\begin{eqnarray*}
(U^T J U)_{mm,m'm'}&=&\sum_{i,j,k,l=1}^p U_{ij,mm}J_{ij,kl}U_{kl,m'm'}\\
&=&\sum_{i,j,k,l=1}^p \mathds{1}_{(i,j)=(m,m)}J_{ij,kl}\mathds{1}_{(k,l)=(m',m')}\\
&=&J_{mm,m'm'}\\
&=&\frac{1}{2}\gamma_{mm'}^2.
\end{eqnarray*}
\end{proof}

\begin{lm}
Let $B=(B_{mn,m'n'})_{m\leq n,m'\leq n',\; (m,n)\in B^*, (m',n')\in B^*}$ defined by
$$
B_{mn,m'n'}=\left\{ \begin{array}{ll}
2(\sigma_{mm'}\sigma_{nn'}+\sigma_{mn'}\sigma_{nm'}) & \text{if $m<n$ and $m'<n'$}\\
2\sqrt{2}\sigma_{mm'}\sigma_{nn'} &\text{if either $m=n$ or $m'=n'$}\\
2\sigma_{mm'}^2 & \text{if $m=n$ and $m'=n'$,}
\end{array}\right.
$$
then, $B=A^{-1}$. Moreover $(U B U^T)_{ij,i'j'}=\sigma_{ii'}\sigma_{jj'}+\sigma_{ij'}\sigma_{ji'}$ for all $(i,j,i',j')\in B^*$ and $(U B U^T)_{ij,i'j'}=0$ for all $(i,j,i',j')\notin B^*$. Recall that we write $(i,j,i',j')\in B^*$ if there exists $A\in B^*$ such that $\{i,j,i',j'\} \subset A$.
\end{lm}

\begin{proof}
We introduce the following notation: if $l\in B_k^*$, let $[l]_k$ to be the index of $l$ in $B_k^*$.

\bigskip

\underline{Step 1:} Let us prove that $B=A^{-1}$.\\

We compute the product $AB$.
Assume that $m,n\in B_k^{*}$ with $m\leq n$ and $m',n'\in B_{k'}^*$ with $m'\leq n'$ and $k\neq k'$. We then have
\begin{eqnarray*}
\left( A\;B\right)_{mn,m'n'}&=&\sum_{(a,b)\in B^*,\;a\leq b} A_{mn,ab} B_{ab,m'n'}\\
&=&\sum_{a,b\in B_k^*,\;a\leq b} A_{mn,ab} B_{ab,m'n'}+\sum_{a,b\in B_{k'}^*,\;a\leq b} A_{mn,ab} B_{ab,m'n'}\\
&=&\sum_{a,b\in B_k^*,\;a\leq b} A_{mn,ab} 0+\sum_{a,b\in B_{k'}^*,\;a\leq b}0 B_{ab,m'n'}\\
&=& 0,
\end{eqnarray*}
using that $B_{ab,m'n'}=0$ if $a,b\in B_k^*$ and $m',n' \in B_{k'}^*$ because $\Sigma$ is block-diagonal, and using that $A_{mn,a,b}=0$ if $m,n\in B_k^*$ and $a,b \in B_{k'}^*$ because $\Sigma^{-1}$ is block-diagonal.
Assume that $m,n,m',n'\in B_{k}^*$ with $m\leq n$ and $m'\leq n'$. We have,
\begin{eqnarray*}
\left(  A\;B\right)_{mn,m'n'}&=&\sum_{(a,b)\in B^*,\;a\leq b} A_{mn,ab} B_{ab,m'n'}\\
&=& \left( A_{B_k^*} B_{B_k^*}\right)_{[m]_k[n]_k,[m']_k[n']_k}\\
&=&\mathds{1}_{([m]_k,[n]_k)=([m']_k,[n']_k)}\\
&=& \mathds{1}_{(m,n)=(m',n')},
\end{eqnarray*} 
thanks to Lemma \ref{lm_inv} applied to the matrix $\Sigma_{B_k^*}$.
We proved that $B=A^{-1}$.

\bigskip

\underline{Step 2.A :} We show that $(U B U^T)_{ij,i'j'}=\sigma_{ii'}\sigma_{jj'}+\sigma_{ij'}\sigma_{ji'}$ for all $(i,j,i',j')\in B^*$.\\

Assume that $(i,j,i',j')\in B^*$. First, assume that $i\neq j$ and $i'\neq j'$. Assume for example that $i<j$ and $i'<j'$ (the other cases are similar).
We then have
\begin{eqnarray*}
&&(U B U^T)_{ij,i'j'}\\
&=&\sum_{\substack{(m,n)\in B^*, m\leq n,\\ (m',n')\in B^*, m'\leq n'}} U_{ij,mn}B_{mn,m'n'}U_{i'j',m'n'}\\
&=&\sum_{\substack{(m,n)\in B^*, m\leq n,\\ (m',n')\in B^*, m'\leq n'}} \frac{1}{\sqrt{2}}(\mathds{1}_{(m,n)=(i,j)}+\mathds{1}_{(m,n)=(j,i)})B_{mn,m'n'} \frac{1}{\sqrt{2}}(\mathds{1}_{(m',n')=(i',j')}+\mathds{1}_{(m',n')=(j',i')})\\
&=&\frac{1}{2} B_{ij,i'j'}\\
&=&\sigma_{ii'}\sigma_{jj'}+\sigma_{ij'}\sigma_{ji'}.
\end{eqnarray*}

Let us take the case where $(i,j,i',j')\in B^*$ with either $i=j$, or $i'=j'$. For example $i=j$ and $i'<j'$. We then have
\begin{eqnarray*}
&&(U^T B U)_{ii,i'j'}\\
&=&\sum_{\substack{(m,n)\in B^*, m\leq n,\\ (m',n')\in B^*, m'\leq n'}} U_{mn,ii}B_{mn,m'n'}U_{m'n',i'j'}\\
&=&\sum_{\substack{(m,n)\in B^*, m\leq n,\\ (m',n')\in B^*, m'\leq n'}}\mathds{1}_{(m,n)=(i,i)}B_{mn,m'n'} \frac{1}{\sqrt{2}}(\mathds{1}_{(m',n')=(i',j')}+\mathds{1}_{(m',n')=(j',i')})\\
&=&\frac{1}{\sqrt{2}}B_{ii,i'j'}\\
&=&\sigma_{ii'}\sigma_{ij'}+\sigma_{ij'}\sigma_{ii'}.
\end{eqnarray*}
It is the same for $i=j$ and $i'>j'$, then for $i\neq j$ and $i'=j'$. We also can prove the equality similarly when $i=i'$ and $j=j'$.
\bigskip

\underline{Step 2.B:} Let us prove that $(U B U^T)_{ij,i'j'}=0$ for all $(i,j,i',j')\notin B^*$.\\

Assume that $(i,j,i',j')\notin B^*$. If $(i,j)\notin B^*$, or if $(i',j')\notin B^*$, we have
\begin{eqnarray*}
&&(U B U^T)_{ij,i'j'}\\
&=&\sum_{\substack{(m,n)\in B^*, m\leq n,\\ (m',n')\in B^*, m'\leq n'}} U_{ij,mn}B_{mn,m'n'}U_{i'j',m'n'}\\
&=&\sum_{\substack{(m,n)\in B^*, m\leq n,\\ (m',n')\in B^*, m'\leq n'}} \frac{1}{\sqrt{2}}(\mathds{1}_{(m,n)=(i,j)}+\mathds{1}_{(m,n)=(j,i)})B_{mn,m'n'} \frac{1}{\sqrt{2}}(\mathds{1}_{(m',n')=(i',j')}+\mathds{1}_{(m',n')=(j',i')})\\
&=&0,
\end{eqnarray*}
because if $(i,j)\notin B^*$, the term $(\mathds{1}_{(m,n)=(i,j)}+\mathds{1}_{(m,n)=(j,i)})$ is equal to 0. Similarly, if $(i',j')\notin B^*$, the term $(\mathds{1}_{(m',n')=(i',j')}+\mathds{1}_{(m',n')=(j',i')})$ is equal to 0.

It remains the case where $i,j\in B_k^*$ and $i',j'\in B_{k'}^*$ with $k\neq k'$. Then, $(U^T B U)_{ij,i'j'}=\sigma_{ii'}\sigma_{jj'}+\sigma_{ij'}\sigma_{ji'}=0$.
\end{proof}
To conlude the proof, we remark that, if $(i,j,i',j')\in B^*$, then
$$
\cov((S_{B^*})_{ij},(S_{B^*})_{i'j'})=\cov(X_iX_j,X_{i'}X_{j'})=\sigma_{ii'}\sigma_{jj'}+\sigma_{ij'}\sigma_{ji'}.
$$
Now, assume that $(i,j,i',j')\notin B^*$. If $(i,j)\notin B^*$ or if $(i',j')\notin B^*$, then $\cov((S_{B^*})_{ij},(S_{B^*})_{i'j'})=0$ because one of the two terms is zero. Assume that $i,j\in B_k^*$ and $i',j'\in B_{k'}^*$ with $k \neq k'$. Then 
$$
\cov((S_{B^*})_{ij},(S_{B^*})_{i'j'})=\cov(X_iX_j,X_{i'}X_{j'})=\sigma_{ii'}\sigma_{jj'}+\sigma_{ij'}\sigma_{ji'}=0.
$$
Thus, the covariance matrix of $\vecc(S_{B^*})$ is equal to the Cram\'{e}r-Rao bound.
\end{proof}\bigskip

\textbf{Proof of Proposition \ref{prop_CR_final}}\\

\begin{proof}
Using the central limit Theorem and Proposion \ref{prop_CR_SigmaB}, we have
$$
\sqrt{n-1}(\vecc(\widehat{\Sigma}_{B^*})-\vecc(\Sigma))\overset{\mathcal{L}}{\underset{n\rightarrow+\infty}{\longrightarrow}}\mathcal{N}(0,CR).
$$
Then, by Proposition \ref{prop_conv_pfix}, we have
$$
\sqrt{n-1}(\vecc(\widehat{\Sigma}_{\widehat{B}})-\vecc(\Sigma))\overset{\mathcal{L}}{\underset{n\rightarrow+\infty}{\longrightarrow}}\mathcal{N}(0,CR),
$$
and by Slutsky,
$$
\sqrt{n}(\vecc(S_{\widehat{B}})-\vecc(\Sigma))\overset{\mathcal{L}}{\underset{n\rightarrow+\infty}{\longrightarrow}}\mathcal{N}(0,CR).
$$
\end{proof}\bigskip

\textbf{Proof of Proposition \ref{prop_conv_shapley_high}}\\

\begin{lm}\label{lm_convshapley1}
Under Conditions 1 to 4, for all $\gamma>1\slash 2$
$$
\| S_{B^*}-\Sigma\|_2 =o_p\left(\frac{\log(n)^\gamma}{\sqrt{n}}\right),
$$
where $\| . \|_2$ is the operator norm, and it is equal to $\lambda_{max}(.)$ on the set of the symmetric positive semi-definite matrices.
\end{lm}

\begin{proof}
\begin{eqnarray*}
 & & \PP\left( \| S_{B^*}-\Sigma\|_2 > \frac{\varepsilon \log(n)^\gamma}{\sqrt{n}} \right)\\
 & = & \PP\left(  \exists k \in [1:K], \; \| S_{B_k^*} - \Sigma_{B_k^*}\|_2> \frac{\varepsilon \log(n)^\gamma}{\sqrt{n}}\right)\\
& \leq & K \max_{k \in [1:K]}\PP \left(   \| S_{B_k^*} - \Sigma_{B_k^*}\|_2 >  \frac{\varepsilon \log(n)^\gamma}{\sqrt{n}}\right)\\
& \leq &   K  \max_{k \in [1:K]} \PP \left( m \max_{i,j \in B_k^*}|s_{ij}-\sigma_{ij}| >  \frac{\varepsilon \log(n)^\gamma}{\sqrt{n}}\right)\\
& \leq &   K  \max_{k \in [1:K]} \PP \left(  \max_{i,j \in B_k^*}|s_{ij}- \widehat{\sigma}_{ij}|+| \widehat{\sigma}_{ij}-\sigma_{ij}| >  \frac{\varepsilon \log(n)^\gamma}{m\sqrt{n}}\right)\\
& \leq &   K m^2  \max_{k \in [1:K]} \max_{i,j\in B_k^*} \PP \left(  |s_{ij}- \widehat{\sigma}_{ij}|+| \widehat{\sigma}_{ij}-\sigma_{ij}| >  \frac{\varepsilon \log(n)^\gamma}{m\sqrt{n}}\right).
\end{eqnarray*}
Now, on the one hand,
\begin{eqnarray*}
&  &   K m^2  \max_{k \in [1:K]}\max_{i,j \in B_k^*}  \PP \left(  |\widehat{\sigma}_{ij}-\sigma_{ij}| >  \frac{\varepsilon \log(n)^\gamma}{2m\sqrt{n}}\right)\\
& \leq &   2K m^2   \exp \left( - C_{\inf} n \frac{\varepsilon^2 \log(n)^{2\gamma}}{4m^2 n}\right) \longrightarrow 0,
\end{eqnarray*}
by Bernstein's inequality. On the other hand,
\begin{eqnarray*}
& &   K m^2 \max_{k \in [1:K]} \max_{i,j \in B_k^*}  \PP \left(  |s_{ij}-\widehat{\sigma}_{ij}| >  \frac{\varepsilon \log(n)^\gamma}{2m\sqrt{n}}\right)\\
&=&   K m^2 \max_{k \in [1:K]} \max_{i,j \in B_k^*}  \PP \left(  |\widehat{\sigma}_{ij}| >  n\frac{\varepsilon \log(n)^\gamma}{2m\sqrt{n}}\right)\longrightarrow 0,
\end{eqnarray*}
by Bernstein's inequality.
\end{proof}

\begin{lm}\label{lm_widehatbeta}
Under Conditions 1 to 5, for all $\gamma> \frac{1}{2}$,
$$
\max_{i \in [1:p]}|\widehat{\beta}_i - \beta _i|= o_p\left(\frac{\log(n)^\gamma}{\sqrt{n}}\right).
$$
\end{lm}

\begin{proof}
We know that $\widehat{\beta}-\beta \sim \mathcal{N}\left(0,\sigma_n^2 [(A^TA)^{-1}]_{-1,-1}\right)$.
To simplify notation, let $Q:= \frac{1}{n}A^TA$. Remark that $Q_{1,1}=1,\; Q_{-1,1}= \frac{1}{n}\sum_{k=1}^n X^{(k)}$ and $Q_{-1,-1}= \frac{1}{n} \sum_{k=1}^n X^{(k)} X^{(k)T}$.
Now, we know that 
\begin{eqnarray*}
[Q^{-1}]_{-1,-1}&=& \left(Q_{-1,-1}-Q_{-1,1}Q_{1,1}^{-1}Q_{1,-1} \right)^{-1}\\
&=&\left( \frac{1}{n}\sum_{k=1}^n X^{(k)} X^{(k)T} - \left[  \frac{1}{n}\sum_{k=1}^n X^{(k)} \right] \left[ \frac{1}{n}\sum_{k=1}^n X^{(k)} \right]^T \right)^{-1}\\
&=& S^{-1}.
\end{eqnarray*}
Thus, $\widehat{\beta}-\beta \sim \mathcal{N}\left(0,\frac{\sigma^2}{n}S^{-1}\right)$.

Now, by Lemma \ref{lm_eingenvalue}, 
$$
\PP\left(\lambda_{\max}(S^{-1})\geq  \frac{2}{(1-\sqrt{y})^2 \lambda_{\inf}}\right) =o(1).
$$
Let $\varepsilon>0$ and $\gamma> \frac{1}{2}$. We have,
\begin{eqnarray*}
& & \PP \left( \max_{i\in [1:p]}| \widehat{\beta}_i - \beta_i| > \frac{\varepsilon \log(n)^\gamma}{n^{\frac{1}{2}}}\right)\\
& \leq & \PP \left( \max_{i\in [1:p]}| \widehat{\beta}_i - \beta_i| > \frac{\varepsilon \log(n)^\gamma}{n^{\frac{1}{2}}},\;\lambda_{\max}(S^{-1})<  \frac{2}{(1-\sqrt{y})^2 \lambda_{\inf}}  \right) + o(1)\\
& \leq & o(1)+ p \max_{i\in [1:p]} \PP \left( | \widehat{\beta}_i - \beta_i| > \frac{\varepsilon \log(n)^\gamma}{n^{\frac{1}{2}}},\;\lambda_{\max}(S^{-1})<  \frac{2}{(1-\sqrt{y})^2 \lambda_{\inf}}  \right)\\
& \leq & o(1)+ p \max_{i\in [1:p]} \PP \left( \frac{ n^{\frac{1}{2}}| \widehat{\beta}_i - \beta_i|}{\sigma_n \sqrt{(S^{-1})_{i,i}}} > \frac{\varepsilon \log(n)^\gamma}{\sigma_n \sqrt{(S^{-1})_{i,i}}},\;\lambda_{\max}(S^{-1})<  \frac{2}{(1-\sqrt{y})^2 \lambda_{\inf}}  \right)\\
& \leq & o(1)+ p \max_{i\in [1:p]} \PP \left( \frac{ n^{\frac{1}{2}}| \widehat{\beta}_i - \beta_i|}{\sigma_n\sqrt{(S^{-1})_{i,i}}} > \frac{(1-\sqrt{y})\sqrt{ \lambda_{\inf}} \varepsilon \log(n)^\gamma}{\sigma_n\sqrt{2}} \right)\\
& \leq & o(1)+ p \exp\left( - C_{\inf} \log(n)^{2 \gamma}\right) \underset{n \to +\infty}{\longrightarrow} 0.
\end{eqnarray*}
\end{proof}

\begin{lm}\label{lm_convshapley2}
Under Conditions 1 to 5, for all $\gamma>1\slash 2$,
$$
\left| \frac{p}{\widehat{\beta}^T S_{B^*}\widehat{\beta}}-  \frac{p}{\beta^T \Sigma \beta} \right| =o_p\left(\frac{\log(n)^\gamma}{\sqrt{n}}\right).
$$
\end{lm}

\begin{proof}
We have
\begin{eqnarray*}
& & \frac{1}{p}\left|\widehat{\beta}^T S_{B^*} \widehat{\beta} - \beta^T \Sigma \beta \right| \\
& \leq & \frac{1}{p}\left|\widehat{\beta}^T S_{B^*} \widehat{\beta} - \beta^T S_{B^*} \beta \right| + \frac{1}{p} \left| \beta^T (S_{B^*}- \Sigma)\beta \right| \\
& \leq &  \frac{\| \widehat{\beta}\|_2+ \| \beta \|_2}{\sqrt{p}}\| S_{B^*}\|_2 \frac{\| \widehat{\beta} - \beta\|_2}{\sqrt{p}}+\frac{\|\beta\|_2^2}{p}\| S_{B^*}-\Sigma\|_2\\
&=& o_p\left(\frac{\log(n)^\gamma}{\sqrt{n}}\right),
\end{eqnarray*}
by Lemmas \ref{lm_widehatbeta} and \ref{lm_convshapley1}. Now, with probability which goes to one $1$, by Lemmas \ref{lm_eingenvalue} and \ref{lm_widehatbeta}, we have $\widehat{\beta}^T S_{B^*} \widehat{\beta} \slash p \geq \lambda_{\inf}(1 - \sqrt{y})^2 \beta_{\inf}^2\slash 2$. Moreover, $\beta^T \Sigma \beta\slash p \geq \lambda_{\inf} \beta_{\inf}^2 \geq \lambda_{\inf}(1 - \sqrt{y})^2 \beta_{\inf}^2\slash2$. Thus, with probability which goes to one $1$, we have
$$
\left| \frac{p}{\beta^T S_{B^*} \beta}-  \frac{p}{\beta^T \Sigma \beta} \right| \leq \frac{4}{\lambda_{\inf}^2(1- \sqrt{y})^4\beta_{\inf}^4} \left| \frac{\beta^T S_{B^*} \beta}{p}-  \frac{\beta^T \Sigma \beta}{p} \right|=o_p\left(\frac{\log(n)^\gamma}{\sqrt{n}}\right).
$$
\end{proof}

We can now prove Proposition \ref{prop_conv_shapley_high}

\begin{proof}

Let $\tilde{\eta}_i$ be the estimator of $\eta_i$ obtained replacing $\Sigma$ by $S_{B^*}$ and $\beta$ by $\widehat{\beta}$ in Equations \eqref{eq_shap_groups} and \eqref{eq_VB}.
For all $\varepsilon>0$ and $\gamma>1\slash 2$, we have 
\begin{eqnarray*}
& & \PP\left( \sum_{i=1}^p \left|\widehat{\eta}_i-\eta_i\right|> \frac{\varepsilon \log(n)^\gamma}{\sqrt{n}}   \right)\\
& \leq &  \PP\left( \sum_{i=1}^p \left|\tilde{\eta}_i-\eta_i\right|> \frac{\varepsilon \log(n)^\gamma}{\sqrt{n}}   \right) + \PP(\widehat{B}\neq B^*)\\
& \leq & \PP\left(p \max_{i \in [1:p]} \left|\tilde{\eta}_i-\eta_i\right|> \frac{\varepsilon \log(n)^\gamma}{\sqrt{n}}   \right) + \PP(\widehat{B}\neq B^*).
\end{eqnarray*}
The term $\PP(\widehat{B}=B^*)$ goes to 0 from Proposition \ref{prop_conclu}. It remains to prove that 
\begin{equation}\label{eq_eta_op}
\PP\left(p \max_{i \in [1:p]} \left|\tilde{\eta}_i-\eta_i\right|> \frac{\varepsilon \log(n)^\gamma}{\sqrt{n}}  \right)  \longrightarrow 0.
\end{equation}
For all $k\in [1:K]$ and all $u\subset B_k^*$, let us write 
\begin{eqnarray*}
V_u^k&:=& \beta_{B_k^*-u}^T\left(\Sigma_{B_k^*-u,B_k^*-u}-\Sigma_{B_k^*-u,u}\Sigma_{u,u}^{-1}\Sigma_{u,B_k^*-u}\right)\beta_{B_k^*-u}\\
\tilde{V}_u^k&:=& \widehat{\beta}_{B_k^*-u}^T\left(S_{B_k^*-u,B_k^*-u}-S_{B_k^*-u,u}S^{-1}_{u,u}S_{u,B_k^*-u}\right)\widehat{\beta}_{B_k^*-u}\\
V&:=&\beta^T \Sigma \beta\\
\tilde{V}&:=&\widehat{\beta}^T S_{B^*} \widehat{\beta}\\
\alpha_u &:=&\frac{V_u^k}{V}\\
\tilde{\alpha}_u & :=& \frac{\tilde{V}_u^k}{\tilde{V}}.
\end{eqnarray*}
Let, for all $C \subset [1:p]$, $C\neq \emptyset$,
$$
L\left((a_u)_{u\in C};C\right):=\left(\frac{1}{|C|}\sum_{u\subset C-i}\begin{pmatrix}
|C|-1\\ |u|
\end{pmatrix}^{-1}
\left[a_{u+i}-a_{u}\right]\right)_{i\in C}.
$$
We then have
$(\eta_i)_{i\in B_k^*}=L\left((\alpha_u)_{u\in B_k^*};B_k^*\right)$ et $(\tilde{\eta}_i)_{i\in B_k^*}=L\left((\tilde{\alpha}_u)_{u\in B_k^*};B_k^*\right)$.

As $L(.;C)$ is linear, it is Lipschitz continuous from $(R^{2^{|C|}},\|.\|_\infty)$ to $(R^{|C|},\|.\|_\infty)$, with constant $l_{|C|}$ (we can show that $l_{|C|}=2$). Let $l:= \max_{j\in [1:m]} l_j<+\infty$ (we have in fact $l=2$). We then have,
$$
p\max_{i \in [1:p]} \left| \tilde{\eta}_i-\eta_i\right| \leq p\;l \max_{k\in [1:K]}\max_{u \subset B_k^*} \left| \tilde{\alpha}_u-\alpha_u\right|.
$$
It suffices to show that 
$$
p \max_{k\in [1:K]}\max_{u \subset B_k^*} \left| \tilde{\alpha}_u-\alpha_u\right|=o_p\left(\frac{\log(n)^\gamma}{\sqrt{n}}\right).
$$
Now,
\begin{eqnarray*}
 p \left| \tilde{\alpha}_u-\alpha_u\right|& \leq & \left| \frac{p \tilde{V}_u^k}{\tilde{V}}-\frac{p V_u^k}{\tilde{V}}\right|+\left| \frac{p V_u^k}{\tilde{V}}- \frac{p V_u^k}{V} \right| \\ 
 & \leq & \frac{p}{\tilde{V}} \left| \tilde{V}_u^k - V_u^k  \right|+V_u^k \left| \frac{p}{\tilde{V}}-\frac{p}{V}\right|.
\end{eqnarray*}

The term $\max_{k\in [1:K]}\max_{u \subset B_k^*}V_u^k$ is bounded from Conditions 2 and 5 and $\left| \frac{p}{\tilde{V}}-\frac{p}{V}\right|=o_p(\log(n)^\gamma\slash \sqrt{n})$ thanks to Lemma \ref{lm_convshapley2}. The term $\frac{p}{\tilde{V}}$ is bounded in probability using Lemma \ref{lm_convshapley2}, Conditions 2 and 5. Thus, it suffices to show that $\max_{k\in [1:K]}\max_{u \subset B_k^*} \left| \tilde{V}_u^k - V_u^k  \right|=o_p(\log(n)^\gamma\slash \sqrt{n})$. We will use that the operator norm of a sub-matrix is smaller than the operator norm of the whole matrix.

For all $k\in [1:K]$ and $u\subset B_k^*$, we have

\begin{eqnarray*}
\left| \tilde{V}_u^k - V_u^k  \right| &\leq& \Big| \widehat{\beta}_{B_k^*-u}^T\left(S_{B_k^*-u,B_k^*-u}-S_{B_k^*-u,u}S^{-1}_{u,u}S_{u,B_k^*-u}\right)\widehat{\beta}_{B_k^*-u}\\
 & & - \beta_{B_k^*-u}^T\left(S_{B_k^*-u,B_k^*-u}-S_{B_k^*-u,u}S^{-1}_{u,u}S_{u,B_k^*-u}\right)\beta_{B_k^*-u} \Big| \\
 & & + \Big|\beta_{B_k^*-u}^T\left(S_{B_k^*-u,B_k^*-u}-S_{B_k^*-u,u}S^{-1}_{u,u}S_{u,B_k^*-u}\right)\beta_{B_k^*-u} \\
 & & - \beta_{B_k^*-u}^T\left(\Sigma_{B_k^*-u,B_k^*-u}-\Sigma_{B_k^*-u,u}\Sigma_{u,u}^{-1}\Sigma_{u,B_k^*-u}\right)\beta_{B_k^*-u} \Big| \\
 & & \leq  \| \widehat{\beta}_{B_k^*-u}- \beta_{B_k^*-u}^T\|_2 \| S_{B_k^*-u,B_k^*-u}-S_{B_k^*-u,u}S^{-1}_{u,u}S_{u,B_k^*-u}\|_2 \left(  \| \widehat{\beta}_{B_k^*-u}\|_2 + \| \beta_{B_k^*-u}\|_2 \right) \ \\
 & & + m \beta_{\sup}^2 \|S_{B_k^*-u}-\Sigma_{B_k^*-u}\|_2  + m \beta_{\sup}^2\|S_{B_k^*-u,u}S_{u,u}^{-1} S_{u,B_k^*-u}-\Sigma_{B_k^*-u,u}\Sigma_{u,u}^{-1} \Sigma_{u,B_k^*-u}\|_2.
\end{eqnarray*}
Thus, we obtain a sum a three terms, and we have to prove that each term is $o_p(\log(n)^\gamma\slash \sqrt{n})$.
The first term is $o_p(\log(n)^\gamma\slash \sqrt{n})$ thanks to Lemmas \ref{lm_eingenvalue} and \ref{lm_widehatbeta}. 

For the second term, $\|S_{B_k^*-u}-\Sigma_{B_k^*-u}\|_2 \leq \|S_{B_k^*}-\Sigma_{B_k^*}\|_2$ so is $o_p(\log(n)^\gamma\slash \sqrt{n})$ from Lemma \ref{lm_convshapley1}.

Finally, for the third term,
\begin{eqnarray*}
& & \|S_{B_k^*-u,u}S_{u,u}^{-1} S_{u,B_k^*-u}-\Sigma_{B_k^*-u,u}\Sigma_{u,u}^{-1} \Sigma_{u,B_k^*-u}\|_2\\
& \leq  & \|S_{B_k^*-u,u}S_{u,u}^{-1}\left( S_{u,B_k^*-u}- \Sigma_{u,B_k^*-u}\right)\|_2\\
&& + \| S_{B_k^*-u,u}\left( S_{u,u}^{-1}- \Sigma_{u,u}^{-1}\right) \Sigma_{u,B_k^*-u}\|_2\\
& & + \| \left(S_{B_k^*-u,u}-\Sigma_{B_k^*-u,u}\right)\Sigma_{u,u}^{-1} \Sigma_{u,B_k^*-u}\|_2\\
& \leq & \| S_{B_k^*} \|_2 \| S^{-1}_{B_k^*}\|_2 \| S_{B_k^*}-\Sigma_{B_k^*}\|_2 +\| S_{B_k^*} \|_2  \| S^{-1}_{B_k^*}-\Sigma_{B_k^*}^{-1}\|_2 \| \Sigma_{B_k^*}\|_2+\|S_{B_k^*}-\Sigma_{B_k^*}\|_2 \| \Sigma_{B_k^*}^{-1}\|_2\|\Sigma_{B_k^*}\|_2 \\
& \leq & \| S \|_2 \| S^{-1}\|_2 \| S-\Sigma\|_2 +\| S \|_2  \| S^{-1}-\Sigma^{-1}\|_2 \| \Sigma\|_2+\|S-\Sigma\|_2 \| \Sigma^{-1}\|_2\|\Sigma\|_2 ,
\end{eqnarray*}
which do not depend on $k$ and $u$.
Finally, remark that $\|\Sigma\|_2$ and $\| \Sigma^{-1}\|_2$ are bounded from Condition 2, that $\| S \|_2$ are $\| S^{-1}\|_2$ bounded in probability from Lemma \ref{lm_eingenvalue}, that $\|S- \Sigma\|_2=o_p(\log(n)^\gamma\slash \sqrt{n})$ from Lemma \ref{lm_convshapley1} and Proposition \ref{prop_conclu} and that 
$$
\| S^{-1}-\Sigma^{-1}\|_2\leq \| \Sigma^{-1}\|_2 \| S^{-1}\|_2 \| S - \Sigma\|_2=o_p(\log(n)^\gamma\slash \sqrt{n}).
$$ 
Thus, we proved that
$$
 p \max_{k\in [1:K]}\max_{u \subset B_k^*} \left| \tilde{\alpha}_u-\alpha_u\right|=o_p(\log(n)^\gamma\slash \sqrt{n}).
$$
\end{proof}
\bigskip

\textbf{Proof of Proposition \ref{prop_conv_shapley_high_ba}}

\begin{lm}\label{lm_convshapley1_ba}
Under Conditions 1, 2 and 3, for all penalization coefficient $\delta \in ]0,1[$ and for all $\varepsilon>0$,
$$
\max_{B(\alpha_1)\leq B \leq B^*}\| S_{B}-\Sigma\|_2 =o_p\left(\frac{1}{n^{(\delta-\varepsilon)\slash 2}}\right),
$$
where $\| . \|_2$ is the operator norm, and it is equal to $\lambda_{max}(.)$ on the set of the symmetric positive semi-definite matrices.
\end{lm}

\begin{proof}
Let $\alpha_1:=\delta\slash 2-\varepsilon \slash 4$.
\begin{eqnarray*}
 & & \PP\left( \max_{B(\alpha_1)\leq B \leq B^*}\| S_{B}-\Sigma\|_2 > \frac{\epsilon}{n^{(\delta-\varepsilon)\slash 2}} \right)\\
 & = & \PP\left(  \exists k \in [1:K],  \; \max_{B(\alpha_1)\leq B \leq B^*}\| (S_B)_{B_k^*} - \Sigma_{B_k^*}\|_2> \frac{\epsilon}{n^{(\delta-\varepsilon)\slash 2}}\right)\\
& \leq & K \max_{k \in [1:K]}\PP \left(   \max_{B(\alpha_1)\leq B \leq B^*}\| (S_B)_{B_k^*} - \Sigma_{B_k^*}\|_2 >  \frac{\epsilon}{n^{(\delta-\varepsilon)\slash 2}}\right)\\
& \leq & K \max_{k \in [1:K]}\PP \left( m  \max_{B(\alpha_1)\leq B \leq B^*}\max_{i,j\in B_k^*}\left| (S_B)_{i,j} - \sigma_{ij}\right| >  \frac{\epsilon}{n^{(\delta-\varepsilon)\slash 2}} \right)\\
& \leq &   K  \max_{k \in [1:K]} \PP \left( m \max_{i,j \in B_k^*}|s_{ij}-\sigma_{ij}|+m n^{-\alpha_1} >  \frac{\epsilon}{n^{(\delta-\varepsilon)\slash 2}} \right)\\
& \leq &  K  \max_{k \in [1:K]} \PP \left( m \max_{i,j \in B_k^*}|s_{ij}-\sigma_{ij}|> C_{\inf}(\varepsilon, \delta, \epsilon) n^{-(\delta-\varepsilon)\slash 2} \right),\\
\end{eqnarray*}
that goes to 0 following the proof of \ref{lm_convshapley1}.
\end{proof}

\begin{lm}\label{lm_convshapley2_ba}
Under Conditions 1, 2, 3 and 5, for all penalization coefficient $\delta \in ]0,1[$ and for all $\varepsilon>0$,
$$
\left| \frac{p}{\widehat{\beta}^T S_{B^*} \widehat{\beta}}-  \frac{p}{\beta^T \Sigma \beta} \right| =o_p\left(\frac{1}{n^{(\delta-\varepsilon)\slash 2}}\right).
$$
\end{lm}

\begin{proof}
The proof is similar to the proof of Lemma \ref{lm_convshapley2}.
\end{proof}

We now can prove Proposition \ref{prop_conv_shapley_high_ba}.
\begin{proof}
For all $B\in \mathcal{P}_p$, we define $\tilde{\eta}(B)_i$ as the estimator of $\eta_i$ obtained replacing $B^*$ by $B$, $\Sigma$ by $S_{B}$ and $\beta$ by $\widehat{\beta}$ in Equations \eqref{eq_shap_groups} and \eqref{eq_VB}. We also define $\widehat{\eta}_i:=\tilde{\eta}(B_{n^{-\delta \slash 2}})_i $.
\begin{eqnarray*}
& & \PP\left( \sum_{i=1}^p \left|\widehat{\eta}_i-\eta_i\right|> \frac{\epsilon}{n^{-(\delta-\varepsilon)\slash 2}}  \right)\\
& \leq &  \PP\left(\max_{B(\alpha_1)\leq B \leq B^*} \sum_{i=1}^p \left|\tilde{\eta}(B)_i-\eta_i\right|> \frac{\epsilon}{n^{-(\delta-\varepsilon)\slash 2}}   \right) + \PP(\{B(\alpha_1)\leq B_{n^{-\delta \slash 2}} \leq B^*\}^c)\\
& \leq & \PP\left(p \max_{B(\alpha_1)\leq B \leq B^*} \max_{i \in [1:p]} \left|\tilde{\eta}(B)_i-\eta_i\right|> \frac{\epsilon}{n^{-(\delta-\varepsilon)\slash 2}}  \right) + \PP(\{B(\alpha_1)\leq B_{n^{-\delta \slash 2}} \leq B^*\}^c).
\end{eqnarray*}
By Proposition \ref{prop_seuil_ba}, $\PP(\{B(\alpha_1)\leq B_{n^{-\delta \slash 2}} \leq B^*\}^c) \longrightarrow 0$.

Finally, we prove that 
$$
\PP\left(p \max_{B(\alpha_1)\leq B \leq B^*} \max_{i \in [1:p]} \left|\tilde{\eta}(B)_i-\eta_i\right|> \frac{\epsilon}{n^{-(\delta-\varepsilon)\slash 2}}  \right)\longrightarrow 0,
$$
following the proof of Proposition \ref{prop_conv_shapley_high}.
\end{proof}

\textbf{Proof of Proposition \ref{prop_shap_generate}}\\

\begin{proof}
Remark that $\Sigma$ verifies Conditions 1 to 3. Let $a>0$.
Let $\check{\Sigma}:=\Sigma$ if $\forall B< B^*,\; \| \Sigma_{B}-\Sigma\|_{\max} \geq a n^{-1 \slash 4}$ and $\check{\Sigma}=I_p$ otherwise.
Let $\check{\eta}$ and $\check{\hat{\eta}}$ be defined as $\eta$ and $\widehat{\eta}$ in Proposition \ref{prop_conv_shapley_high} but replacing $\Sigma$ by $\check{\Sigma}$. As $\check{\Sigma}$ verify the Conditions 1 to 3 and the slightly modified Condition 4 given in Proposition \ref{prop_generate_sigma}, conditionally to $\check{\Sigma}$
$$
\sum_{i=1}^p \left|\check{\hat{\eta}}_i-\check{\eta}_i\right|=o_p\left(\frac{\log(n)^\gamma}{\sqrt{n}}\right).
$$
Thus, for all $\varepsilon>0$,
$$
\PP\left( \left.\sum_{i=1}^p \left|\check{\hat{\eta}}_i-\check{\eta}_i\right| > \frac{\varepsilon\log(n)^\gamma}{\sqrt{n}} \right| \check{\Sigma} \right) \longrightarrow 0,
$$
so, by dominated convergence theorem,
$$
\sum_{i=1}^p \left|\check{\hat{\eta}}_i-\check{\eta}_i\right|=o_p\left(\frac{\log(n)^\gamma}{\sqrt{n}}\right),
$$
unconditionally to $\check{\Sigma}$.

We conclude saying that $\check{\Sigma}=\Sigma$ with probability which converges to 1 from Proposition \ref{prop_generate_sigma}, so $\check{\hat{\eta}}=\widehat{\eta}$ and $\check{\eta}=\eta$ with probability which converges to 1.
\end{proof}

\end{document}